\documentclass[a4paper, 12pt]{amsart}
\usepackage{amsmath,amsfonts,amssymb,mathrsfs}
\usepackage{latexsym, amscd, epsfig}
\usepackage[utf8]{inputenc}
\usepackage{color}
\usepackage{amsmath}
\usepackage{amssymb}
\usepackage{amsthm}
\usepackage{graphicx}
\usepackage{esint}
\usepackage[colorlinks=true,linkcolor=blue]{hyperref}
\usepackage{verbatim}
\usepackage{accents}
\newcommand{\ubar}[1]{\underaccent{\bar}{#1}}

\theoremstyle{plain}
\newtheorem{theorem}{Theorem}
\newtheorem{proposition}{Proposition}[section]
\newtheorem{lemma}[proposition]{Lemma}
\newtheorem{corollary}[proposition]{Corollary}

\newtheorem{remark}[proposition]{Remark}

\theoremstyle{problem}
\newtheorem {problem}{Problem}

\numberwithin{equation}{section}
\def\Xint#1{\mathchoice
	{\XXint\displaystyle\textstyle{#1}}
	{\XXint\textstyle\scriptstyle{#1}}
	{\XXint\scriptstyle\scriptscriptstyle{#1}}
	{\XXint\scriptscriptstyle\scriptscriptstyle{#1}}
	\!\int}
\def\XXint#1#2#3{{\setbox0=\hbox{$#1{#2#3}{\int}$ }
		\vcenter{\hbox{$#2#3$ }}\kern-.6\wd0}}

\def\dashint{\Xint-}

\def\a{\alpha}
\def\b{\beta}
\def\c{\cdot}
\def\d{\delta}

\def\g{\gamma}
\def\G{\Gamma}
\def\k{\kappa}
\def\ld{\lambda}
\def\Ld{\Lambda}
\def\n{\nabla}
\def\o{\omega}
\def\O{\Omega}

\def\bp{\mathbf{p}}
\def\pt{\partial}
\def\q{\quad}
\def\r{\rho}
\def\th{\theta}
\def\ul{\ubar}
\def\ur{\Upsilon}
\def\v{\epsilon}
\def\vp{\varphi}
\def\B{\mathcal{B}}
\def\D{\mathcal{D}}

\def\MG{\mathcal{G}}
\def\h{\mathfrak{h}}
\def\J{\mathcal{J}}
\def\K{\mathcal{K}}
\def\N{\mathcal{N}}
\def\MO{\mathcal{O}}
\def\bp{\mathbf{p}}

\def\R{\mathbb{R}}
\def\S{\mathcal{S}}
\def\t{\mathfrak{t}}
\def\mn{\mathbf{n}}
\def\u{\mathbf{u}}

\newcommand{\bx}{{\nobreak\hfil\penalty50\hskip2em\hbox{}
\nobreak\hfil $\square$\qquad\parfillskip=0pt\finalhyphendemerits=0\par\medskip}}

\usepackage{xspace}

\begin{document}

\title[three-dimensional axially symmetric jet flows with nonzero vorticity]{\bf The jet problem for three-dimensional axially symmetric full compressible subsonic flows with nonzero vorticity}

\author{Yan Li}
\address{School of Mathematical Sciences, Shanghai Jiao Tong University, 800 Dongchuan Road, Shanghai, 200240, China}
\email{liyanly@sjtu.edu.cn}

\begin{abstract}
	In this paper, we show that for given Bernoulli function and entropy function at the upstream, if the incoming mass flux is within a suitable range, then there exists a unique outer pressure such that  smooth subsonic three-dimensional axially symmetric jet flows for steady full Euler system with nonzero vorticity exist and have certain far fields behavior. 
	A key observation is that we can transform the jet problem for three-dimensional axially symmetric steady full Euler system with nonzero vorticity into a variational problem in terms of the stream function. Moreover, using uniform estimates of the stream function and the iteration method, we exclude the singularity of the jet flows near the symmetry axis.
\end{abstract}

\keywords{axially symmetric flows, free boundary, steady Euler system, subsonic jet, vorticity.}
\subjclass[2010]{
	35Q31, 35R35, 35J20, 35J70, 35M32, 76N10}


\maketitle
\section{Introduction and main results}
\subsection{Background and motivation}The free boundary problems, such as jets and cavities, are  important topics in mathematics and hydrodynamics due to their theoretical and practical significance. The earlier mathematical theory of free boundary problems relied on complex analysis and hodograph transformations, which were applicable to two-dimensional flows 
but had no effective analogues in three-dimensional case (\cite{BZ1957_book,Gilbarg_jet_book}). A significant contribution to three-dimensional flows with free boundary was the existence of axially symmetric incompressible cavity flows given by Garabedian, Lewy and Schiffer (\cite{GLS1952}), where the idea was based on the principle of minimum virtual mass. Later, Alt, Caffarelli and Friedman used a variational method to deal with the Bernoulli type free boundary problems (\cite{AC81,ACF84}), and developed a general framework to investigate jets and cavities in both two-dimensional and three-dimensional cases (\cite{ACF82_asy_jet,ACF83_axially_jet,ACF85}).
In particular, the existence and uniqueness of three-dimensional axially symmetric irrotational subsonic jet flows were established in \cite{ACF85}. In addition, some interesting three-dimensional axially symmetric jet problems were studied in \cite{CDX2020_Rethy,CDZ2021,DW2021_two_axi}. More relevant results about free boundary problems see \cite{F82_book_variational,ACF_two_fluids_I,ACF_two_fluids_II,WX2013_degenerate_free,WX2019_free_boundary,CDXgravity} and references therein.

A recent advance of the free boundary problems is the well-posedness of two-dimensional compressible subsonic jet flows with general vorticity (\cite{LSTX2023}). This result is based on an observation that the stream function formulation for two-dimensional compressible steady Euler system enjoys a variational structure, even when the flows have nontrivial vorticity. Then the jet problem can be reformulated as a domain variation problem, which fits to the framework developed in the irrotational case (cf. \cite{ACF85}).
An important question is whether one can establish the well-posedness of three-dimensional compressible jet flows with general vorticity. As a first step, we study three-dimensional axially symmetric compressible subsonic jet flows with nonzero vorticity.

Before stating the main results of the paper, we briefly introduce some researches on the compressible flows in fixed infinitely long nozzles, which are closely related to the jet problem.
The existence and uniqueness of irrotational subsonic flows in infinitely long nozzles were established in \cite{XX2007,XX2010_axially_nozzle,DXY}, which give a positive answer to the problem posed by Bers in \cite{Bers1958_book}. 
By using the stream function formulation, the well-posedness of two-dimensional subsonic flows with nonzero vorticity in infinitely long nozzles was obtained in \cite{XX2010}.
This result was extended to three-dimensional axially symmetric cases  (\cite{DD2011_axially_small_vorticity,DL2013_axi_nonisentropic,DWX2018_swirl}). More generalizations about rotational subsonic flows in infinitely long nozzles cf.  \cite{CX2012_periodic_nozzle,CX2014_bounded_nozzle,CDX2012_full_Euler,CHWX2019_large_vorticity,DD2016_axially_large_vorticity,DL2015_axi_nonisentropic_large,DX2014_stagnation}.  

\subsection{The problem and main results}Three-dimensional steady ideal flows are governed by the following full Euler system
\begin{equation}\label{Euler origional}
\begin{cases}\begin{split}
& \n\cdot(\r \mathbf{u})=0,\\
& \n \cdot(\r \mathbf{u}\otimes\mathbf{u})+\n p=0,\\
& \n\cdot(\r\mathbf{u}E+p\mathbf{u})=0,
\end{split}\end{cases}
\end{equation}
where $\mathbf u=(u_1,u_2,u_3)$ denotes the flow velocity, $\r$ is the density, $p$ is the pressure, and $E$ is the total energy of the flow. For  the polytropic gas, one has
$$E=\frac{|\mathbf{u}|^2}{2}+\frac{p}{(\g-1)\r},$$
where the constant $\gamma>1$ is called the adiabatic exponent. The local sound speed and the Mach number of the flow are defined as
\begin{equation}\label{sonic mach}
c=\sqrt{\frac{\g p}{\r}} \quad \text{and} \quad M=\frac{|\mathbf{u}|}{c},
\end{equation}
respectively. The flow is called subsonic if $M<1$, sonic if $M=1$, and supersonic if $M>1$, respectively.

We specialize to axially symmetric compressible flows and take the symmetry axis to be the $x_1$-axis. Denote $x=x_1$ and $y=\sqrt{x_2^2+x_3^2}$. For simplicity we assume the swirl velocity of the flow is zero. Let $u$ and $v$ be the axial velocity and the radial velocity of the flow, respectively. Then
\begin{equation*}
	\begin{cases}\begin{split}
			&u_1(x_1,x_2,x_3)=u(x,y),\\
			&u_2(x_1,x_2,x_3)=v(x,y)\frac{x_2}{y},\\
			&u_3(x_1,x_2,x_3)=v(x,y)\frac{x_3}{y}.
	\end{split}\end{cases}
\end{equation*}
Hence in cylindrical coordinates, the full Euler system \eqref{Euler origional} can be rewritten as
\begin{equation}\label{Euler system}
	\begin{cases}\begin{split}
			&\n\c(y\r \u)=0,\\
			&\n\c(y\r \u\otimes \u)+y\n p=0,\\
			&\n\c(y(\r E+p)\u)=0,
	\end{split}\end{cases}
\end{equation}
where $\u=(u,v)$ and $\n=(\pt_x,\pt_y)$.
If the flow is away from vacuum, it follows from \eqref{Euler system} that
\begin{equation}\label{BS equations}
	\u\c\n B=0\q \text{and}\q
	\u\c\n S=0,
\end{equation}
where $B$ is the Bernoulli function defined by
\begin{equation}\label{B def}
	B=\frac{|\u|^2}{2}+\frac{\g p}{(\g-1)\r},
\end{equation}
and $S$ is the entropy function defined by
\begin{equation}\label{S def}
	S=\frac{\g p}{(\g-1)\r^\g}.
\end{equation}

Now given an axially symmetric nozzle in $\R^3$ with the symmetry axis and the upper solid boundary of the nozzle expressed as
\begin{equation}\label{nozzle}
	\N_0:=\{(x, 0): x\in \mathbb{R}\}
	\quad \text{and}\quad	
	\N:=\{(x,y): x=N(y),\, y\in[1,\bar H)\},
\end{equation}
respectively, where $\bar{H}>1$, $N\in C^{1,\bar\a}([1,\bar{H}])$ for some  $\bar\a\in(0,1)$ and it satisfies
\begin{equation}\label{nozzle condition}
	N(1)=0 \q \text{and} \q \lim_{y\to\bar{H}-}N(y)=-\infty.
\end{equation}

The main goal of this paper is to study the following jet problem.

\begin{problem}\label{probelm 1}
	Given the mass flux $Q$, the Bernoulli function $\bar{B}(y)$, and the entropy function $\bar{S}(y)$ of the flow at the upstream, find $(\r,\u,p)$, the free boundary $\G$, and the outer pressure $\ul{p}$ which we assume to be constant, such that the following statements hold.
	\begin{enumerate}
		\item[\rm (1)] The free boundary $\G$ joins the nozzle boundary $\N$ as a continuous surface and tends asymptotically horizontal at downstream as $x\to+\infty$.
		
		\item[\rm (2)] The solution $(\r,\u,p)$ solves the Euler system \eqref{Euler system} in the flow region $\mathcal{O}$ bounded by $\N_0$, $\N$, and $\G$. It takes the incoming data at the upstream, i.e.,
		\begin{equation}\label{upstream condition}
		\frac{|\u|^2}2+\frac{\gamma p}{(\gamma-1)\rho}\to\bar B(y) \q\text{and}\q \frac{\g p}{(\g-1)\r^\g}\to \bar{S}(y), \q \text{as } x\to-\infty,
		\end{equation}
		and
		\begin{equation}\label{Q}
			\int_0^1 y(\r u)(0,y)dy=Q.
		\end{equation}
		Furthermore, it satisfies the boundary conditions
		\begin{equation*}
			p=\ul{p} \q\text{on } \G, \q\text{and}\q \u\c\mn=0 \q\text{on } \N\cup \G,
		\end{equation*}
		where $\mn$ is the unit normal along $\N\cup\G$.
	\end{enumerate}
\end{problem}

The main results in this paper can be stated as follows.

\begin{theorem}\label{result}
Let the nozzle boundary $\N$ defined in \eqref{nozzle} satisfy \eqref{nozzle condition}. Given the Bernoulli function and the entropy function $(\bar{B},\bar{S})\in (C^{1,1}([0,\bar{H}]))^2$ and mass flux $Q>0$ at the upstream.
Suppose that
\begin{equation}\label{BS min}
	B_*:=\inf_{y\in[0,\bar{H}]}\bar{B}(y)>0,\q S_*:=\inf_{y\in[0,\bar{H}]}\bar{S}(y)>0,
\end{equation}
and
\begin{equation}\label{BS condition1}
	\bar{B}'(0)=0,\q \bar{B}'(\bar{H})\geq0,
	\q  \bar{S}'(0)=\bar{S}'(\bar{H})=0.
\end{equation}
There exists a constant $\k=\k(B_*,S_*,\gamma,\N)\in(0,1)$, such that if
\begin{equation}\label{BS condition2}
\|y^{-2}\bar{B}''(y)\|_{L^\infty((0,\bar{H}])}+\|y^{-2}\bar{S}''(y)\|_{L^\infty((0,\bar{H}])}\leq\k,
\end{equation}
then there exist $Q_*=\k^{\frac1{4\gamma}}$ and $Q_m=Q_m(B_*,S_*,\gamma,\N)>Q_*$, 
such that for any $Q\in(Q_*,Q_m)$, there are functions $\rho,\mathbf{u},p\in C^{1,\alpha}(\mathcal{O})\cap C^0(\overline{\mathcal{O}})$ (for any $\alpha\in(0,1)$) where $\mathcal O$ is the flow region, the free boundary $\G$, and the outer pressure $\ubar p$, such that $(\r,\u,p,\G,\ubar p)$ solves  Problem \ref{probelm 1}.  Furthermore, the following properties hold.
\begin{itemize}
	
\item[(i)] (Smooth fit) The free boundary $\G$ joins the nozzle boundary $\N$ as a $C^1$ surface.
	
\item[(ii)] The free boundary $\Gamma$ is given by a graph $x=\Upsilon(y)$, $y\in (\ubar H, 1]$, where $\Upsilon$ is a $C^{2,\alpha}$ function, $\ubar H\in(0,1)$, and $\lim_{y\rightarrow \ubar H+} \Upsilon(y)=\infty$. For $x$ sufficiently large, the free boundary  can also be written as $y=f(x)$ for some $C^{2,\alpha}$ function $f$ which satisfies
	$$\lim_{x\rightarrow \infty}f(x)=\ubar H
	\quad\text{and}\quad  \lim_{x\rightarrow \infty}f'(x)=0.$$
	
\item[(iii)] The flow is globally uniformly subsonic and has negative radial velocity in the flow region $\MO$, i.e.,
   \begin{equation*}\label{velocity}
		\sup_{\overline{\MO}}(|\u|^2-c^2)<0 \q\text{and}\q v<0 \ \text{in }\MO.
	\end{equation*}

\item[(iv)] (Upstream and downstream asymptotics)
The flow satisfies the asymptotic behavior
  \begin{equation}\label{asymptotic upstream}
  \|(\r,\u,p)(x,\c)-(\bar{\r}(\c),\bar{u}(\c),0,\bar{p})\|_{C^{1,\alpha}((0,\bar{H}))}\to0, \q\text{as } x\to-\infty
  \end{equation}
  and
  \begin{equation}\label{asymptotic downstream}
  \|(\r,\u,p)(x,\c)-(\ul{\r}(\c),\ul{u}(\c),0,\ul{p})\|_{C^{1,\alpha}((0,\ul{H}))}\to0, \q\text{as } x\to+\infty,
  \end{equation}
  where
  $\bar p$ is a positive constant, $(\bar\rho,\bar u)\in (C^{1,1}([0,\bar H]))^2$ and $(\ul{\r},\ul{u})\in (C^{1,\alpha}((0,\ul{H}]))^2$ are positive functions. Moreover, the upstream asymptotics $\bar\r$, $\bar u$, $\bar p$, and the downstream asymptotics $\ul{\r}$, $\ul u$, $\ul H$ are uniquely determined by $\bar B$, $\bar S$, $Q$, $\g$, $\bar H$, and $\ubar p$.

\item[(v)] (Uniqueness of the outer pressure) The outer pressure $\ubar p$ such that the solution $(\rho,\mathbf u,p,\Gamma)$ satisfies the properties (i)-(iv) is uniquely determined by $\bar B$, $\bar S$, $Q$, $\gamma$, and $\bar H$.
\end{itemize}
\end{theorem}

\begin{remark}\label{rmk_1}
	Suppose the flow satisfies the far field conditions with high order compatibility at the upstream, i.e.,
	\begin{align*}
		(\rho(x,y),u(x,y),v(x,y),p(x,y))\to(\bar\rho(y),\bar u(y),\bar v(y),\bar p(y))
	\end{align*}
	and
	\begin{equation*}
	\n(\rho(x,y),u(x,y),v(x,y),p(x,y))\to\n(\bar\rho(y),\bar u(y),\bar v(y),\bar p(y))
	\end{equation*}
	for some smooth functions $(\bar\rho,\bar u,\bar v,\bar p)$ as $x\to-\infty$. It follows from the Euler system \eqref{Euler system} that
	\begin{equation}\label{eq_0}
		(y\bar\rho\bar v)'=0,\q (y\bar\r \bar u\bar v)'=0,\q (y\bar \r \bar v^2)'+y\bar p'=0,\q  (y\bar\r \bar v\bar B)'=0.
	\end{equation}
	The first equation in \eqref{eq_0} shows that $y\bar\rho \bar v\equiv C$ for some constant $C$. Since $y\bar\r \bar v=0$ on the $x$-axis implies $C=0$ and since the solution is supposed to satisfy $\bar\rho>0$, one has $\bar v=0$. Then it follows from the third equation in \eqref{eq_0} that $y\bar p'=0$.
	Hence the pressure $\bar p$ is a constant.  Thus in view of \eqref{S def}, for non-isentropic Euler flows the density $\bar\rho$ is a function. Same arguments hold for the downstream flow states.
\end{remark}

\begin{remark} 	
For the convenience of stating the main results, we give the Bernoulli function, the entropy function, and the mass flux of the flow at the upstream. If we give suitable density, axial velocity, and pressure of the flow at the upstream, we can also prove the existence of the subsonic jet flows.
\end{remark}

\begin{remark} 	
In this paper, we study jet flows without swirl to avoid the technicality. It makes sense since the swirl velocity of the flow  will be always zero throughout the flow region if we assume the swirl is zero at the upstream and the streamlines have a simple topological structure (which can be guaranteed by $v<0$ in the flow region), cf. \cite{DD2012_uniqueness_axially}. Three-dimensional axially symmetric jet flows with swirl will be discussed in our future work.
\end{remark}

In order to prove Theorem \ref{result}, we follow the procedure in \cite{LSTX2023}, where two-dimensional subsonic jet flows for steady Euler system with general vorticity were  investigated. We first  reformulate three-dimensional axially symmetric full Euler system in terms of the Bernoulli function, the entropy function and the vorticity (cf. Proposition \ref{equivalent pro}). Then further reduce the full Euler system into a second order quasilinear equation of the stream function,  which is elliptic if and only if the flow is subsonic (cf. Lemma \ref{quasilinearequ lem}). A key observation is that as in the two-dimensional isentropic case, this quasilinear elliptic equation of the stream function enjoys a variational structure in the subsonic state (cf. Lemma \ref{EL}).
Thus Problem \ref{probelm 1} can be translated into a variational problem, and we can apply the techniques developed in \cite{ACF85,LSTX2023} to study the existence and regularity of the solution.

The main difficulties of three-dimensional axially symmetric jet problem with nonzero vorticity come from the appearance of $1/y$ in the stream function formulation, for which we introduce the condition \eqref{BS condition2}. One of the difficulties is that the solution may be singular near the symmetry axis. As we study the variational problems in a series of truncated domains at first, we give appropriate conditions on the boundaries of the truncated domains, so that we can obtain a uniform decay estimate of the solutions to the domain truncated problems near the symmetry axis (cf. Lemma \ref{psi bound lem}). This is sufficient to exclude the singularity of the solution near the symmetry axis (cf. Proposition \ref{subsonic pro}). Actually we  show that the stream function is $C^{2,\alpha}$ up to the symmetry axis by iteration(cf. Appendix \ref{appendix}).  Besides, the specific boundary conditions of the domain truncated problems are also used to show the positive lower bound of the free boundary (cf. Proposition \ref{y graph}). 
Another difficulty is that the presence of $1/y$ in the equation makes the shifting argument in $y$-direction not applicable, which is crucial for the proof of the asymptotic behavior of the jet flow in two-dimensional case, cf.  \cite[Proposition 7.4]{LSTX2023}. Here we use the energy estimates as in \cite{DD2011_axially_small_vorticity} to obtain the asymptotic behavior of the solution.

The rest of the paper is organized as follows. In Section \ref{sec stream formulation}, we reformulate the Euler system and the jet problem in terms of the stream function. In Section \ref{sec variational formulation},  we give the variational formulation for the jet problem after the subsonic and domain truncations. The existence and regularity of solutions for the truncated free boundary problems are established in Section \ref{sec existence and regularity}. Fine properties of solutions for the truncated problems such as the monotonicity property and some uniform estimates are obtained in Section \ref{sec fine properties}. The continuous fit and smooth fit of the free boundary are established in Section \ref{sec fit}. In Section \ref{sec remove truncation}, we remove the domain and subsonic truncations, and then get the far fields behavior of solutions. This completes the proof for the existence of solutions  to the jet problem. In Section \ref{sec uniqueness}, we show the uniqueness of the outer pressure for the subsonic jet problem. 
The regularity of the stream function up to the symmetry axis is discussed in the appendix. 

In the rest of the paper, $(\pt_1,\pt_2)$ represents $(\pt_x,\pt_y)$, and for convention the repeated indices mean the summation.

\section{Stream function formulation and subsonic truncation}\label{sec stream formulation}

In this section, we use the stream function to reduce the full Euler system \eqref{Euler system} into a {second order} quasilinear equation, which is elliptic in the subsonic region and becomes singular at the sonic state. To deal with the possible degeneracy near the sonic state, we introduce a subsonic truncation so that the modified equation is always uniformly elliptic.

\subsection{The equation for the stream function}
We start by determining the incoming flow density, axial velocity and pressure  $(\bar{\r},\bar{u},\bar{p})$ with the given Bernoulli function $\bar{B}$, entropy function $\bar{S}$, and mass flux $Q$ at the upstream. Suppose the flow satisfies the upstream asymptotic behavior \eqref{asymptotic upstream} with some functions $(\bar\rho,\bar u)\in (C^{1,1}([0,\bar H]))^2$ and a constant $\bar p>0$. In view of \eqref{B def} and \eqref{S def} one has
\begin{equation}\label{BS_relation1}
B=\frac{|\mathbf u|^2}{2}+\left(\frac{\g p}{\g-1}\right)^{\frac{\g-1}{\g}}S^{\frac1\g},
\end{equation}
so that
\begin{equation}\label{BSbar relation}
	\bar{B}(y)=\frac{\bar{u}^2(y)}{2}+\left(\frac{\g \bar p}{\g-1}\right)^{\frac{\g-1}{\g}}\bar{S}^{\frac1\g}(y).
\end{equation}
Denote $\bar{D}:=(\bar B\bar S^{-\frac1\g})$. Then the upstream speed $\bar u$ satisfies
\begin{equation}\label{ubar express}
	\bar u(y)=\sqrt{2\bar S^{\frac1\g}(y)\Big(\bar D(y)-\Big(\frac{\g \bar p}{\g-1}\Big)^{\frac{\g-1}{\g}}\Big)}.
\end{equation}
For any positive constant $d$, let
\begin{equation}\label{pcm def}
	p_c(d):=\frac{\g-1}{\g}\left(\frac{2d}{\g+1}\right)^{\frac{\g}{\g-1}} \q\text{and}\q p^*(d):=\frac{\g-1}{\g}d^{\frac{\g}{\g-1}}.
\end{equation}
By straightforward calculations, the flow is subsonic at the upstream if and only if $\bar p\in (p_c(D^*), p^*(D_*))$, where
\begin{equation}\label{Dbar bounds}
	D_*:=\inf_{y\in[0,\bar{H}]}\bar{D}(y) \q\text{and}\q
	 D^*:=\sup_{y\in[0,\bar{H}]}\bar{D}(y).
\end{equation}
Note that
\begin{equation}\label{Dbar estimate}
	0<B_*(S^*)^{-\frac1\g}\leq D_*\leq B_*S_*^{-\frac1\g}\leq D^*\leq B^*S_*^{-\frac1\g},
\end{equation}
where $B_*$ and $S_*$ are defined in \eqref{BS min}, and
\begin{align}\label{BS max}
	&B^*:=\sup_{y\in[0,\bar{H}]}\bar{B}(y)
	\quad\text{and}\quad 
	S^*:=\sup_{y\in[0,\bar{H}]}\bar{S}(y).
\end{align}
Moreover, the conditions \eqref{BS condition1} and \eqref{BS condition2} imply
\begin{equation}\label{BS_*^*}
	B^*\leq B_*+\k\bar H^4 \quad\text{and} \quad
	S^*\leq S_*+\k\bar H^4 .
\end{equation}
Thus if $\k$ is sufficiently small depending on $B_*$, $S_*$, $\bar H$, and $\g$, then $D^*\leq (\g+1)D_*/2$, consequently $p_c(D^*)\leq p^*(D_*)$.

The incoming data $(\bar{\r},\bar{u},\bar{p})$ can be determined as follows.

\begin{proposition}\label{incoming data}
 Let $(\bar{B}, \bar{S})\in (C^{1,1}([0,\bar{H}]))^2$ be the Bernoulli function and the entropy function at the upstream which satisfy \eqref{BS min}-\eqref{BS condition2}. Let $Q>0$ be the mass flux. Suppose that the flow satisfies the upstream asymptotic behavior \eqref{asymptotic upstream} for some positive functions $(\bar\rho,\bar u)\in (C^{1,1}([0,\bar H]))^2$ and some constant $\bar p>0$. There exists a constant $\bar{\k}=\bar{\k}(B_*,S_*,\bar H,\g)>0$ sufficiently small, such that if $\k$ defined in \eqref{BS condition2} satisfies $\k\in(0,\bar{\k})$ and $Q\in (Q_*, Q^*)$ where $Q_*=\k^{\frac1{4\g}}$ and $Q^*=Q^*(B_*,S_*,\bar H,\g)>Q_*$, then the upstream state $(\bar{\r},\bar{u},\bar{p})$ in \eqref{asymptotic upstream} is uniquely determined by $\bar{B},\ \bar{S}$, $Q$, and $\g$, and $\bar{p}\to p_c(D^*)$ as $Q\to Q^*$, where $D^*$ is defined in \eqref{Dbar bounds}. Furthermore,
 \begin{equation}\label{rhobar_estimate}
 C^{-1}\leq \bar\r\leq C,\q \bar\rho'(0)=\bar\rho'(\bar H)=0,
 \quad \text{and}\quad
 \|\bar{\r}'\|_{L^{\infty}([0,\bar{H}])}\leq C\k,
 \end{equation}
 and
	\begin{equation}\label{ubar estimates}
		C^{-1}\k^{\frac1{4\g}}\leq \bar{u}\leq C,\q \bar{u}'(0)=0,\q \bar{u}'(\bar{H})\geq0, \q\text{and}\q \|\bar{u}'\|_{L^\infty([0,\bar{H}])}\leq C\k^{1-\frac1{4\g}},
	\end{equation}
	for some $C=C(B_*,S_*,\bar H,\g)>0$.
\end{proposition}

\begin{proof}
	Since the flow satisfies the asymptotic behavior \eqref{asymptotic upstream}, it follows from \eqref{ubar express} that
	\begin{equation}\label{Q bar}
		Q(\bar p)=\int_0^{\bar{H}}y\bar{\r}(y;\bar p)\sqrt{2\bar{S}^{\frac1\g}(y)\Big(\bar D(y)-\Big(\frac{\g\bar p}{\g-1}\Big)^{\frac{\g-1}{\g}}\Big)}dy,
	\end{equation}
	where
	\begin{equation}\label{rhobar express}
		\bar{\r}(y;\bar p)=\left(\frac{\g \bar{p}}{(\g-1)\bar{S}(y)}\right)^{\frac1\g}.
	\end{equation}
	Since
	$$\frac{d}{d\bar p}Q(\bar p)
	<0 \q \text{for } \bar{p}\in(p_c(D^*),p^*(D_*))$$
	with $D^*$ and $D_*$ as in \eqref{Dbar bounds}, the mapping $\bar p\mapsto Q(\bar p)$ is strictly monotone decreasing in $(p_c(D^*),p^*(D_*))$. This means that as long as $Q\in (Q(p^*(D_*)),Q(p_c(D^*)))$, there exists a unique $\bar{p}\in(p_c(D^*),p^*(D_*))$ satisfying \eqref{Q bar}; accordingly, $(\bar\r,\bar u)$ is uniquely determined by \eqref{rhobar express} and \eqref{ubar express}. Furthermore, let $Q^*:=Q(p_c(D^*))$, then $\bar p\to p_c(D^*)$ if $Q\to Q^*$.
	
	It follows from the expression for $\bar\r$ in \eqref{rhobar express} and \eqref{BS_*^*} that
	$$\left(\frac{\g\bar p}{(\g-1)(S_*+\k\bar H^4)}\right)^{\frac1\g}\leq \bar\r\leq \left(\frac{\g\bar p}{(\g-1)S_*}\right)^{\frac1\g}
	\quad \text{and}\quad
	\bar \r'(y)=-\frac{\bar \rho(y)}{\g\bar S(y)}\bar S'(y).$$
	Since $\bar{p}\in(p_c(D^*),p^*(D_*))$ where $D_*\leq B_*S_*^{-\frac1\g}\leq D^*$ (cf. \eqref{Dbar estimate}), using \eqref{BS min}-\eqref{BS condition2} one gets \eqref{rhobar_estimate}, provided $0<\k\leq1$.
	
	To derive the lower bound for $\bar u$ in \eqref{ubar estimates}, we find a suitable $Q_*\in (Q(p^*(D_*)),Q^*)$ firstly.
	In view of \eqref{Q bar} and \eqref{pcm def}, one has
	\begin{align*}
		Q(p^*(D_*))&=\int_0^{\bar{H}}y\left(\frac{\g  p^*(D_*)}{(\g-1)\bar{S}(y)}\right)^{\frac1\g}\sqrt{2\bar{S}^{\frac1\g}(y)\Big(\bar D(y)-\Big(\frac{\g p^*(D_*)}{\g-1}\Big)^{\frac{\g-1}{\g}}\Big)}dy\\
		&=\int_0^{\bar{H}}yD_*^{\frac{1}{\g-1}}\bar{S}^{-\frac1\g}(y)\sqrt{2\bar{S}^{\frac1\g}(y)(\bar D(y)-D_*)}dy.
	\end{align*}
	Using \eqref{Dbar estimate} and \eqref{BS_*^*} gives
	\begin{equation}\label{SD estimate}
		0\leq (S^*)^{\frac1\g}(D^*-D_*)\leq(S_*+\k\bar H^4)^{\frac1\g}(B_*+\k\bar H^4)S_*^{-\frac1\g}-B_*\leq C_0\k^{\frac1\g}
	\end{equation}
	for some $C_0=C_0(B_*,S_*,\bar H,\g)>0$, provided $0<\k\leq1$. Thus
	\begin{align*}
		Q(p^*(D_*))\leq\bar{H}^2B_*^{\frac1{\g-1}}S_*^{-\frac1{\g-1}}\sqrt{2C_0\k^{\frac1\g}} \leq \k^{\frac1{4\g}},
	\end{align*}
	provided $0<\k\leq \min\left\{1,\big(2\bar H^4(B_*/S_*)^{\frac{2}{\g-1}}C_0\big)^{-2\g}\right\}$. Similarly, one also deduces that
	\begin{align*}
		Q^*&=\int_0^{\bar{H}}y\left(\frac{\g p_c(D^*)}{(\g-1)\bar{S}(y)}\right)^{\frac1\g}\sqrt{2\bar{S}^{\frac1\g}(y)\Big(\bar D(y)-\Big(\frac{\g p_c(D^*)}{\g-1}\Big)^{\frac{\g-1}{\g}}\Big)}dy\\
		&\geq\int_{\frac{\bar H}{2}}^{\bar{H}}y\left(\frac{2D^*}{\g+1}\right)^{\frac1{\g-1}}\bar{S}^{-\frac1\g}(y)\sqrt{2\bar S^{\frac1\g}(y)\Big(\frac{\g-1}{\g+1}\bar D(y)+\frac{2}{\g+1}\left(\bar D(y)-D^*\right)\Big)}dy\\
		&\geq \frac{\bar H^2}{4}\left(\frac2{\g+1}\right)^{\frac1{\g-1}}B_*^{\frac1{\g-1}}(S^*)^{-\frac1{\g-1}}\sqrt{\frac{2(\g-1)}{\g+1}B_*-\frac{4}{\g+1} (S^*)^{\frac1\g}(D^*-D_*)}\\
		&\geq \frac{\bar H^2}{4}\left(\frac2{\g+1}\right)^{\frac1{\g-1}}B_*^{\frac1{\g-1}}(S_*+\bar H)^{-\frac1{\g-1}}\sqrt{\frac{\g-1}{\g+1}B_*},
	\end{align*}
	provided $0<\k\leq\min\left\{1,\left((\g-1)B_*/(4C_0)\right)^\g\right\} $. Hence there exists a constant  $\tilde\k=\tilde\k(B_*,S_*,\bar H,\g)>0$ such that $Q^*>\k^{\frac1{4\g}}$ for any $\k\in[0,\tilde\k)$. Let $Q_*:=\k^{\frac1{4\g}}$.
	
	Now we prove the estimate \eqref{ubar estimates} for $Q\in(Q_*,Q^*)$. Using \eqref{Q bar}-\eqref{SD estimate} yields
	\begin{align*}
		Q&=\int_0^{\bar{H}}y\left(\frac{\g \bar{p}}{(\g-1)\bar{S}(y)}\right)^{\frac1\g}\sqrt{2\bar{S}^{\frac1\g}(y)\Big(\bar D(y)-D_*+D_*-\Big(\frac{\g \bar{p}}{\g-1}\Big)^{\frac{\g-1}{\g}}\Big)}dy\\
		&\leq \int_0^{\bar{H}}y\left(\frac{\g p^*(D_*)}{(\g-1)\bar{S}(y)}\right)^{\frac1\g}\sqrt{2C_0\k^{\frac1\g}+2(S^*)^{\frac1\g}\Big(D_*-\Big(\frac{\g \bar{p}}{\g-1}\Big)^{\frac{\g-1}{\g}}\Big)}dy\\
		&\leq \bar{H}^2B_*^{\frac1{\g-1}}S_*^{-\frac1{\g-1}}\sqrt{2C_0\k^{\frac1\g}+2(S^*)^{\frac1\g}\Big(D_*-\Big(\frac{\g \bar{p}}{\g-1}\Big)^{\frac{\g-1}{\g}}\Big)}.
	\end{align*}
	Thus
	$$C_0\k^{\frac1\g}+(S^*)^{\frac1\g}\Big(D_*-\Big(\frac{\g \bar{p}}{\g-1}\Big)^{\frac{\g-1}{\g}}\Big)
	\geq \frac{Q^2}{2\bar H^4}B_*^{-\frac2{\g-1}}S_*^{\frac2{\g-1}}
	\geq  \frac{\k^{\frac1{2\g}}}{2\bar H^4}B_*^{-\frac2{\g-1}}S_*^{\frac2{\g-1}}.$$
	Note that $S^{*}\leq S_*+\k\bar H^4$ and $D_*\leq B_*S_*^{-\frac1\g}$. Therefore there exists a constant $\bar{\k}=\bar\k(B_*,S_*,\bar H,\g)\in(0,\tilde\k)$ such that if $0<\k<\bar{\k}$, then
	$$S_*^{\frac1\g}\left(D_*-\left(\frac{\g \bar{p}}{\g-1}\right)^{\frac{\g-1}{\g}}\right)\geq C_1\k^{\frac1{2\g}}$$
	for some $C_1=C_1(B_*,S_*,\bar H,\g)>0$. This implies $\bar{u}$ given by \eqref{ubar express} satisfies
	\begin{equation}\label{ubar ineq1}
		\begin{split}
			\sqrt{2C_1}\k^{\frac1{4\g}} \leq\sqrt{2S_*^{\frac1\g}\Big(D_*-\Big(\frac{\g \bar{p}}{\g-1}\Big)^{\frac{\g-1}{\g}}\Big)}\leq \bar{u} \leq \sqrt{2B^*}.
	\end{split}\end{equation}
	Furthermore, in view of \eqref{BSbar relation}, a straightforward computation gives
	$$\bar{u}'(y)=\left(\bar{B}'(y)-\frac1\g\left(\frac{\g\bar p}{(\g-1)\bar S(y)}\right)^{\frac{\g-1}{\g}}\bar{S}'(y)\right)\bar{u}^{-1}(y).$$
	Then with the help of \eqref{BS condition1}, \eqref{BS condition2} and \eqref{ubar ineq1}, one obtains $\bar{u}'(0)=0$,  $\bar{u}'(\bar{H})\geq0$, and the $L^\infty$-bound for $\bar u'$ in \eqref{ubar estimates}.
	This finishes the proof.
\end{proof}

From now on, we always assume that at the upstream the Bernoulli function $\bar B$ and the entropy function $\bar S$, which belong to  $C^{1,1}([0,\bar{H}])$, satisfy \eqref{BS min}-\eqref{BS condition2} with $\k$ in \eqref{BS condition2} sufficiently small depending on $B_*,\, S_*,\, \gamma$, and $\bar H$. Moreover, we let the mass flux $Q\in(Q_*,Q^*)$, where $Q_*$ and $Q^*$ are defined in Proposition \ref{incoming data}.

It follows from the continuity equation \eqref{Euler system}$_1$ that there is a stream function $\psi$ satisfying
\begin{equation}\label{psi gradient}
	\n\psi(x,y)=(-y\r v, y\r u).
\end{equation}
Note that \eqref{BS equations} implies the Bernoulli function $B$ and the entropy function $S$ are constants along each streamline. With Proposition \ref{incoming data} at hand, we show that $B$ and $S$ can be expressed as functions of the stream function $\psi$ under suitable assumptions.
\begin{proposition}\label{BS pro}
	Let $(\bar\rho,\bar u,\bar p)$ be the incoming  density, axial velocity and pressure of the flow determined in Proposition \ref{incoming data}. Suppose that $(\rho,
	\u,p)$ is a solution to the full Euler system \eqref{Euler system} and each streamline is globally well-defined in the flow region. Then there exist two $C^{1,1}$ functions $\B:[0,Q]\to\R$ and $\S:[0,Q]\to\R$ such that
	\begin{equation}\label{BBSS}
	(B,S)(x,y)=(\B,\S)(\psi(x,y))
	\end{equation}
   and
	\begin{equation}\label{BSpsi relation}
		\B(\psi)=\frac{|\nabla \psi|^2}{2y^2\rho^2}+\r^{\g-1}\S(\psi).
	\end{equation}
Moreover,
\begin{equation}\label{k0}
	\k_0:=\|\B'\|_{C^{0,1}((0,Q])}+\|\S'\|_{C^{0,1}((0,Q])}\leq C\k^{1-\frac1{2\gamma}},
\end{equation}
where $\k$ is the same constant as in \eqref{BS condition2} and  $C=C(B_*,S_*,\bar H,\gamma)>0$.
\end{proposition}
\begin{proof}
	Since the Bernoulli function $B$ and the entropy function $S$ are conserved along each streamline, which is globally well-defined in the flow region, then $B$ and $S$ are uniquely determined by their values at the upstream. Let $\h(z):[0,Q]\to[0,\bar H]$ be the position of the streamline at the upstream where the stream function takes the value $\psi=z$, i.e.
	\begin{equation}\label{h def}
		z=\int_0^{\h(z)}y(\bar{\r}\bar{u})(y)dy.
	\end{equation}
The function $\h$ is well-defined as $\bar \rho\bar u>0$.
Then one has
\begin{equation}\label{BS_psi_bar}
(B,S)(x,y)=(\bar B,\bar S)(\h(\psi(x,y)))=(\B,\S)(\psi(x,y)),
\end{equation}
where $\B$ and $\S$ are defined as
\begin{equation}\label{B Bbar}
	\B(z):=\bar B(\h(z))=\frac{\bar{u}^2(\h(z))}{2}+\frac{\g \bar{p}}{(\g-1)\bar{\r}(\h(z))}
\end{equation}
and
\begin{equation}\label{S Sbar}
	\S(z):=\bar S(\h(z))=\frac{\g\bar{p}}{(\g-1)\bar{\r}^\g(\h(z))}
\end{equation}
for $z\in[0,Q]$, respectively. Therefore \eqref{BBSS} follows. In addition, since by \eqref{psi gradient} the flow velocity $\u=(u,v)$ satisfies $|\u|^2=|\n\psi/(y\r)|^2$, using the definitions of the Bernoulli function in \eqref{B def} and the entropy function in \eqref{S def} as well as \eqref{BBSS} yields \eqref{BSpsi relation}.

To show \eqref{k0}, differentiating \eqref{h def} with respect to $z$ gives
$$\h'(z)=\frac{1}{\h(z)(\bar \r\bar u)(\h(z))},$$
so that differentiating \eqref{B Bbar} and \eqref{S Sbar} one has
\begin{equation}\label{BS_deri}\begin{split}
	&\B'(z)=\frac{\bar{B}'(\h(z))}{\h(z)(\bar{\r}\bar{u})(\h(z))},
	\q \S'(z)=\frac{\bar{S}'(\h(z))}{\h(z)(\bar{\r}\bar{u})(\h(z))},\\
	&\B''(z)=\frac1{\h^2(z)(\bar{\r}^2\bar{u}^2)(\h(z))}\left(\bar{B}''(\h(z))-\bar{B}'(\h(z))\left(\frac{\bar{\r}'(\h(z))}{\bar{\r}(\h(z))}+\frac{\bar{u}'(\h(z))}{\bar{u}(\h(z))}
	+\frac1{\h(z)}\right)\right),\\
	&\S''(z)=\frac1{\h^2(z)(\bar{\r}^2\bar{u}^2)(\h(z))}\left(\bar{S}''(\h(z))-\bar{S}'(\h(z))\left(\frac{\bar{\r}'(\h(z))}{\bar{\r}(\h(z))}+\frac{\bar{u}'(\h(z))}{\bar{u}(\h(z))}
	+\frac1{\h(z)}\right)\right).
\end{split}\end{equation}
Combining the above equalities with \eqref{BS condition2}, \eqref{rhobar_estimate}, and \eqref{ubar estimates} together yields \eqref{k0}.
\end{proof}

For later purpose, we extend the Bernoulli function $\B$ and the entropy function $\S$ from $[0,Q]$ to $\R$ as follows: First, in view of \eqref{BS condition1}, $\bar B$ and $\bar S$ can be extended to $C^{1,1}$ functions defined on $\mathbb R$, which is still denoted by $\bar B$ and $\bar S$ respectively, such that
\begin{equation*}
	\bar{B}'(y)=\bar S'(y)=0 \text{ on } (-\infty,0],
	\q \bar{B}'(y)\geq0, \ \bar{S}'(y)=0  \text{ on } [\bar H,\infty).
\end{equation*}
Besides, the extensions can be made such that
\begin{equation}\label{label_11}
	B_\ast=\inf_{\R}\bar B\leq B^*\leq\sup_{\R} \bar B<\infty \q\text{and}\q  S_\ast=\inf_{\R}\bar S\leq \sup_{\R} \bar S=S^*
\end{equation}
and
\begin{equation}\label{label_13}
	\|\bar B'\|_{C^{0,1}(\R)}\leq \|\bar B'\|_{C^{0,1}([0,\bar H])}
	\q\text{and}\q
	\|\bar S'\|_{C^{0,1}(\R)}=\|\bar S'\|_{C^{0,1}([0,\bar H])},
\end{equation}
where $B_*$ and $S_*$ are defined in \eqref{BS min}, $B^*$ and $S^*$ are defined in \eqref{BS max}, and $\sup_{\R} \bar B$ depends on $B_*$ and $\|\bar B\|_{C^{1,1}([0,\bar H])}$.
Furthermore, the upstream density $\bar \rho$ and axial velocity $\bar u$ can be extended to $C^{1,1}$ functions in $\R$ such that \eqref{rhobar express} and \eqref{ubar express} always hold. Consequently, by
\eqref{B Bbar} and \eqref{S Sbar} one gets extensions of $\mathcal B$ and $\S$ to $C^{1,1}$ functions in $\R$ (still denoted by $\mathcal{B}$ and $\S$), which satisfy
\begin{equation}\label{eq:sign_B}
	\mathcal{B}'(z)=\S'(z)= 0 \text{ on } (-\infty, 0], \q
	\mathcal{B}'(z)\geq 0, \ \S'(z)=0 \text{ on } [Q,\infty).
\end{equation}
In view of \eqref{label_11}, \eqref{label_13}, and \eqref{BS_*^*}, one has
\begin{align}\label{eq:BS_bounds}
	0<B_*\leq\B(z)\leq B_*+C\k \q\text{and}\q
	0<S_*\leq\S(z)\leq S_*+C\k,\q\text{for } z\in\R.
\end{align}
Moreover, it follows from \eqref{BS_deri} that
\begin{equation}\label{BS_deri_R}
\|\B'\|_{C^{0,1}(\R)}\leq C\|\B'\|_{C^{0,1}((0,Q])}\q\text{and}\q
\|\S'\|_{C^{0,1}(\R)}=\|\S'\|_{C^{0,1}((0,Q])},
\end{equation}
where the constant $C$ depends on $B_*$, $S_*$, $\bar H$, and $\g$.

Now we deduce the relation between $(\B,\S)$ and the vorticity  $\o:=\pt_{x}v-\pt_{y}u$.
Differentiating \eqref{BS_relation1} and noting that $$\rho=\left(\frac{\g p}{(\g-1)S}\right)^{\frac1\g},$$ one gets
\begin{equation}\label{B partial}
\pt_{x}B=u\pt_{x}u+v\pt_{x}v+\frac{\r^{\g-1}}\g\pt_{x}S+\frac{\pt_{x}p}{\r}\q\text{and}\q \pt_{y}B=u\pt_{y}u+v\pt_{y}v+\frac{\r^{\g-1}}\g\pt_{y}S+\frac{\pt_{y}p}{\r}.
\end{equation}
These together with
\begin{equation}\label{equiv_eq1}
u\pt_{x}u+v\pt_{y}u+\frac{\pt_{x}p}{\r}=0 \q\text{and}\q u\pt_{x}v+v\pt_{y}v+\frac{\pt_{y}p}{\r}=0,
\end{equation}
which are obtained from the first three equations in \eqref{Euler system}, give
\begin{equation}\label{BSdxdy}
\pt_{x}B-\frac {\r^{\g-1}}\g \pt_{x}S=v\omega \q\text{and}\q \pt_{y}B-\frac {\r^{\g-1}}\g \pt_{y}S=-u\omega.
\end{equation}
Since it follows from \eqref{psi gradient} and \eqref{BBSS} that
\begin{equation}\label{BS gradient}
(\pt_xB,\pt_yB)=y\r \B'(\psi)(-v,u)\q\text{and}\q (\pt_xS,\pt_yS)=y\r \S'(\psi)(-v,u),
\end{equation}
substituting \eqref{BS gradient} into \eqref{BSdxdy} yields
\begin{equation}\label{vorticity express}
\o=-y\r\B'(\psi)+\frac{y\r^\g}{\g}\S'(\psi)
\end{equation}
provided $|\u|\neq0$.
Thus the smooth solution of the full Euler system \eqref{Euler system} satisfies a system consisting of the continuity equation  \eqref{Euler system}$_1$, \eqref{BS equations} and \eqref{vorticity express}. 
Actually, the following proposition shows that the two systems are equivalent under appropriate conditions.

\begin{proposition}\label{equivalent pro}
Let $\tilde\MO\subset \R^2$ be the domain bounded by two streamlines $\N_0$ (the $x$-axis) and
$$
\tilde{\N}:=\{(x,y): x=\tilde N(y),\, \ul h<y<\bar H\},$$
where $0<\ul{h}<\bar{H}<\infty$ and $\tilde N:(\ul h,\bar H)\to\mathbb R$ is a $C^1$ function with
$$\lim_{y\to\bar H-}\tilde N(y)=-\infty \q\text{and}\q \lim_{y\to\ul{h}+}\tilde N(y)=+\infty.$$
Let $\rho:\overline{\tilde \MO}\to(0,\infty)$, $\mathbf u=(u,v):\overline{\tilde \MO}\to\mathbb R^2$ and $p:\overline{\tilde \MO}\to(0,\infty)$
be $C^{1,1}$ in $\tilde \MO$ and continuous up to $\partial {\tilde \MO}$ except finitely many points.
Suppose $\u$ satisfies the slip boundary condition $\u\c\mn=0$ on $\pt \tilde{\MO}$, $(\rho,\u,p)$ satisfies the upstream asymptotic behavior  \eqref{asymptotic upstream}, and
\begin{equation}\label{v negative}
v<0 \q\text{in }\tilde{\MO}.
\end{equation}
Then $(\rho,\u,p)$ solves the full Euler system \eqref{Euler system} in $\tilde \MO$ if and only if $(\rho,\u,p)$ solves the system consisting of  \eqref{Euler system}$_1$, \eqref{BS equations} and \eqref{vorticity express}.
\end{proposition}

\begin{proof}
	In view of \eqref{v negative}, through each point in $\tilde\MO$ there is one and only one streamline which  satisfies
	\begin{equation*}
		\left\{\begin{aligned}
			&\frac{dx}{ds}=u(x(s),y(s))\\
			&\frac{dy}{ds}=v(x(s),y(s))
		\end{aligned}\right.
	\end{equation*}
	and can be defined globally in $\tilde\MO$. Furthermore, any streamline through some point in $\tilde{\MO}$ cannot touch $\pt\tilde\MO$. To show this, without loss of generality, we assume that there exists a streamline through $(-x_0,y_0)$ (with $x_0>0$ sufficiently large) in $\tilde{\MO}$ passing through $(\tilde x, 0)$. By the continuity equation \eqref{Euler system}$_1$ and the slip boundary condition along each streamline, one has
	$$0=\int_{0}^{y_0}s(\r u)(-x_0,s)ds.$$
	This contradicts the asymptotic behavior \eqref{asymptotic upstream}.
	As each streamline is globally well-defined in $\tilde\MO$, from Proposition \ref{BS pro} and previous analysis, it suffices to derive the full Euler system \eqref{Euler system} from the system consisting of \eqref{Euler system}$_1$, \eqref{BS equations} and \eqref{vorticity express}.

Combining the equalities \eqref{vorticity express} and \eqref{psi gradient} gives
$$y\r\u\c\n\left(\frac{\o}{y\r}-\frac{\r^{\g-1}}\g\S'(\psi)\right)=0.$$
With the help of the definition of the entropy function in \eqref{S def} and the expression of its gradient in \eqref{BS gradient}, one can check that the above equality is equivalent to
\begin{equation}\label{equ1}
y\r\u\c\n\left(\frac{\o}{y\r}\right)+\n\c\left(\frac{\pt_{y}p}{\r},-\frac{\pt_{x}p}{\r}\right)=0.
\end{equation}
Since the continuity equation \eqref{Euler system}$_1$ implies
$$y\r\u\c\n\left(\frac{\o}{y\r}\right)=\u\c\n\o-\o\frac{\u\c\n(y\r)}{y\r}=\u\c\n\o+\o\n\c\u$$
where $\o=\pt_xv-\pt_yu$,
it follows from \eqref{equ1} that
$$\pt_{x}\left(u\pt_{x}v+v\pt_{y}v+\frac{\pt_{y}p}{\r}\right)-\pt_{y}\left(u\pt_{x}u+v\pt_{y}u+\frac{\pt_{x}p}{\r}\right)=0.$$
Therefore there exists a function $\Phi$ such that
$$\pt_{x}\Phi=u\pt_{x}u+v\pt_{y}u+\frac{\pt_{x}p}{\r} \q\text{and}\q \pt_{y}\Phi=u\pt_{x}v+v\pt_{y}v+\frac{\pt_{y}p}{\r}.$$
Moreover, using \eqref{BS equations} and \eqref{B partial} one has
$$\u\c\left(u\pt_{x}u+v\pt_{y}u+\frac{\pt_{x}p}{\r}, u\pt_{x}v+v\pt_{y}v+\frac{\pt_{y}p}{\r}\right)=0,$$
that is,
$$\u\c\n\Phi=0.$$
This means $\Phi$ is a constant along each streamline, in particular, along $\N_0$ and $\tilde{\N}$. Besides, the asymptotic behavior  \eqref{asymptotic upstream} yields $\pt_{y}\Phi\to0$ as $x\to-\infty$, so that $\Phi\to C$ as $x\to-\infty$.
Thus the function $\Phi\equiv C$ in the whole domain $\tilde{\MO}$. This implies that $\pt_{x}\Phi=\pt_{y}\Phi\equiv0$ in $\tilde{\MO}$, i.e., \eqref{equiv_eq1} holds globally in $\tilde{\MO}$.

Combining the continuity equation \eqref{Euler system}$_1$,  \eqref{BS equations} and \eqref{equiv_eq1} together gives the full Euler system \eqref{Euler system}. This finishes the proof of the proposition.
\end{proof}

Let us digress for the study on the subsonic states. It follows from the definitions of the Bernoulli function \eqref{B def} and the entropy function \eqref{S def} as well as \eqref{BBSS} that, for given value  $z$, the flow density $\rho$ and the flow speed $q$ satisfy
\begin{equation*}\label{BSq relation}
	\B(z)=\frac{q^2}{2}+\r^{\g-1}\S(z),
\end{equation*}
and
\begin{equation}\label{eq:q_rho}
q=\mathfrak q(\rho,z)=\sqrt{2(\B(z)-\r^{\g-1}\S(z))}.
\end{equation}
Moreover, using \eqref{S def} and the definition of the sound speed \eqref{sonic mach} one has
$$c=\sqrt{(\g-1)\rho^{\g-1}\mathcal S(z)}.$$
Let
\begin{equation}\label{rho cm}
	\varrho_c(z):=\left(\frac{2\B(z)}{(\g+1)\S(z)}\right)^{\frac1{\g-1}} \q\text{and}\q \varrho_m(z):=\left(\frac{\B(z)}{\S(z)}\right)^{\frac1{\g-1}}
\end{equation}
be the critical density and maximum density respectively. The flow speed $q=\mathfrak q(\rho,z)$ is well-defined if and only if $\rho\leq \varrho_m(z)$. The flow is subsonic, i.e., $q\in[0,c)$, if and only if
$q\in[0,\mathfrak{q}(\varrho_c(z),z))$, or equivalently   $\r\in(\varrho_c(z),\varrho_m(z)]$. Denote the square of the momentum and the square of the critical momentum by
\begin{equation}\label{t}
	\t(\rho, z):=\rho^2\mathfrak{q}^2(\rho, z)=2\rho^2(\B(z)-\rho^{\g-1}\S(z))
\end{equation}
and
\begin{equation}\label{tc}
	\t_c(z):=\t(\varrho_c(z),z)=(\g-1)\left(\frac{2\B(z)}{\g+1}\right)^{\frac{\g+1}{\g-1}}\S^{-\frac2{\g-1}}(z),
\end{equation}
respectively. In view of \eqref{eq:BS_bounds}, there exist positive constants $C$, $t_*$, and $t^*$, which depend on $B_*,\, S_*,\, \bar H$, and $\g$, such that
\begin{align}\label{tc_bound}
	0<C\leq\varrho_c(z)<\varrho_m(z)\leq C^{-1}
	\q\text{and}\q 0<t_*\leq\mathfrak t_c(z)\leq {t}^*, \q \text{for }z\in \R.
\end{align}

\begin{lemma}\label{g}
Suppose the density function $\r$ and the stream function $\psi$ satisfy \eqref{BSpsi relation}. Then the following statements hold.
\begin{itemize}
	\item [(i)] The density function $\r$ can be expressed as a function of $|\n\psi/y|^2$ and $\psi$ in the subsonic region, i.e.,
	\begin{equation}\label{rho g}
		\r=\frac 1{g(|\frac{\n\psi}{y}|^2,\psi)}, \q \text{if } \r\in (\varrho_c(\psi),\varrho_m(\psi)],
	\end{equation}
	where $\varrho_c$ and $\varrho_m$ are functions defined in \eqref{rho cm}, and
	$$g:\{(t,z):0\leq t<\mathfrak t_c(z),\, z\in \R\}\to \R$$
	is a function smooth in $t$ and $C^{1,1}$ in $z$ with $\mathfrak t_c$ defined in \eqref{tc}. Furthermore,
	\begin{equation}\label{g bound}
	\left(\sup_{z\in\R}\varrho_m(z)\right)^{-1}=:g_*\leq g(t,z)\leq g^*:=\left(\inf_{z\in\R}\varrho_c(z)\right)^{-1}.
	\end{equation}
	\item [(ii)] The function $g$ satisfies the identity
	\begin{equation}\label{dzg dtg}
	\pt_z g(t,z)=-2\pt_t g(t,z)\frac{\B'(z)-g^{1-\g}(t,z)\S'(z)}{g^2(t,z)}, \quad\text{for } t\in[0,\mathfrak t_c(z)), \, z\in\R.
	\end{equation}
\end{itemize}
\end{lemma}

\begin{proof}
(i) From the expression \eqref{t}, the straightforward computations give
\begin{equation}\label{dF}
\pt_\rho\t(\rho, z)=2\rho\left(2\B(z)-(\g+1)\rho^{\g-1}\S(z)\right).
\end{equation}
Now one can see that with $\varrho_c$ and $\varrho_m$ defined in \eqref{rho cm}, the following statements hold:
\begin{enumerate}
  \item[\rm (a)] $\rho\mapsto\t(\rho,z)$ achieves its maximum $\t_c(z)$ at $\rho=\varrho_c(z)$;
  \item[\rm (b)] $\pt_\rho\t(\rho,z)<0$ when $\varrho_c(z)<\rho\leq\varrho_m(z)$.
\end{enumerate}
Thus by the inverse function theorem, if $\rho\in(\varrho_c(z),\varrho_m(z)]$, $\rho$ can be expressed as a function of $t:=\t(\rho,z)\in[0,\t_c(z))$ and $z$, that is, $\rho=\rho(t,z)$. Let
\begin{equation}\label{rho g tz}
g(t,z)=\frac1{\rho(t,z)}, \q\text{for } t\in[0,\mathfrak t_c(z)).
\end{equation}
The function $g$ is smooth in $t$ by the inverse function theorem and $C^{1,1}$ in $z$ by the $C^{1,1}$ regularity of $\B$ and $\S$.
This together with \eqref{BSpsi relation} completes the proof of (i).

(ii) Differentiating \eqref{t}
with respect to $t:=\mathfrak t(\rho,z)$ gives
\begin{equation}\label{bdt}
	\frac12=\frac{\pt_t\r}{\r}(t-(\g-1)\r^{\g+1}\S(z)),
\end{equation}
and differentiating \eqref{t} with respect to $z$ gives
\begin{equation}\label{bdz}
	\frac{\pt_z\r}{\r}(t-(\g-1)\r^{\g+1}\S(z))=-\r^2(\B'(z)-\r^{\g-1}\S'(z)).
\end{equation}
Combining \eqref{bdt} and \eqref{bdz} together one has
$$\pt_z\r=-2\r^2\pt_t\r(\B'(z)-\r^{\g-1}\S'(z)).$$
In view of \eqref{rho g tz} one has \eqref{dzg dtg}.
\end{proof}

\begin{lemma}\label{quasilinearequ lem}
Let $(\r,\u,p)$ be a solution to the system consisting of \eqref{Euler system}$_1$, \eqref{BS equations}, and \eqref{vorticity express}. Assume $(\r,\u,p)$ satisfies the assumptions in Proposition \ref{equivalent pro}. Then in the subsonic region $|\n\psi/y|^2<\t_c(\psi)$, the stream function $\psi$ solves
\begin{equation}\label{elliptic equ}
\n\c\bigg(g\bigg(\left|\frac{\n\psi}{y}\right|^2,\psi\bigg)\frac{\n\psi}{y}\bigg)
=\frac{y\B'(\psi)}{g(|\frac{\n\psi}{y}|^2,\psi)}-\frac{y\S'(\psi)}{\g g^\g(|\frac{\n\psi}{y}|^2,\psi)},
\end{equation}
where $g$ is defined in \eqref{rho g}, $\B$ is the Bernoulli function and $\S$ is the entropy function as in Proposition \ref{BS pro}. Moreover, the equation \eqref{elliptic equ} is elliptic if and only if $|\n\psi/y|^2<\t_c(\psi)$.
\end{lemma}

\begin{proof}
Expressing the vorticity $\o$ in terms of the stream function $\psi$ and using \eqref{vorticity express} one has
$$-\n\c\left(\frac{\n\psi}{y\r}\right)=\o=-y\r\B'(\psi)+\frac{y\r^\g}{\g}\S'(\psi).$$
In view of \eqref{rho g} the above equation can be rewritten into \eqref{elliptic equ}.

The equation \eqref{elliptic equ} can be written in the nondivergence form as follows
$$\mathfrak  a^{ij}\bigg(\frac{\n\psi}{y},\psi\bigg)\pt_{ij}\psi+\mathfrak b\bigg(\frac{\n\psi}{y},\psi\bigg)=\frac{y^2\B'(\psi)}{g(|\frac{\n\psi}{y}|^2,\psi)}-\frac{y^2\S'(\psi)}{\g g^\g(|\frac{\n\psi}{y}|^2,\psi)},$$
where the matrix
$$(\mathfrak a^{ij})=g\bigg(\left|\frac{\n\psi}{y}\right|^2,\psi\bigg)I_{2}+2\pt_t g\bigg(\left|\frac{\n\psi}{y}\right|^2,\psi\bigg)\frac{\n\psi}{y}\otimes\frac{\n\psi}{y}$$
is symmetric with the eigenvalues
$$\b_0=g\bigg(\left|\frac{\n\psi}{y}\right|^2,\psi\bigg) \q\text{and}\q \b_1=g\bigg(\left|\frac{\n\psi}{y}\right|^2,\psi\bigg)+2\pt_tg\bigg(\left|\frac{\n\psi}{y}\right|^2,\psi\bigg)\left|\frac{\n\psi}{y}\right|^2,$$
and
$$\mathfrak b=\pt_zg\bigg(\left|\frac{\n\psi}{y}\right|^2,\psi\bigg)|\nabla \psi|^2-\frac{\pt_y\psi}{y}\bigg[g\bigg(\left|\frac{\n\psi}{y}\right|^2,\psi\bigg)+2\pt_tg\bigg(\left|\frac{\n\psi}{y}\right|^2,\psi\bigg)\left|\frac{\n\psi}{y}\right|^2\bigg].
$$
Besides, differentiating the identity $t=\t(\frac{1}{g(t,z)},z)$ gives
\begin{equation}\label{dtg drF}
\pt_tg(t,z)=-\frac{g^2(t,z)}{\pt_\rho\t(\frac1{g(t,z)},z)} \q\text{for } t\in[0,\t_c(z)).
\end{equation}
This implies
$$\pt_tg(t,z)\geq0 \q \text{and}\q \lim_{t\to\t_c(z)-}\pt_tg(t,z)=+\infty.$$
Thus $\b_0$ has uniform upper and lower bounds depending only on $B_*$, $S_*$, $\bar H$, and $\g$, and $\b_1$ has a uniform lower bound but blows up when $|\n\psi/y|^2$ approaches $\t_c(\psi)$. Therefore, the equation \eqref{elliptic equ} is elliptic as long as $|\n\psi/y|^2<\t_c(\psi)$, and is singular when $|\n\psi/y|^2=\t_c(\psi)$. This completes the proof of the lemma.
\end{proof}

\subsection{Reformulation for the jet flows in terms of the stream function}

Let
\begin{equation}\label{Ld}
\Ld:=\r_0\sqrt{2\bar{B}(\bar{H})-2\r_0^{\g-1}\bar{S}(\bar{H})}
\q\text{with}\q
\r_0 :=\left(\frac{\g\ul p}{(\g-1)\bar{S}(\bar{H})}\right)^{\frac1\g}
\end{equation}
be the constant momentum on the free boundary, where $\ul p$ is the outer pressure on the free boundary as in Problem \ref{probelm 1}. From the previous analysis, Problem \ref{probelm 1} can be solved as long as the following problem in terms of the stream function $\psi$ is solved.

\begin{problem}\label{problem 2}
Let $(\bar{B}, \bar{S})\in (C^{1,1}([0,\bar{H}]))^2$ be the Bernoulli function and the entropy function at the upstream satisfying \eqref{BS min}-\eqref{BS condition2}. Let $Q\in(Q_*,Q^*)$ be the mass flux where $Q_*$ and $Q^*$ are defined in Proposition \ref{incoming data}. One looks for a triple $(\psi,\G_\psi,\Ld)$ satisfying $\psi\in C^{2,\alpha}(\{\psi<Q\})\cap C^{1}(\overline{\{\psi<Q\}})$ for any  $\alpha\in(0,1)$, $\pt_{x}\psi>0$ in $\{0<\psi<Q\}$ and
\begin{equation}\label{jet_equ}
	\left\{
	\begin{aligned}
		&\n\c\bigg(g\bigg(\left|\frac{\n\psi}{y}\right|^2,\psi\bigg)\frac{\n\psi}{y}\bigg)=\frac{y\B'(\psi)}{g(|\frac{\n\psi}{y}|^2,\psi)}-\frac{y\S'(\psi)}{\g g^\g(|\frac{\n\psi}{y}|^2,\psi)} &&\text{in } \{0<\psi<Q\},\\
		&\psi =0 &&\text{on } \N_0,\\
		&\psi =Q &&\text{on } \N \cup \Gamma_\psi,\\
		&\left|\frac{\nabla \psi}y\right| =\Lambda &&\text{on } \Gamma_\psi,
	\end{aligned}
	\right.
\end{equation}
where the free boundary $\G_\psi:=\pt\{\psi<Q\}\backslash \N$. Furthermore, the free boundary $\G_\psi$ and the flow
\begin{equation}\label{eq:rup_psi}
(\r,\u,p)=\left(\frac1{g(|\n\psi/y|^2,\psi)},\frac{g(|\n\psi/y|^2,\psi)\nabla^{\perp}\psi}{y},\frac{(\gamma-1)\S(\psi)}{\gamma g^\gamma(|\n\psi/y|^2,\psi)}\right)
\end{equation}
are expected to have the following properties.
\begin{enumerate} 
\item The free boundary $\Gamma_\psi$ is given by a graph $x=\Upsilon(y)$, $y\in (\ubar{H}, 1]$ for some $C^{2,\alpha}$ function $\Upsilon$ and some $\ubar{H}\in (0,1)$.

\item The free boundary $\G_\psi$ fits the nozzle at $A=(0,1)$ continuous differentiably, i.e., $\Upsilon(1)=N(1)$ and $\Upsilon'(1)=N'(1)$.

\item For $x$ sufficiently large, the free boundary is also an $x$-graph, i.e., it can be written as $y=f(x)$ for some $C^{2,\alpha}$ function $f$. Furthermore, one has 
$$\lim_{x\to+\infty}f(x)=\ul{H}  \q \text{and}\q \lim_{x\to+\infty}f'(x)=0.$$

\item The flow is subsonic in the flow region, i.e., $|\n\psi/y|^2<\t_c(\psi)$ in $\{\psi<Q\}$, where $\t_c$ is defined in \eqref{tc}.

\item At the upstream the flow satisfies the asymptotic behavior \eqref{asymptotic upstream}, where the upstream states $(\bar\rho,\bar u,\bar p)$ are determined in Proposition \ref{incoming data}. At the downstream the flow satisfies
$$\lim_{x\to+\infty}(\r,\u,p)=(\ul{\r},\ul{u},0,\ul p)$$
for some functions $\ul{\r}=\ul{\r}(y)$ and $\ul{u}=\ul{u}(y)$, and some positive constant $\ul p$.
\end{enumerate}
\end{problem}
\begin{remark}
In fact, the conditions \eqref{BS condition1} and \eqref{BS condition2} imply the vorticity of the flow is small. It follows from \eqref{k0} that the right-hand side of the equation in \eqref{jet_equ} is small.
\end{remark}

The following proposition shows that once we obtain a solution $\psi$ to Problem \ref{problem 2}, then $\rho,\,\mathbf u$, and $p$ defined in \eqref{eq:rup_psi} solve the Euler system \eqref{Euler system}, even if $\rho,\,\mathbf u$, and $p$ may not be $C^{1,1}$ as required in Proposition \ref{equivalent pro}.

\begin{proposition}\label{prop:equiv_sol}
	 Assume that a function $\psi\in C^{2,\alpha}(\{\psi<Q\})\cap C^{1}(\overline{\{\psi<Q\}})$ is a subsonic solution of \eqref{jet_equ}, where $\alpha\in(0,1)$. Then the flow $(\rho,\mathbf u,p)$ defined in \eqref{eq:rup_psi} solves the Euler system \eqref{Euler system} in the flow region $\{0<\psi<Q\}$.
\end{proposition}
\begin{proof}
	Due to the regularity of the function $\psi$, the functions $\rho$, $\mathbf u:=(u,v)$, and $p$ defined in \eqref{eq:rup_psi} belong to $C^{1,\alpha}(\{\psi<Q\})\cap C^0(\overline{\{\psi<Q\}})$. Since $\rho$ and $\mathbf u$  satisfy \eqref{psi gradient}, it follows from the definition of the function $g$ in Lemma \ref{g} (i) that  the density $\rho=1/g$ satisfies
	\begin{align*}
		\frac{u^2+v^2}{2}+\r^{\g-1}\S(\psi)=	\frac{|\nabla\psi|^2}{2y^2\rho^2}+\r^{\g-1}\S(\psi)=\mathcal{B}(\psi)
	\end{align*}
	in the flow region.
	Differentiating the above equality with respect to $x$ and $y$, respectively, one gets
	\begin{equation}\label{eq_1}
		u\partial_{x}u+v\partial_{x}v +\frac{\r^{\g-1}}\g\S'(\psi)\pt_{x}\psi+\frac{\pt_{x}p}{\r}=\B'(\psi)\partial_{x}\psi,
	\end{equation}
	and
	\begin{align}\label{eq_2}
		u\partial_{y}u+v\partial_{y}v+\frac{\r^{\g-1}}\g\S'(\psi)\pt_{y}\psi+\frac{\pt_{y}p}{\r}=\B'(\psi)\partial_{y}\psi.
	\end{align}
	Note that by the quasilinear equation in \eqref{jet_equ} one has
	$$\B'(\psi)-\frac{\r^{\g-1}}{\g}\S'(\psi)=-\frac{\pt_{x}v-\pt_{y}u}{y\rho}.$$
	Thus the equalities \eqref{eq_1} and \eqref{eq_2} together with \eqref{psi gradient} yield \eqref{equiv_eq1}. Note that $\rho$ and $\mathbf u$ satisfy the continuity equation (the first equation in \eqref{Euler system}). Then the Euler system \eqref{Euler system} follows from the continuity equation, \eqref{equiv_eq1}, and the equality $y\rho\mathbf u\cdot \nabla \B(\psi)=0$ where $\B(\psi)$ satisfies \eqref{BSpsi relation}. This finishes the proof of the proposition.
\end{proof}

\subsection{Subsonic truncation}

One of the major difficulties to solve Problem \ref{problem 2} is that the equation \eqref{elliptic equ} becomes degenerate as the flows approach the sonic state. Hence we use a subsonic truncation so that after the truncation the equation \eqref{elliptic equ} is always uniformly elliptic.

Let $\varpi:\R\to[0,1]$ be a smooth nonincreasing function such that
$$\varpi(s)=\begin{cases}
	1& \text{ if } s\leq-1,\\
	0& \text{ if } s\leq -\frac12,
\end{cases}
\q\text{and}\q
|\varpi'|+|\varpi''|\leq8.$$
For $\v\in(0,\min\{t_*,1\}/4)$ with $t_*$ as in \eqref{tc_bound}, let $\varpi_\v(s):=\varpi(s/\v)$. We define
\begin{equation}\label{gm}
	g_\v(t,z)=g(t,z)\varpi_\v(t-\t_c(z))+(1-\varpi_\v(t-\t_c(z)))g^*,
\end{equation}
where $\t_c(z)$ is defined in \eqref{tc} and $g^*$ is the upper bound for $g$ in \eqref{g bound}. The properties of $g_\v$ are summarized in the following lemma.

\begin{lemma}\label{g properties}
Let $g$ be the function determined in Lemma \ref{g}, and let $g_\v$ be the subsonic truncation of $g$ defined in \eqref{gm}. Then the function $g_\v(t,z)$ is smooth in $t$ and $C^{1,1}$ in $z$. Furthermore, there exist constants $C_*,\, C^*>0$ depending on $B_*$, $S_*$, $\bar H$, and $\g$ such that for all $(t,z)\in[0,\infty)\times\R$ one has
  \begin{align}
  &C_*\leq g_\v(t,z)\leq C^*, \label{gm bound}\\
  &C_*\leq g_\v(t,z)+2\pt_tg_\v(t,z)t\leq C^*\v^{-1},\label{g dtgmt}\\
  &|\pt_zg_\v(t,z)|\leq C^*\v^{-1}\k_0,\quad
  |t\pt_zg_\v(t,z)|\leq C^*\v^{-1}\k_0. \label{dzgm}
  \end{align}
\end{lemma}

\begin{proof}
	It follows from Lemma \ref{g} and the definition of $\varpi_\v$ that $g_\v$ is smooth with  respect to $t$ and $C^{1,1}$ with respect to $z$.
	
	\emph{(i).}  Clearly \eqref{gm bound} follows directly from the definition of $g_\v$ in \eqref{gm}, \eqref{g bound}, and \eqref{tc_bound}.
	
	\emph{(ii).} To show \eqref{g dtgmt}, we first claim that if $0\leq t\leq\mathfrak t_c(z)-\frac{\v}2$ then
	\begin{equation}\label{dtg bound}
		0<\pt_tg(t,z)\leq \v^{-1}g(t,z).
	\end{equation}
     In view of the expression of $\pt_tg(t,z)$ in \eqref{dtg drF}, it suffices to estimate $\pt_\rho\t(\rho,z)$ at $\rho=\rho(t,z)$ for $0\leq t\leq\mathfrak t_c(z)-\frac{\v}2$. In fact, for $0\leq t\leq\mathfrak t_c(z)-\frac{\v}2$, since $\rho>\varrho_c(t,z)$ and $\t_c(z)=(\g-1)\varrho_c^{\g+1}(z)\S(z)$,
    it follows from \eqref{dF} and \eqref{t} that
    \begin{align*}
    	\pt_\r\t\left(\rho,z\right)
    	&=\frac{2}{\rho}\left(\t\left(\rho,z\right)-(\g-1)\rho^{\g+1}\S(z)\right)\leq\frac{2}{\rho}\left(\t\left(\rho,z\right)-(\g-1)\varrho_c^{\g+1}(z)\S(z)\right)\\
    	&= \frac{2}{\rho}\left(\t\left(\rho,z\right)-\t_c(z)\right)\leq-\frac{\v}{\rho}.
    \end{align*}
    Then using \eqref{rho g tz} and \eqref{dtg drF} one has
    \eqref{dtg bound}.
	
	A straightforward computation gives
	\begin{equation}\label{gmdt}
		\pt_tg_\v(t,z)=\pt_tg(t,z)\varpi_\v(t-\t_c(z))+(g(t,z)-g^*)\varpi_\v'(t-\t_c(z)).
	\end{equation}
     Combining \eqref{dtg bound}, \eqref{g bound} and the definition of $\varpi_\epsilon$ together gives
     \begin{equation}\label{dtgm bound}
     	0<\pt_tg_\v(t,z)\leq C\v^{-1}
     \end{equation}
     for some $C=C(B_*,S_*,\bar H,\gamma)$.
     Note that $\pt_tg_\v(t,z)$ is supported in $\{(t,z):0\leq t\leq\t_c(z)-\v/2, \, z\in\R\}$.
     The estimate \eqref{g dtgmt} follows directly from \eqref{gm bound} and \eqref{dtgm bound}.
	
	\emph{(iii).}  Direct computations yield
	\begin{equation}\label{dzg express}\begin{split}
		\pt_zg_\v(t,z)=&\pt_zg(t,z)\varpi_\v(t-\t_c(z))+(g^*-g(t,z))\varpi'_\v(t-\t_c(z))\t_c'(z)\\
		\stackrel{\eqref{dzg dtg}}{=}&-2\pt_t g(t,z)\frac{\B'(z)-g^{1-\g}(t,z)\S'(z)}{g^2(t,z)}\varpi_\v(t-\t_c(z))\\
		&+(g^*-g(t,z))\varpi'_\v(t-\t_c(z))\t_c'(z),
	\end{split}\end{equation}
   where
   \begin{equation}\label{tc_deri}
  	\t_c'(z)=\left(\frac2{\gamma+1}\right)^{\frac{\gamma+1}{\gamma-1}}\left((\gamma+1)\left(\frac{\B(z)}{\S(z)}\right)^{\frac2{\gamma-1}}\B'(z)-2\left(\frac{\B(z)}{\S(z)}\right)^{\frac{\gamma+1}{\gamma-1}}\S'(z)\right)
  \end{equation}
   by the definition of $\t_c$ in \eqref{tc}.
   Thus using the expression of $\pt_tg_\v$ in \eqref{gmdt}, the boundedness of $g(t,z)$ in \eqref{g bound}, the definition of $\k_0$ in \eqref{k0}, and \eqref{BS_deri_R} one has
   \begin{equation}\label{dzgm_dtgm}
   	|\pt_zg_\v(t,z)|
   	\leq C(|\B'(z)|+|\S'(z)|)\pt_tg_\v(t,z)\leq C\k_0\pt_tg_\v(t,z)
   \end{equation}
   for some $C=C(B_*,S_*,\bar H,\gamma)$. This together with \eqref{dtgm bound} gives the first estimate in \eqref{dzgm}. Furthermore, noting that $\pt_zg_\v=0$ for $t\geq \t_c(z)-\frac{\v}2$ and that $\t_c(z)\leq t^*$ (cf. \eqref{tc_bound}), using \eqref{dzgm_dtgm} and \eqref{dtgm bound} again one has
   $$|t\pt_zg_\v(t,z)|
   \leq t^* C\k_0\pt_tg_\v(t,z)\leq C^*\v^{-1}\k_0.$$
   This finishes the proof for \eqref{dzgm}.
\end{proof}

\section{Variational formulation for the free boundary problem}\label{sec variational formulation}

In this section, we show that the quasilinear equation \eqref{elliptic equ} is an Euler-Lagrange equation for an energy functional, so that the jet problem can be characterized by a variational problem.

Let $\O$ be the domain bounded by $\N_0$ and $\N\cup([0,\infty)\times\{1\})$. Notice that $\O$ is unbounded, so we make an approximation by considering the problems in a series of truncated domains $\O_{\mu,R}:=\O\cap\{-\mu<x<R\}$, where $\mu$ and $R$ are positive numbers.  To define the functional, we set
\begin{equation}\label{G def}
G_\v(t,z):=\frac12\int_0^tg_\v(s,z)ds+\frac{\g-1}{\g}(g^{-\g}_\v(0,z)\S(z)-g_\v^{-\g}(0,Q)\S(Q)),
\end{equation}
and
\begin{equation}\label{Phi def}
\Phi_\v(t,z):=-G_\v(t,z)+2\pt_tG_\v(t,z)t.
\end{equation}
In view of the elliptic condition \eqref{g dtgmt},  straightforward computations yield
\begin{align}\label{Phim dt}
	0<\frac12C_*\leq\pt_t\Phi_\v(t,z)=\frac12g_\v(t,z)+\pt_tg_\v(t,z)t\leq\frac12C^*\v^{-1}.
\end{align}
Since $\Phi_\v(0,Q)=-G_\v(0,Q)=0,$ one has $\Phi_\v(t,Q)>0$ for $t>0$.

Given a function $\psi^\sharp_{\mu,R}\in C(\pt \O_{\mu,R})\cap H^1(\Omega_{\mu,R})$ with $0\leq\psi^\sharp_{\mu,R}\leq Q$. We consider the following minimization problem:
\begin{equation}\label{variation problem}
\text{find }\psi\in K_{\psi^\sharp_{\mu,R}}  \, \text{ s.t. } \, J_{\mu,R,\Ld}^\v(\psi)= \inf_{\phi\in K_{\psi^\sharp_{\mu,R}}} J_{\mu,R,\Ld}^\v(\phi),
\end{equation}
where the functional
\begin{equation}\label{Jm def}
J^\v_{\mu,R,\Ld}(\phi):=\int_{\Omega_{\mu,R}}y\bigg[G_\v\bigg(\left|\frac{\n\phi}{y}\right|^2,\phi\bigg)+\ld_\v^2\chi_{\{\phi<Q\}}\bigg]dX, \q \ld_\v: =\sqrt{\Phi_\v(\Ld^2,Q)}
\end{equation}
with $\Ld$ given by \eqref{Ld}, and the admissible set
\begin{equation}\label{K_admissible}
	K_{\psi^\sharp_{\mu,R}}:=\{\phi\in H^1(\O_{\mu,R}):\phi=\psi^\sharp_{\mu,R}\text{ on } \pt\O_{\mu,R}\}.
\end{equation}

The existence and the H\"{o}lder regularity of minimizers for \eqref{variation problem} will be shown in Lemmas \ref{minimizer existence} and \ref{psi holder}, respectively. Assuming the existence and continuity of a minimizer $\psi$ firstly, we derive the Euler-Lagrange equation for the minimization problem \eqref{variation problem} in the open set $\O_{\mu,R}\cap\{\psi<Q\}$.

\begin{lemma}\label{EL}
Let $\psi$ be a minimizer of the minimization problem \eqref{variation problem}. Assume that $\O_{\mu,R}\cap\{\psi<Q\}$ is open. Then $\psi$ is a solution to
\begin{equation}\label{EL equ}
\n\c\bigg(g_\v\bigg(\left|\frac{\n\psi}{y}\right|^2,\psi\bigg)\frac{\n\psi}{y}\bigg)=y\pt_zG_\v\bigg(\left|\frac{\n\psi}{y}\right|^2,\psi\bigg) \q\text{in } \O_{\mu,R}\cap\{\psi<Q\}.
\end{equation}
Furthermore, if $|\n\psi/y|^2\leq\t_c(\psi)-\epsilon$, it holds that
\begin{equation}\label{Gdz expression}
\pt_zG_\v\bigg(\left|\frac{\n\psi}{y}\right|^2,\psi\bigg)=\frac{\B'(\psi)}{g(|\frac{\n\psi}{y}|^2,\psi)}-\frac{\S'(\psi)}{\g g^\g(|\frac{\n\psi}{y}|^2,\psi)}.
\end{equation}
\end{lemma}

\begin{proof}
Let $\eta\in C^\infty_0(\O_{\mu,R}\cap\{\psi<Q\})$. Direct computations give
$$\frac{d}{d\tau}J_{\mu,R,\Ld}^\v(\psi+\tau\eta)\bigg|_{\tau=0}
=\int_{\O_{\mu,R}}2\pt_tG_\v\bigg(\left|\frac{\n\psi}{y}\right|^2,\psi\bigg)\frac{\n\psi}{y}\c\n\eta
+y\pt_zG_\v\bigg(\left|\frac{\n\psi}{y}\right|^2,\psi\bigg)\eta.$$
By the definition of $G_\v$ in \eqref{G def}, minimizers of \eqref{variation problem} satisfy the equation \eqref{EL equ}.

Next, by the definitions of $G_\v$ in \eqref{G def} and $g_\v$ in \eqref{gm} one has
$$\pt_zG_\v(t,z)=\frac12\int_0^t\pt_z g(s,z)ds-(\g-1)g^{-\g-1}(0,z)\pt_zg(0,z)\S(z)+\frac{\g-1}{\g}g^{-\g}(0,z)\S'(z)$$
in $\mathcal R_\v:=\{(t,z):0\leq t\leq\t_c(z)-\v, z\in\R\}$.
Using \eqref{dzg dtg} yields
\begin{align*}
\frac12\int_0^t\pt_z g(s,z)ds&=-\int_0^t\pt_t g(s,z)\frac{\B'(z)-g^{1-\g}(s,z)\S'(z)}{g^2(s,z)}ds\\
&=\int_0^t\pt_t\r(s,z)\B'(z)ds-\frac1\g\int_0^t\pt_t\r^\g(s,z)\S'(z)ds\\
&=\B'(z)(\r(t,z)-\r(0,z))-\frac1\g\S'(z)(\r^\g(t,z)-\r^\g(0,z)).
\end{align*}
Taking $t=0$ in \eqref{bdz} gives
\begin{equation*}\label{dtr0z}
\pt_z\r(0,z)=\frac{\B'(z)-\r^{\g-1}(0,z)\S'(z)}{(\g-1)\r^{\g-2}(0,z)\S(z)},
\end{equation*}
so that
\begin{align}\label{II}
-(\g-1)g^{-\g-1}(0,z)\pt_zg(0,z)\S(z)&=(\g-1)\r^{\g-1}(0,z)\pt_z\r(0,z)\S(z)\nonumber\\
&=\r(0,z)\left(\B'(z)-\r^{\g-1}(0,z)\S'(z)\right).
\end{align}
Moreover, it holds
$$\frac{\g-1}{\g}g^{-\g}(0,z)\S'(z)=\frac{\g-1}{\g}\r^\g(0,z)\S'(z).$$
Combining the above equalities together gives
\begin{equation}\label{Gmdz expression}
	\pt_zG_\v(t,z)=\frac{\B'(z)}{g(t,z)}-\frac{\S'(z)}{\g g^\g(t,z)} \q\text{in } \mathcal R_\v.
\end{equation}
This finishes the proof for the lemma.
\end{proof}

The following lemma shows that minimizers of \eqref{variation problem} satisfy the free boundary condition in a weak sense.

\begin{lemma}\label{free BC}
Let $\psi$ be a minimizer of the minimization problem \eqref{variation problem}. Assume that $\O_{\mu,R}\cap\{\psi<Q\}$ is open. Then
$$\Phi_\v\bigg(\left|\frac{\n\psi}{y}\right|^2,\psi\bigg)=\ld_\v^2 \quad\text{on } \G_{\psi}:=\O_{\mu,R}\cap\pt\{\psi<Q\}$$
in the sense that
$$\lim_{s\to0+}\int_{\pt\{\psi<Q-s\}}y\bigg[\Phi_\v\bigg(\left|\frac{\n\psi}{y}\right|^2,\psi\bigg)-\ld_\v^2\bigg](\eta\c\nu)d\mathcal{H}^1=0 \q\text{for any }
\eta\in C_0^\infty(\O_{\mu,R};\R^2),$$
where $\mathcal{H}^1$ is the one-dimensional Hausdorff measure.
\end{lemma}

\begin{remark}[Relation between $\lambda_\epsilon$ and $\Lambda$]\label{ld Ld}
	If the free boundary $\Gamma_\psi$ is smooth and $\psi$ is smooth near $\Gamma_\psi$, then it follows from the monotonicity of $t \mapsto \Phi_\epsilon(t,z)$ for each $z$ and the definition of $\lambda_\epsilon$ in \eqref{Jm def}  that $|\nabla\psi/y|=\Lambda$ on $\Gamma_\psi.$
		Moreover, since
		$$\lambda_\epsilon^2=\Phi_\epsilon(\Lambda^2, Q) =\Phi_\epsilon(0, Q)+\int_0^{\Lambda^2}\pt_t\Phi_\epsilon(s,Q)ds,$$
		then in view of \eqref{Phim dt} and $\Phi_\epsilon(0,Q)=0$ one has
		\begin{align}\label{ld_Ld}
			\frac12C_*\Lambda^2	\leq\lambda_\epsilon^2\leq \frac12C^*\epsilon^{-1}\Lambda^2,
		\end{align}
		where $C_*$ and $C^*$ are the same constants as in \eqref{g dtgmt}.
\end{remark}

\noindent\emph{Proof of Lemma \ref{free BC}.}
Let $\eta=(\eta_1,\eta_2)\in C_0^\infty({\O_{\mu,R}};\R^2)$ and $\eta_\tau(X):=X+\tau\eta(X)$ for $\tau\in\R$. If $|\tau|$ is sufficiently small, then $\eta_\tau$ is a diffeomorphism of $\Omega_{\mu,R}$. Define $\psi_\tau(X):=\psi(\eta_\tau^{-1}(X))$. Since $\psi_\tau\in \K_{\psi^\sharp_{\mu,R}}$ and $\psi$ is a minimizer, then
\begin{align*}
0\leq& J^\v_{\mu,R,\Ld}(\psi_\tau)-J^\v_{\mu,R,\Ld}(\psi)\\
=&\int_{{\O_{\mu,R}}}(y+\tau\eta_2)\bigg[{G_\v}\bigg(\left|\frac{\n\psi(\n\eta_\tau)^{-1}}{y+\tau\eta_2}\right|^2,\psi\bigg)
+\ld_\v^2\chi_{\{\psi<Q\}}\bigg]{\rm det}(\n\eta_{\tau})\\
&-\int_{{\O_{\mu,R}}}y\bigg[{G_\v}\bigg(\left|\frac{\n\psi}{y}\right|^2,\psi\bigg)+\ld_\v^2\chi_{\{\psi<Q\}}\bigg]\\
=&\tau\int_{{\O_{\mu,R}}}y\bigg[{G_\v}\bigg(\left|\frac{\n\psi}{y}\right|^2,\psi\bigg)+\ld_\v^2\chi_{\{\psi<Q\}}\bigg]\n\c\eta\\
&+\tau\int_{{\O_{\mu,R}}}\bigg[{G_\v}\bigg(\left|\frac{\n\psi(\n\eta_\tau)^{-1}}{y+\tau\eta_2}\right|^2,\psi\bigg)
+\ld_\v^2\chi_{\{\psi<Q\}}\bigg]\eta_2\\
&-2\tau\int_{{\O_{\mu,R}}}y\pt_t {G_\v}\bigg(\left|\frac{\n\psi}{y}\right|^2,\psi\bigg)\frac{\n\psi}{y}\bigg(\frac{\eta_2I}{y+\tau\eta_2}+\n\eta\bigg)\frac{(\n\psi)^T}{y}+o(\tau).
\end{align*}
Hence dividing $\tau$ on both sides of the above equality and letting $\tau\to0$ yield
\begin{align*}
0=&\int_{{\O_{\mu,R}}}y\bigg[{G_\v}\bigg(\left|\frac{\n\psi}{y}\right|^2,\psi\bigg)+\ld_\v^2\chi_{\{\psi<Q\}}\bigg]\n\c\eta\\
&+\int_{{\O_{\mu,R}}}\bigg[{G_\v}\bigg(\left|\frac{\n\psi}{y}\right|^2,\psi\bigg)
+\ld_\v^2\chi_{\{\psi<Q\}}\bigg]\eta_2\\
&-2\int_{{\O_{\mu,R}}}y\pt_t {G_\v}\bigg(\left|\frac{\n\psi}{y}\right|^2,\psi\bigg)\frac{\n\psi}{y}\bigg(\frac{\eta_2I}{y}+\n\eta\bigg)\frac{(\n\psi)^T}{y}.
\end{align*}
Note that
\begin{align*}
&y\bigg[{G_\v}\bigg(\left|\frac{\n\psi}{y}\right|^2,\psi\bigg)+\ld_\v^2\bigg]\n\c\eta\\
=&\n\cdot\bigg\{y\bigg[{G_\v}\bigg(\left|\frac{\n\psi}{y}\right|^2,\psi\bigg)+\ld_\v^2\bigg]\eta\bigg\}
-2y\pt_tG_\v\bigg(\left|\frac{\n\psi}{y}\right|^2,\psi\bigg)\frac{\n\psi}{y}\n\bigg(\frac{\n\psi}{y}\bigg)\eta\\
&-y\pt_zG_\v\bigg(\left|\frac{\n\psi}{y}\right|^2,\psi\bigg)\n\psi\c\eta
-\bigg[G_\v \bigg(\left|\frac{\n\psi}{y}\right|^2,\psi\bigg)+\ld_\v^2\bigg]\eta_2
\end{align*}
and
$$\frac{\n\psi}{y}\n\bigg(\frac{\n\psi}{y}\bigg)\eta+\frac{\n\psi}{y}\n\eta\frac{(\n\psi)^T}{y}
=\frac{\n\psi}{y}\c\n\bigg(\frac{\n\psi}{y}\c\eta\bigg)+\frac1{y^3}\left(\eta_1\pt_x\psi\pt_y\psi-\eta_2(\pt_x\psi)^2\right).$$
Using the divergence theorem and \eqref{EL equ} one has
\begin{align*}
0=&\lim_{s\to0+}\int_{{\O_{\mu,R}}\cap\{\psi<Q-s\}}\Bigg\{y\bigg[{G_\v}\bigg(\left|\frac{\n\psi}{y}\right|^2,\psi\bigg)+\ld_\v^2\bigg]\n\c\eta
+\bigg[{G_\v}\bigg(\left|\frac{\n\psi}{y}\right|^2,\psi\bigg)
+\ld_\v^2\bigg]\eta_2\\
&-2\pt_t {G_\v}\bigg(\left|\frac{\n\psi}{y}\right|^2,\psi\bigg)\left|\frac{\n\psi}{y}\right|^2\eta_2-2y\pt_t {G_\v}\bigg(\left|\frac{\n\psi}{y}\right|^2,\psi\bigg)\frac{\n\psi}{y}\n\eta\frac{(\n\psi)^T}{y}\Bigg\}\\
=&\lim_{s\to0+}\int_{\pt\{\psi<Q-s\}}y\bigg[G_\v\bigg(\left|\frac{\n\psi}{y}\right|^2,\psi\bigg)
-2\pt_t G_\v\bigg(\left|\frac{\n\psi}{y}\right|^2,\psi\bigg)\left|\frac{\n\psi}{y}\right|^2+\ld_\v^2\bigg](\eta\cdot\nu)\\
+&\lim_{s\to0+}\int_{{\O_{\mu,R}}\cap\{\psi<Q-s\}}\bigg[2\n\cdot\bigg(\pt_t G_\v\bigg(\left|\frac{\n\psi}{y}\right|^2,\psi\bigg)\frac{\n\psi}{y}\bigg)-y\pt_z G_\v\bigg(\left|\frac{\n\psi}{y}\right|^2,\psi\bigg)\bigg]\n\psi\cdot\eta\\
=&\lim_{s\to0+}\int_{\pt\{\psi<Q-s\}}y\bigg[-\Phi_\v\bigg(\left|\frac{\n\psi}{y}\right|^2,\psi\bigg)+\ld_\v^2\bigg](\eta\cdot\nu).
\end{align*}
\bx

\section{The existence and regularity for the free boundary problem}\label{sec existence and regularity}

In this section, we study the existence and regularity of minimizers for the minimization problem \eqref{variation problem}, as well as the regularity of the free boundary away from the nozzle.

\subsection{Existence of minimizers}\label{subsec_existence}
For the ease of notations in the rest of this section, let
\begin{equation}\label{notation JDG}
\D:=\O_{\mu,R}=\O\cap\{-\mu<x<R\},\quad
\MG(\bp,z):=G_\v(|\bp|^2,z),\q \J:=J^\v_{{\mu,R},\Ld}, \q  \ld:=\ld_\v.
\end{equation}
Then $\D$ is a bounded Lipschitz domain in $\R^2$ and is contained in the infinite strip $\mathbb R\times[0,\bar H]$, $\MG: \R^2\times \R\to \R$ is smooth in $\bp$ and $C^{1,1}$ in $z$, and $\ld$ is a positive constant. The problem \eqref{variation problem} can be rewritten as
\begin{equation}\label{minimization problem}
\text{find }\psi\in \K_{\psi^\sharp}  \, \text{ s.t. } \, \J(\psi)= \inf_{\phi\in \K_{\psi^\sharp}} \J(\phi)
\end{equation}
with
\begin{equation*}\label{J}
\J(\phi):=\int_\D y\left[\MG\left(\frac{\n\phi}{y},\phi\right)+\ld^2\chi_{\{\phi<Q\}}\right]dX,
\end{equation*}
and the admissible set
$$\K_{\psi^\sharp}:=\{\phi\in H^1(\D):\phi=\psi^\sharp \text{ on } \pt\D\}.$$
Here $\psi^\sharp:\overline{\D}\rightarrow \R$ is given, $\psi^\sharp\in H^1(\D)\cap C(\pt\D)$, and $0\leq \psi^\sharp\leq Q$ on $\pt\D$.

The properties of $\MG$ are summarized in the following proposition.

\begin{proposition}\label{Gproperties pro}
Let $G_\v$ be defined in \eqref{G def} and $\MG$ be defined in \eqref{notation JDG}. Then the following properties hold.
\begin{enumerate}
  \item[\rm (i)] There exist two positive constants $\mathfrak b_*=\mathfrak b_*(B_*,S_*,\bar H,\g)$ and $\mathfrak b^*=\mathfrak b^*(B_*,S_*,\bar H,\g,\v)$ such that
  \begin{align}
  &\mathfrak b_*|\bp|^2\leq p_i\pt_{p_i}\MG(\bp,z)\leq \mathfrak b_*^{-1}|\bp|^2,\label{G dp}\\
  &\mathfrak b_*|\mathbf{\xi}|^2\leq \mathbf{\xi}_i\pt_{p_ip_j}\MG(\bp,z)\mathbf{\xi}_j\leq \mathfrak b^*|\mathbf{\xi}|^2, \quad \text{for all } \mathbf{\xi}\in \R^2.\label{G dpp}
  \end{align}

  \item[\rm (ii)] One has
    \begin{equation}\label{G support}
  	\begin{split}
  	\pt_z\MG(\bp,z)=0 \text{ in } \R^2\times(-\infty,0],\q
  	\pt_z\MG(\bp,z)\geq 0  \text{ in } \R^2\times [Q,\infty).
  	\end{split}
  \end{equation}
   \item[\rm (iii)] There exists a constant
  \begin{equation}\label{delta}
  	\d:=C\v^{-1}\k_0(\k_0+1)\leq C_\d,
  \end{equation}
  where $\k_0$ is defined in \eqref{k0}, $C=C(B_*,S_*,\bar H,\g)$, and $C_\d=C_\d(B_*,S_*,\bar H,\g,\v)$, such that
  \begin{align}
  &\v^{-1}|\pt_z\MG(\bp,z)|+|\bp\c\pt_{{\bp}z}\MG(\bp,z)|+|\pt_{\bp z}\MG(\bp,z)|+|\pt_{zz}\MG(\bp,z)|\leq \d,\label{G dzdzpdzz}\\
  & \MG(\mathbf 0,Q)=0,\q \MG(\bp,z)\geq \frac{\mathfrak b_*}2|\bp|^2-\delta\min\{Q,(Q-z)_+\}.\label{G geq}
  \end{align}
\end{enumerate}
\end{proposition}

\begin{proof}
\emph{(i).} In view of the definition of $G_\v$ in \eqref{G def}, straightforward computations yield
\begin{align*}
&p_i\pt_{p_i}\MG(\bp,z)=g_\v(|\bp|^2,z)|\bp|^2,\\
&\pt_{p_ip_j}\MG(\bp,z)=g_\v(|\bp|^2,z)\d_{ij}+2\pt_tg_\v(|\bp|^2,z)p_ip_j.
\end{align*}
Thus one gets \eqref{G dp} and \eqref{G dpp} from \eqref{gm bound} and \eqref{g dtgmt}.

\emph{(ii).} Note that
$$g_\v(t,z)=g(t,z) \quad\text{in } \{(t,z): 0\leq t\leq \t_c(z)-\v,z\in\R\}.$$
By the definition of $\MG$ one has
\begin{equation}\label{dzG_express1}\begin{split}
\pt_z\MG(\bp,z)&=\frac12\int_0^{|\bp|^2}\pt_zg_\v(\tau,z)d\tau-(\g-1)g^{-\g-1}(0,z)\pt_zg(0,z)\S(z)+\frac{\g-1}{\g}g^{-\g}(0,z)\S'(z)\\
&\stackrel{\eqref{II}}{=}\frac12\int_0^{|\bp|^2}\pt_zg_\v(\tau,z)d\tau+\frac{\B'(z)}{g(0,z)}-\frac{\S'(z)}{\g g^\g(0,z)}.
\end{split}\end{equation}
This together with the explicit expression for $\pt_zg_\epsilon$ in \eqref{dzg express} and an integration by parts yields that
\begin{equation}\label{dzGm expression}
	\begin{split}
		\pt_z \mathcal G(\bp,z)=&-\int_0^{|\bp|^2}\left(
		\frac{\mathcal{B}'(z)}{g(\tau,z)}-\frac{\mathcal{S}'(z)}{\g g(\tau,z)^\g}-\frac{1}2
		(g^\ast - g(\tau,z)) \t_c'(z) \right)\varpi'_\epsilon(\tau-\t_c(z)) \ d\tau\\
		&+\left(\frac{\mathcal{B}'(z)}{g(|\bp|^2, z)}-\frac{\mathcal{S}'(z)}{\g g(|\bp|^2, z)^\g}\right)\varpi_\epsilon(|\bp|^2-\t_c(z)).
	\end{split}
\end{equation}
Note that $\mathcal B'(z)=\mathcal S'(z)=0$ on $(-\infty, 0]$ and $\B'(z)\geq0$, $\S'(z)=0$ on $[Q,\infty)$, cf. \eqref{eq:sign_B}.
If we can show
\begin{align}\label{label_8}
	g(\tau,z)(g^\ast - g(\tau,z)) \left(\frac{2\B(z)}{(\g+1)\S(z)}\right)^{\frac2{\g-1}}\leq 1
\end{align}
in the set $\mathcal R:=\{(\tau,z):\tau\in [0,\t_c(z)),\, z\in(-\infty,0]\cup[Q,\infty)\}$, then it follows from the expression of $\t_c'$ in \eqref{tc_deri}, $\varpi'_\v\leq 0$, $\varpi_\v\geq 0$, and $g\geq 0$ that  $\pt_z \mathcal G(\bp, z)$ is a positive multiple of $\mathcal{B}'(z)$ in $\mathcal R$.
In view of \eqref{eq:sign_B}, we consequently conclude that \eqref{G support} holds true.
Now it remains to prove \eqref{label_8}. In fact, by \eqref{g bound} and \eqref{rho cm} one has
\begin{align}\label{label_12}
	0\leq 4g(\tau,z)(g^\ast-g(\tau,z))\leq (g^\ast)^2 = \left(\inf_{z\in\mathbb R}\frac{2\B(z)}{(\g+1)\S(z)}\right)^{-\frac2{\g-1}}.
\end{align}
Since the constant $\k$ in \eqref{BS condition2} can be sufficiently small depending on $B_*$, $S_*$, $\bar H$, and $\g$, using \eqref{label_12} and the bounds of $\B$ and $\S$ in \eqref{eq:BS_bounds}, one obtains \eqref{label_8}. This finishes the proof of (ii).

\emph{(iii).} It follows from the upper bound of $\k_0$ in \eqref{k0} (where $\k$ is sufficiently small) that the constant $\d$ defined in \eqref{delta} satisfies $\d\leq C_\d$ for some $C_\d=C_\d(B_*,S_*,\bar H,\g,\v)$.

\emph{Proof for \eqref{G dzdzpdzz}.}
Combining the expression of $\pt_z\MG$ in \eqref{dzG_express1}, the inequality \eqref{dzgm_dtgm}, the estimate for $g_\v$ in \eqref{gm bound}, and the definition of $\k_0$ in \eqref{k0} together yields the estimate for the first term in \eqref{G dzdzpdzz}.

A direct differentiation of $\pt_z\mathcal G$ gives
\begin{align*}
	\pt_{p_iz}\mathcal G(\bp,z) = p_i \pt_zg_\v (|\bp|^2,z).
\end{align*}
Note that $\pt_zg_\v(t,z)=0$ for $t\geq \t_c(z)-\frac{\v}{2}$.
Thus using \eqref{dzgm} one has
$$|\bp\c\pt_{{\bp}z}\mathcal G(\bp,z)|+|\pt_{\bp z}\mathcal G(\bp,z)|\leq  C^*\v^{-1} \kappa_0 (t^*+\sqrt{t^*})
\leq C\v^{-1}\k_0,$$
where $t^*$ is the upper bound of $\t_c(z)$ in \eqref{tc_bound} and $C=C(B_*,S_*,\bar H,\g)$. This gives the estimate for the second and third terms in \eqref{G dzdzpdzz}.

In the end, by the expression for $\pt_z\MG$ in \eqref{dzG_express1} one has
\begin{equation}\label{dzzG}\begin{split}
		\pt_{zz}\MG(\bp,z)=&\frac12\int_0^{|\bp|^2}\pt_{zz}g_\v(\tau,z)d\tau+\frac{\B''(z)}{g(0,z)}-\frac{\B'(z)}{g^2(0,z)}\pt_zg(0,z)\\
		&-\frac{\S''(z)}{\g g^\g(0,z)}+\frac{\S'(z)}{g^{\g+1}(0,z)}\pt_zg(0,z).
\end{split}\end{equation}
To estimate the first term, we infer from \eqref{dzg express}
that
\begin{align*}
\pt_{zz}g_\v(\tau,z) =&2\left(\mathcal{B}''(z)\pt_t\left(\frac{1}{g(\tau,z)}\right) + \mathcal{B}'(z) \pt_{zt} \left(\frac{1}{g(\tau,z)}\right) \right) \varpi_\v\left(\tau-\t_c(z)\right) \\
&-\frac2{\gamma}\left(\mathcal{S}''(z)\pt_t\left(\frac{1}{g^\gamma(\tau,z)}\right) + \mathcal{S}'(z) \pt_{zt} \left(\frac{1}{g^\gamma(\tau,z)}\right) \right) \varpi_\v\left(\tau-\t_c(z)\right) \\
&+  2\pt_z g(\tau,z)\frac{d}{dz}\varpi\left(\tau-\t_c(z)\right) +(g(\tau,z)-g^\ast)\frac{d^2}{dz^2}\varpi\left(\tau-\t_c(z)\right).
\end{align*}
We plug it into the expression of $\pt_{zz}\MG$. For the term involving $\pt_{zt}(\frac{1}{g(\tau,z)})$, an integration by parts yields
\begin{align*}
	&\int_0^{|\bp|^2} \pt_{zt}\left(\frac{1}{g(\tau,z)}\right)\varpi_\epsilon\left(\tau-\t_c(z)\right) \ d\tau\\ =& -\frac{\pt_zg}{g^2}(|\bp|^2, z)\varpi_\epsilon \left(|\bp|^2-\t_c(z)\right)+\frac{\pt_zg}{g^2}(0,z)
	+\int_0^{|\bp|^2}\frac{\pt_zg}{g^2}(\tau,z) \varpi'_\epsilon(\tau-\t_c(z))\ d\tau.
\end{align*}
The term involving $\pt_{zt}(\frac1{g^{\gamma}(\tau,z)})$ can be  estimated similarly as above. Note that
$$|\varpi''_\epsilon(\tau-\t_c(z))|\leq 8\epsilon^{-2}\quad \text{and}\quad |{\rm supp}(\varpi''(\tau-\t_c(z)))|\leq \v.$$
Thus we infer \eqref{k0}, \eqref{dzgm} and \eqref{g bound} that
\begin{align*}
	|\pt_{zz}\MG|\leq C\v^{-1}\k_0(\k_0+1),
\end{align*}
where $C=C(B_*,S_*,\bar H,\gamma)$.
Thus we obtain the estimate for the last term in \eqref{G dzdzpdzz}.

\emph{Proof for \eqref{G geq}.}
In view of \eqref{G def}, one has $\mathcal G(\mathbf0,Q)=0$.
The second inequality in \eqref{G geq} follows directly from the strong convexity in $\bp$ and the estimates of derivatives in $z$. More precisely, the strong convexity property \eqref{G dpp} gives that
\begin{align*}
	\mathcal G(\bp,z)\geq \mathcal G(\mathbf 0,z)+ \frac{\mathfrak b_\ast}{2}|\bp|^2,
\end{align*}
where we have also used that $\nabla_{\mathbf p} \mathcal G(\mathbf 0,z)=0$ (since  $\nabla_{\mathbf p} \mathcal G(\bp,z)=2{\mathbf p}\pt_tg_\v(|\bp|^2,z)$). To estimate $\mathcal G(\mathbf0,z)$ we use a first order Taylor expansion at $z=Q$ together with the estimate of $|\pt_z\MG|$ in \eqref{G dzdzpdzz}. Note that $\v<1/4$. Then for $z\in [0,Q]$, one has
\begin{align*}
	\mathcal G(\mathbf0,z)\geq -\delta(Q-z).
\end{align*}
Moreover, in view of \eqref{G support}, one has
$$\mathcal G(\mathbf 0,z)=\mathcal G(\mathbf0,0)\geq -\delta Q \ \text{ for }z<0,\q\mathcal G(\mathbf 0,z)\geq\mathcal G(\mathbf 0,Q)=0  \ \text{ for } z>Q.$$ Combining the above estimates together we obtain the second inequality of \eqref{G geq}.
This completes the proof of the proposition.
\end{proof}

With Proposition \ref{Gproperties pro} at hand, the existence of minimizers for the minimization problem \eqref{minimization problem} follows from standard theory for calculus of variations.

\begin{lemma}\label{minimizer existence}
Assume $\MG$ satisfies \eqref{G dpp} and \eqref{G geq}. Then the minimization problem \eqref{minimization problem} has a minimizer.
\end{lemma}
\begin{proof}
Obviously, there exists a function $\tilde\psi\in \K_{\psi^\sharp}$ such that  $\J(\tilde\psi)<\infty$. In view of \eqref{G geq}, one can find a minimizing sequence $\{\psi_n\}$ such that
$$\lim_{n\to\infty}\J(\psi_n)=\inf_{\phi\in \K_{\psi^\sharp}}\J(\phi)$$
and
$$\frac{\mathfrak b_*}{2\bar H}\int_\D|\n\psi_n|^2\leq \frac{\mathfrak b_*}2\int_\D\frac{|\n\psi_n|^2}{y}\leq \J(\psi_n)+\d Q\bar H|\D|.$$
Then it follows from the Poincar\'{e} inequality that $\psi_n-\psi^\sharp$ is uniformly bounded in $H^{1}(\D)$. Thus there exist a subsequence (relabeled) and a function $\psi\in H^{1}(\D)$ such that
$$\psi_n\to\psi  \ \text{ in }L^2(\D)\cap L^2(\pt\D),\q \psi_n\to\psi \ \text{ a.e. in } \D, \q \text{and}\q \n\psi_n\rightharpoonup\n\psi \  \text{ in } L^2(\D).$$
Note that $\MG$ is convex with respect to $\bp$. Using Fatou's lemma yields
\begin{equation}\begin{split}
\J(\psi)&=\int_{\D}y\left[\MG\left(\frac{\n\psi}{y},\psi\right)+\ld^2\chi_{\{\psi<Q\}}\right]\\
&\leq\liminf_{n\to\infty}\int_\D y\left[\MG\left(\frac{\n\psi_n}{y},\psi_n\right)+\ld^2\chi_{\{\psi_n<Q\}}\right]\\
&=\liminf_{n\to\infty}\J(\psi_n),
\end{split}\end{equation}
Therefore $\psi\in\K_{\psi^\sharp}$ is a minimizer.
\end{proof}

\subsection{$L^\infty$-estimate and H\"{o}lder regularity}
In this subsection we establish the  $L^{\infty}$-bound and the H\"{o}lder regularity for a minimizer.
Firstly, one can show that a minimizer of the minimization problem \eqref{minimization problem} is a supersolution of the elliptic equation
\begin{equation}\label{elliptic equG}
\pt_i\pt_{p_i}\MG\left(\frac{\n\psi}{y},\psi\right)-y\pt_z\MG\left(\frac{\n\psi}{y},\psi\right)=0.
\end{equation}

\begin{lemma}\label{supersol lem}
Let $\psi$ be a minimizer of \eqref{minimization problem}. Then $\psi$ satisfies
\begin{equation}\label{supersolution}
\int_{\D}\left[\pt_{p_i}\MG\left(\frac{\n\psi}{y},\psi\right)\pt_i\eta+y\pt_z\MG\left(\frac{\n\psi}{y},\psi\right)\eta\right]\geq0,
\q\text{for all } \eta\geq0, \ \eta\in C_0^\infty(\D).
\end{equation}
\end{lemma}
\begin{proof}
For any $\tau>0$, since $\{\psi+\tau\eta<Q\}\subset \{\psi<Q\}$, it holds
\begin{align*}
0\leq&\frac 1{\tau}(\J(\psi+\tau\eta)-\J(\psi))\\
\leq&\frac1\tau\int_{\D}y\left[\MG\left(\frac{\n(\psi+\tau\eta)}{y},\psi+\tau\eta\right)-\MG\left(\frac{\n\psi}{y},\psi\right)\right]\\
=&\frac1\tau\int_{\D}y\left[\MG\left(\frac{\n(\psi+\tau\eta)}{y},\psi+\tau\eta\right)-\MG\left(\frac{\n\psi}{y},\psi+\tau\eta\right)\right]\\
&+\frac1\tau\int_{\D}y\left[\MG\left(\frac{\n\psi}{y},\psi+\tau\eta\right)-\MG\left(\frac{\n\psi}{y},\psi\right)\right].
\end{align*}
Letting $\tau\to0$ yields \eqref{supersolution}.
\end{proof}

Next, one has the following $L^\infty$-estimate for a minimizer $\psi$.

\begin{lemma}\label{lem_minimizer_bound}
Let $\psi$ be a minimizer of \eqref{minimization problem}, then
$$0\leq \psi\leq Q.$$
\end{lemma}
\begin{proof}
Set $\psi^\tau=\psi+\tau\min\{0,Q-\psi\}$ for $\tau>0$. Since $\psi^\sharp\leq Q$ on $\pt\D$, one has $\psi^\tau=\psi^\sharp$ on $\pt\D$ and thus $\psi^\tau\in \K_{\psi^\sharp}$. Since $\{\psi^\tau<Q\}=\{\psi<Q\}$ and $\MG$ is convex with respect to $\bp$, using the same arguments as in Lemma \ref{supersol lem} gives
\begin{align*}
0&\leq \frac1{\tau}(\J(\psi^\tau)-\J(\psi))\\
&\leq \int_{\D} \pt_{p_i}\MG
\left(\frac{\n \psi}{y},\psi^\tau\right)\pt_i(\min\{0,Q-\psi\})+\frac1\tau\int_{\D}y\left[\MG\left(\frac{\n\psi}{y},\psi^{\tau}\right)-\MG\left(\frac{\n\psi}{y},\psi\right)\right].
\end{align*}
Taking $\tau\to0$, by \eqref{G support} and \eqref{G dp} one gets
\begin{align*}
0&\leq\int_{\D\cap\{\psi>Q\}}-\pt_{p_i}\MG\left(\frac{\n\psi}{y},\psi\right)\pt_i\psi+\int_{\D\cap\{\psi>Q\}}y\pt_z\MG\left(\frac{\n\psi}{y},\psi\right)(Q-\psi)\\
&\leq\int_{\D\cap\{\psi>Q\}}-\pt_{p_i}\MG\left(\frac{\n\psi}{y},\psi\right)\pt_i\psi
\leq-\mathfrak b_*\int_{\D\cap\{\psi>Q\}}\frac{|\n\psi|^2}{y},
\end{align*}
This implies $\psi\leq Q$.

Similarly the lower bound can be obtained by setting $\tilde\psi^\tau=\psi-\tau\min(0,\psi)$ with $\tau>0$. This completes the proof of the lemma.
\end{proof}

Now we prove a comparison principle for the associated elliptic equation \eqref{elliptic equG}.

\begin{lemma}\label{comparison principle}
	Given a domain $\tilde\D\subset \R\times (a,a+\tilde h)$ for some $a\in[0,\bar H]$ and $\tilde h>0$. Let $\psi\in H^1(\tilde\D)$ be a supersolution of the equation \eqref{elliptic equG} in the sense of \eqref{supersolution} and $\phi\in H^1(\tilde\D)$ be a solution of \eqref{elliptic equG} in the sense of
	\begin{equation}\label{solution}
		\int_{\tilde\D}\left[\pt_{p_i}\MG\left(\frac{\n\phi}{y},\phi\right)\pt_i\eta+y\pt_z\MG\left(\frac{\n\phi}{y},\phi\right)\eta\right]=0,\quad \text{for all } \eta \in C_0^\infty(\tilde\D).
	\end{equation}
	Assume that $\psi\geq\phi$ on $\pt \tilde\D$. Then $\psi\geq\phi$ in $\tilde\D$ as long as $\tilde h\leq h_0$ for some  constant $h_0>0$ depending on $B_*,\, S_*,\, \bar H,\, \g$, and $\v$.
\end{lemma}

\begin{proof}
	Take $\eta=(\phi-\psi)^+$ as the test function in \eqref{supersolution} and \eqref{solution}, then
	$$\int_{\tilde\D}\left[\pt_{p_i}\MG\left(\frac{\n\phi}{y},\phi\right)-\pt_{p_i}\MG\left(\frac{\n\psi}{y},\psi\right)\right]\pt_i\eta
	+\int_{\tilde\D}y\left[\pt_z\MG\left(\frac{\n\phi}{y},\phi\right)-\pt_z\MG\left(\frac{\n\psi}{y},\psi\right)\right]\eta\leq 0.$$
	Using the convexity of $\MG$ in \eqref{G dpp} and the triangle inequality yields
	\begin{align*}
		&\int_{\tilde\D} \left[\pt_{p_i}\MG\left(\frac{\n\phi}{y},\phi\right)-\pt_{p_i}\MG\left(\frac{\n\psi}{y},\psi\right)\right]\pt_i\eta\\
		\geq& \mathfrak b_*\int_{\tilde\D}\frac{|\n\eta|^2}{y}+\int_{\tilde\D}\left[\pt_{p_i}\MG\left(\frac{\n\psi}{y},\phi\right)-\pt_{p_i}\MG\left(\frac{\n\psi}{y},\psi\right)\right]\pt_i\eta,
	\end{align*}
	where the last integral in the above inequality can be estimated from \eqref{G dzdzpdzz} as
	\begin{align*}
		&\int_{\tilde\D}\left[\pt_{p_i}\MG\left(\frac{\n\psi}{y},\phi\right)-\pt_{p_i}\MG\left(\frac{\n\psi}{y},\psi\right)\right]\pt_i\eta\\
		=&\int_{\tilde\D}\left[(\phi-\psi)\int_0^1\pt_{p_iz}\MG\left(\frac{\n\psi}{y},s\phi+(1-s)\psi\right)ds\right]\pt_i\eta
		\geq-\d\int_{\tilde\D}|\n\eta|\eta.
	\end{align*}
	Similarly, using \eqref{G dzdzpdzz} and the triangle inequality yields
	$$\int_{\tilde\D}y\left[\pt_z\MG\left(\frac{\n\phi}{y},\phi\right)-\pt_z\MG\left(\frac{\n\psi}{y},\psi\right)\right]\eta
	\geq -\d\int_{\tilde\D}(|\n\eta|\eta+y\eta^2).$$
	Combining the above estimates together gives
	$$\mathfrak b_*\int_{\tilde\D}\frac{|\n\eta|^2}{y}\leq 2\d\int_{\tilde\D}|\n\eta|\eta+\d\int_{\tilde\D}y\eta^2.$$
	Since the domain $\tilde\D$ is contained in the finite strip $\R\times(a,a+\tilde h)$ where $a\in[0,\bar H]$ and $\tilde h>0$,
	applying the Cauchy-Schwarz inequality and the Poincar\'{e} inequality to $\eta(x,\cdot)$ for each $x$, one has
	\begin{align*}
	\int_{\tilde\D}|\n\eta|^2
	&\leq \frac{\bar H+\tilde h}{\mathfrak b_*}\left[\frac{\mathfrak b_*}{2(\bar H+\tilde h)}\int_{\tilde\D}|\n\eta|^2+\frac{2\d^2(\bar H+\tilde h)}{\mathfrak b_*}\int_{\tilde D}\eta^2+\d(\bar H+\tilde h)\int_{\tilde\D}\eta^2\right]\\
	&\leq \frac{1}2\int_{\tilde\D}|\n\eta|^2+C(\bar H+\tilde h)^2\left(\frac{2\d}{\mathfrak b_*}+1\right)\frac{\d}{\mathfrak b_*}\tilde h^2\int_{\tilde\D}|\n\eta|^2,
	\end{align*}
   where $C>0$ is a universal constant. Note that $\d/\mathfrak b_*\leq C(B_*,S_*,\bar H,\g,\v)$ by the definitions of $\mathfrak b_*$ and $\d$ in Proposition \ref{Gproperties pro}.
   Thus if $\tilde h\leq h_0$ for some $h_0=h_0(B_*,S_*,\bar H,\g,\v)$ sufficiently small, then $\n\eta=0$ in $\tilde\D$, which implies that $\eta=0$ in $\tilde\D$. This completes the proof of the lemma.
\end{proof}
\begin{remark}\label{rmk_CP}
	It follows from the proof of Lemma \ref{comparison principle} that if the $L^{\infty}$-bounds of $|\pt_{\bp z}\MG|$ and $|\pt_{zz}\MG|$ are sufficiently small depending on $B_*$, $S_*$, $\bar H$, $\g$, and $\tilde h$, then the comparison principle still holds in the domain $\tilde D$ even if the width $\tilde h$ is large.
\end{remark}

Now we are in position to prove the H\"{o}lder regularity for minimizers. The proof follows from standard Morrey type estimates.

\begin{lemma}\label{psi holder}
Let $\psi$ be a minimizer of \eqref{minimization problem}, then $\psi\in C_{{\rm loc}}^{0,\a}(\D)$ for any $\a\in(0,1)$. Moreover,
$$\|\psi\|_{C^{0,\a}({\D'})}\leq C(\mathfrak b_*,\d, \D', Q,\ld,\a) \q\text{ for any } {\D'}\Subset \D.$$
\end{lemma}
\begin{proof}
Given any $B_r\subset \D'\Subset \D$ with $r$ sufficiently small such that the comparison principle (cf. Lemma \ref{comparison principle}) holds in $B_r$. Let $\phi\in H^1(B_r)$ be a minimizer for the energy functional $\int_{B_r}y\MG(\frac{\n\phi}{y},\phi)$ among $\phi\in H^1(B_r)$ with boundary trace $\phi=\psi$ on $\pt B_r$.
Then $\phi$ is a solution of the following problem
\begin{equation}\label{homosol}
	\begin{cases}
		\pt_i\pt_{p_i}\MG\left(\frac{\n\phi}{y},\phi\right)-y\pt_z\MG\left(\frac{\n\phi}{y},\phi\right)=0 &\text{in } B_r\\
		\phi=\psi &\text{on } \pt B_r.\end{cases}
\end{equation}
By the comparison principle one has $0\leq \phi\leq Q$ in $B_r$. Moreover, it follows from the interior regularity for elliptic PDE (cf. for example \cite{HL_book}) that
\[\|\nabla\phi\|_{L^2(B_{\hat r})}\leq C \frac{\hat r}{r}\|\nabla \phi\|_{L^2(B_{r})}+ C r^2\quad \text{for any}\,\, 0< \hat r\leq r,\]
for some $C=C(\mathfrak b_*, \delta,\D')$. Define
	\[
	\tilde{\psi}:=\left\{
	\begin{aligned}
		& \phi\quad && \text{in } B_r,\\
		& \psi\quad && \text{in }\D \setminus B_r.
	\end{aligned}\right.
	\]
	Clearly, $\tilde{\psi}\in H^1(\D)$.
	Since $\psi$ is a minimizer of \eqref{minimization problem},
	it holds
	\begin{equation}\label{eq1}
	\int_{B_r}y\left[\MG
	\left(\frac{\n\psi}{y},\psi\right)-\MG\left(\frac{\n\phi}{y},\phi\right)\right]\leq \ld^2\int_{B_r}y(\chi_{\{\phi<Q\}}-\chi_{\{\psi<Q\}}).	
	\end{equation}
By the Taylor expansion, \eqref{G dpp} and \eqref{G dzdzpdzz}, one has
\begin{align*}
	&\int_{B_r}y\left[\MG\left(\frac{\n\psi}{y},\psi\right)-\MG\left(\frac{\n\phi}{y},\phi\right)\right]\\
	\geq&\int_{B_r}\left[\pt_{p_i}\MG\left(\frac{\n\phi}{y},\phi\right)\pt_i(\psi-\phi)+y\pt_{z}\MG\left(\frac{\n\phi}{y},\phi\right)(\psi-\phi)\right]\\
	&+\int_{B_r}\frac y2\left(\mathfrak b_*\left|\frac{\nabla\psi-\n\phi}y\right|^2-\d\left|\frac{\n\psi-\n\phi}{y}\right||\psi-\phi|-\d|\psi-\phi|^2\right).
\end{align*}
Then using \eqref{homosol} and the Cauchy-Schwarz inequality to the right-hand side of the above inequality gives
\begin{align}\label{eq2}
	\int_{B_r}y\left[\MG\left(\frac{\n\psi}{y},\psi\right)-\MG\left(\frac{\n\phi}{y},\phi\right)\right]
	\geq&\frac{\mathfrak b_*}4\int_{B_r}\frac{|\nabla\psi-\n\phi|^2}y-C\int_{B_r}y|\psi-\phi|^2,
\end{align}
where $C=C(\mathfrak b_*,\d)$. Combining \eqref{eq1} and \eqref{eq2} together, 
one has
$$\int_{B_r}|\n\psi-\n\phi|^2\leq C(\mathfrak b_*,\d,\D')\int_{B_r}(|\psi-\phi|^2+\lambda^2\chi_{\{\psi=Q\}})
\leq C(\mathfrak b_*,\d,\D', Q, \ld)r^2.$$
The desired H\"{o}lder regularity then follows from the standard Morrey type estimates (cf. \cite[the proof for Theorem 3.8]{HL_book}).
\end{proof}

\subsection{Lipschitz regularity and nondegeneracy}
In this subsection we establish the (optimal) Lipschitz regularity and the nondegeneracy property of minimizers, which play an important role to get the measure theoretic properties of the free boundary.

Consider the rescaled function
\begin{equation}\label{def:psi*}
\psi^*_{\bar X,r}(X):=\frac{Q-\psi(\bar X+rX)}{r},\q \text{for} \  \bar X\in\D, \ r\in(0,1).
\end{equation}
If $\psi$ is a minimizer of \eqref{minimization problem}, then $\psi^*_{\bar X,r}$ is a minimizer of
\begin{equation*}\label{J scale}
\J_{\bar X,r}(\phi):=\int_{\D_{\bar X,r}}(\bar y+ry)\left[\tilde\MG_r\left(\frac{\n\phi}{\bar y+ry},\phi\right)
+\ld^2\chi_{\{\phi>0\}}\right]dX
\end{equation*}
over the admissible set
$$\K_{\bar X,r,\psi^{\sharp}}:=\left\{\phi\in H^1(\D_{\bar X,r}):\phi(X)=(Q-\psi^\sharp(\bar X+rX))/r \text{ on } \pt\D_{\bar X,r}\right\},$$
where
\begin{equation}\label{Gsacle}
\tilde\MG_r(\bp,z):=\MG(-\bp,Q-rz) \q\text{and}\q
\D_{\bar X,r}:=\left\{(X-\bar X)/r: X\in\D\right\}.
\end{equation}
Moreover, $\psi^*_{\bar X,r}$ satisfies
\begin{equation}\label{ellipticsub scale}
\pt_i\pt_{p_i}\tilde\MG_r\left(\frac{\n\psi^*_{\bar X,r}}{\bar y+ry},\psi^*_{\bar X,r}\right)-(\bar y+ry)\pt_z\tilde\MG_r\left(\frac{\n\psi^*_{\bar X,r}}{\bar y+ry},\psi^*_{\bar X,r}\right)\geq0 \ \ \text{in }\D_{\bar X,r},
\end{equation}
and
\begin{equation}\label{ellipticequ scale}
	\pt_i\pt_{p_i}\tilde\MG_r\left(\frac{\n\psi^*_{\bar X,r}}{\bar y+ry},\psi^*_{\bar X,r}\right)-(\bar y+ry)\pt_z\tilde\MG_r\left(\frac{\n\psi^*_{\bar X,r}}{\bar y+ry},\psi^*_{\bar X,r}\right)=0 \ \text{ in } \D_{\bar X,r}\cap\{\psi^*_{\bar X,r}>0\}.
\end{equation}
It follows from Proposition \ref{Gproperties pro}
that the function $\tilde\MG_r$ satisfies the following properties:
  \begin{align}
	&\mathfrak b_\ast|\bp|^2 \leq p_i \pt_{p_i} \tilde{\mathcal G}_r (\bp, z)\leq \mathfrak b_\ast^{-1} |\bp|^2,\label{Gr_dp}\\
	&\mathfrak b_\ast |\xi|^2 \leq \xi_i \pt_{p_ip_j}\tilde{\mathcal G}_r(\bp,z)\xi_j\leq \mathfrak b^\ast|\xi|^2, \quad \text{for all } \mathbf{\xi}\in \R^2,\label{Gr_dpp}\\
	&\pt_z\tilde\MG_r(\bp,z)=0  \text{ in } \R^2\times[Q/r,\infty), \q \pt_z\tilde\MG_r(\bp,z)\leq0 \text{ in } \R^2\times (-\infty,0], \label{label_1}\\
& \v^{-1}|\pt_z \tilde\MG_{r}(\bp,z)|
+|\pt_{\bp z}\tilde\MG_{r}(\bp,z)|\leq r\delta\leq rC_\d, \q |\pt_{zz}\tilde\MG_{r}(\bp,z)|\leq r^2\delta,\label{eq_rescale3}\\
&\tilde\MG_{r}(0,0)=0,\q\tilde\MG_{r}(\bp,z)\geq\frac{\mathfrak b_*}2|\bp|^2-\d rz_+,\label{eq_rescale2}
\end{align}
where the constants $\mathfrak b_*=\mathfrak b_*(B_*,S_*,\bar H,\g)$,  $\mathfrak b^*=\mathfrak b^*(B_*,S_*,\bar H,\g,\v)$, $\d$, and $C_\d=C_\d(B_*, S_*,\bar H, \g,\v)$ are the same as in Proposition \ref{Gproperties pro}. Besides, using $\pt_z\tilde\MG_r(\bp,0)\leq0$ and the second estimate in \eqref{eq_rescale3} one has
\begin{align}
	& \pt_z\tilde\MG_r (\bp, z)\leq  r^2 \delta z_+. \label{eq_rescale1}
\end{align}
Furthermore, similar arguments as in Lemma \ref{comparison principle} (see also Remark \ref{rmk_CP}) give that, there exists a constant
\begin{equation}\label{r_comparison_principle}
	r^*:=r^*(B_*, S_*,\bar H, \g,\v)>0
\end{equation}
such that solutions to the rescaled equation \eqref{ellipticequ scale} enjoy the comparison principle in $\D_{\bar X,r}\cap B_4$ as well  provided $r\leq r^*$.

\begin{lemma}\label{linear growth}
Let $\psi$ be a minimizer of \eqref{minimization problem}. Let $\bar{X}=(\bar x,\bar y)\in\D\cap\{\psi<Q\}$ satisfy
${\rm dist}(\bar{X},\G_{\psi})\leq \bar r\min\{\ld,1\}{\rm dist}(\bar{X},\pt\D)$, where $\bar r=\bar r(B_*, S_*,\bar H, \g,\v)\in(0,1/4)$. Then there exists a constant  $C=C(B_*, S_*,\bar H, \g,\v)>0$ such that
$$Q-\psi(\bar{X})\leq C\ld\bar y {\rm dist}(\bar{X},\G_{\psi}).$$
\end{lemma}

\begin{proof}
Let $\psi_{\bar X,r}^*(X)$ be defined in \eqref{def:psi*} with $r:={\rm dist}(\bar{X},\G_{\psi})\in(0,\bar r{\rm dist}(\bar X,\pt \D))$, where $\bar r>0$ is a small constant to be determined later. Then $\psi_{\bar X,r}^*$ satisfies \eqref{ellipticsub scale}-\eqref{ellipticequ scale}. Suppose
\begin{align}\label{m0}
	m_0:=\psi^*_{\bar X,r}(0)\geq C_{m}\ld\bar y
\end{align}
for a sufficiently large constant
$C_{m}=C_{m}(B_*, S_*,\bar H, \g,\v)>0$.
If we can derive a contradiction, then the desired conclusion follows.
For notational convenience, we omit the subscript $\bar X$ and $r$ of $\psi^*_{\bar X,r}$ in the following proof.

\emph{Step 1. Construction of the barrier function.}  For any given ${X_0}\in \G_{\psi^*}\cap\overline{B}_1$, let $\phi$ be the solution of
\begin{equation}\label{barrier}
\begin{cases}
\pt_i\pt_{p_i}\tilde\MG_r\left(\frac{\n\phi}{\bar y+ry},\phi\right)-(\bar y+ry)\pt_z\tilde \MG_r\left(\frac{\n\phi}{\bar y+ry},\phi\right)=0 & \text{in } B_{2}({X_0}),\\
\phi=\psi^* & \text{on } \pt B_{2}({X_0}).
\end{cases}\end{equation}
In view of \eqref{ellipticsub scale}, the function $\phi\geq\psi^*$ in $B_{2}({X_0})$ provided $r\leq r^*\bar y/\bar H\leq r^*$, where  $r^*$ is given by \eqref{r_comparison_principle}.
In particular, $\phi(0)\geq \psi^*(0)=m_0$. Thus it follows from \eqref{eq_rescale3}, \eqref{m0} and the Harnack inequality (cf. \cite[Theorems 8.17 and 8.18]{GT_book}) that there exists a constant $C_0=C_0(B_*, S_*,\bar H, \g,\v)>0$ such that
\begin{equation}\label{barrier1}
\phi(X)\geq C_0m_0 \q \text{for any } X\in B_{\frac12}({X_0}),
\end{equation}
provided $r\leq r_1\ld\bar y$ for some $r_1=r_1(B_*, S_*,\bar H, \g,\v)>0$.
We claim that if $r\leq r_1\bar y$ for a possibly smaller $r_1$ with the same parameter dependence, then there exists a constant $\bar C<1/2$ such that
\begin{equation}\label{barrier2}
\phi(X)\geq \bar C C_0m_0(2-|X-{X_0}|) \q \text{for any } X\in B_{2}({X_0}).
\end{equation}
Indeed, by \eqref{barrier1}, it suffices to prove \eqref{barrier2} in $B_{2}({X_0})\backslash B_{\frac12}({X_0})$. For any fixed large constant $\tau>4$, set
$$\Phi_0(X):=C_0m_0\left(e^{-\tau|X-{X_0}|^2}-e^{-4\tau}\right).$$
Then for $X\in B_{2}(X_0)\backslash B_{\frac12}(X_0)$, it follows from \eqref{Gr_dpp}, \eqref{eq_rescale3}, and \eqref{eq_rescale1} that
\begin{align*}
&\pt_i\pt_{p_i}\tilde{\MG}_r\left(\frac{\n\Phi_0}{\bar y+ry},\Phi_0\right)-(\bar y+ry)\pt_z\tilde \MG_r\left(\frac{\n\Phi_0}{\bar y+ry},\Phi_0\right)\\
=&\pt_{p_ip_j}\tilde\MG_r\left(\frac{\n\Phi_0}{\bar y+ry},\Phi_0\right)\frac{\pt_{ij}\Phi_0}{\bar y+ry}
-r\pt_{p_ip_2}\tilde\MG_r\left(\frac{\n\Phi_0}{\bar y+ry},\Phi_0\right)\frac {\pt_i\Phi_0}{(\bar y+ry)^2}\\
&+\pt_{p_iz}\tilde\MG_r\left(\frac{\n\Phi_0}{\bar y+ry},\Phi_0\right)\pt_i\Phi_0-(\bar y+ry)\pt_z\tilde\MG_r\left(\frac{\n\Phi_0}{\bar y+ry},\Phi_0\right)\\
\geq& C_0m_0e^{-\tau|X-X_0|^2}\bigg[\frac{2\tau\mathfrak b_*}{\bar y+ry} \left(2\tau|X-X_0|^2-1\right)
-\frac{2\tau r\mathfrak b^*}{(\bar y+ry)^2}|X-X_0|\\
&-2\tau r\delta|X-X_0|-r^2\d(\bar y+ry)\bigg]\\
\geq&0,
\end{align*}
provided $r/\bar y$ is sufficiently small depending on $B_*,\, S_*,\,\bar H, \, \g$, and $\v$. Note that  $\phi\geq \Phi_0$ on $\pt(B_{2}(X_0)\backslash B_{\frac12}(X_0))$. Then it follows from the comparison principle that $\phi\geq \Phi_0$ in $B_{2}({X_0})\backslash B_{\frac12}({X_0})$. 
Since one can find a positive constant $\bar C=\bar C(\tau)<1/2$ such that
$$\Phi_0\geq \bar CC_0m_0(2-|X-X_0|) \q \text{for any } X\in B_2(X_0)\backslash B_{\frac12}(X_0),$$
the claim \eqref{barrier2} follows.

\emph{Step 2. Gradient estimate.} We claim if $r\leq r_2\bar y$ for some $r_2=r_2(B_*, S_*,\bar H, \g,\v)>0$, then
\begin{equation}\label{gradient estimate}
\int_{B_{2}({X_0})}|\n(\phi-\psi^*)|^2\leq C(\mathfrak b_*)\ld^2\bar y^2\int_{B_{2}({X_0})}\chi_{\{\psi^*=0\}}.
\end{equation}
Indeed, the minimality of $\psi^*$ yields
\begin{align*}
	&\int_{B_{2}({X_0})}(\bar y+ry)\left[\tilde\MG_r\left(\frac{\n\psi^*}{\bar y+ry},\psi^*\right)-\tilde\MG_r\left(\frac{\n\phi}{\bar y+ry},\phi\right)\right]\\
	\leq & \ld^2\int_{B_{2}({X_0})}(\bar y+ry)(\chi_{\{\phi>0\}}-\chi_{\{\psi^*>0\}})
	\leq \ld^2\int_{B_{2}({X_0})}(\bar y+ry)\chi_{\{\psi^*=0\}}.
\end{align*}
Using the Taylor expansion, \eqref{barrier}, \eqref{Gr_dpp} and \eqref{eq_rescale3}, one has
\begin{equation}\label{ineq1}\begin{split}
		&\int_{B_{2}({X_0})}(\bar y+ry)\left[\tilde\MG_r\left(\frac{\n\psi^*}{\bar y+ry},\psi^*\right)-\tilde\MG_r\left(\frac{\n\phi}{\bar y+ry},\phi\right)\right]\\
		\geq&\frac12\int_{B_2(X_0)}(\bar y+ry)\bigg(\mathfrak b_*\left|\frac{\n(\psi^*-\phi)}{\bar y+ry}\right|^2-r\delta\left|\frac{\n(\psi^*-\phi)}{\bar y+ry}\right||\psi^*-\phi|
		-{r^2}\delta|\psi^*-\phi|^2\bigg).
	\end{split}
\end{equation}
By the H\"{o}lder and Poincar\'{e} inequalities, the last two terms can be absorbed by the first term on the right-hand side of \eqref{ineq1}, provided $r\leq r_2\bar y$ with $r_2$ sufficiently small depending on $B_*,\, S_*,\,\bar H, \,\g$, and $\v$. Thus one has \eqref{gradient estimate}.

\emph{Step 3. Poincar\'{e} type estimate. } We claim the following Poincar\'{e} type estimate
\begin{equation}\label{pctype estimate}
C_0m_0|\D_0|^{\frac12}\leq C\|\n(\phi-\psi^*)\|_{L^2(B_{2}({X_0}))} \q\text{with}\ \ \D_0:=\{X\in B_{{2}}({X_0}): \psi^*(X)=0\},
\end{equation}
provided $r\leq r_1\min\{\ld,1\}\bar y$ where $r_1$ is the same as in Step 1. For $Z\in B_1({X_0})$, consider a transformation $\mathscr A_Z$ from $B_{2}({X_0})$ to itself which fixes $\pt B_{2}({X_0})$ and maps $Z$ to ${X_0}$, for instance, $\mathscr A_Z^{-1}(X)=\frac{2-|X-{X_0}|}{2}(Z-{X_0})+X$. Set $\psi_{Z}^*(X):=\psi^*(\mathscr A_Z^{-1}(X))$ and $\phi_Z(X):=\phi(\mathscr A_Z^{-1}(X))$. Given a direction $\xi\in \pt B_1$, define $r_\xi:=\inf \mathcal{R}_\xi$ if
$$\mathcal{R}_\xi:=\{s:1/4\leq s\leq 2,\,\psi_{Z}^*({X_0}+s\xi)=0\}\neq \emptyset.$$
Hence
\begin{equation*}\label{ineq2}\begin{split}
\phi_{Z}({X_0}+r_\xi\xi)&=\int_{2}^{r_\xi}\frac{d}{ds}(\phi_Z-\psi^*_{Z})({X_0}+s\xi)ds\leq \int_{r_\xi}^{2}\left|\n(\phi_Z-\psi^*_{Z})({X_0}+s\xi)\right|ds\\
&\leq\sqrt{2-r_\xi}\left(\int_{r_\xi}^{2}|\n(\phi_Z-\psi^*_{Z})({X_0}+s\xi)|^2ds\right)^{\frac12}.
\end{split}\end{equation*}
On the other hand, by \eqref{barrier2}, one has
\begin{equation*}\label{ineq3}
\phi_Z({X_0}+r_\xi\xi)\geq \bar CC_0m_0\left(2-\left|\frac{2-r_\xi}{2}(Z-{X_0})+r_\xi\xi\right|\right)\geq \frac12\bar CC_0m_0(2-r_\xi).
\end{equation*}
Combining the above two inequalities together gives
$$(C_0m_0)^2(2-r_\xi)\leq \frac4{\bar C^2}\int_{r_\xi}^{2}|\n(\phi_Z-\psi^*_{Z})({X_0}+s\xi)|^2ds.$$
An integration of $\xi$ over $\pt B_1$ yields
\begin{equation}\label{eq:scale1}
(C_0m_0)^2\int_{B_{2}({X_0})\backslash B_{1/2}(Z)}\chi_{\{\psi^*=0\}}\leq C\int_{B_{2}({X_0})}|\n(\phi-\psi^*)|^2.
\end{equation}
More precisely, noting that $|\mathscr A^{-1}_Z(X)-Z|\geq 1/2$ implies $|X-X_0|\geq 1/4$, changing back to the original variables and using the polar coordinates one has
\begin{align*}
	\int_{B_2(X_0)\setminus B_{1/2}(Z)} \chi_{\{\psi^*=0\}}& \leq C\int_{B_2(X_0)\setminus B_{1/4}(X_0)} \chi_{\{\psi^*_Z =0\}}\\
	&=C \int_{\pt B_1} \int_{1/4}^2  \chi_{\{\psi^*_Z=0\}} s\ ds d\xi \leq C \int_{\pt B_1} (2-r_\xi) \ d\xi.
\end{align*}
For the right hand side of \eqref{eq:scale1} we simply estimate
\begin{align*}
	\int_{\pt B_1}\int_{r_\xi}^2|\nabla (\phi_Z-\psi^*_Z)(X_0+s\xi)|^2dsd\xi \leq C \int_{B_2(X_0)} |\nabla (\phi-\psi^*)|^2 dX.
\end{align*}
Thus \eqref{eq:scale1} follows. A further integration over $Z\in B_1({X_0})$ yields
$$(C_0m_0)^2|\D_0|\leq  C\int_{B_{2}({X_0})}|\n(\phi-\psi^*)|^2,$$
where $C$ is a universal constant. This is the desired estimate \eqref{pctype estimate}.

\emph{Step 4. Decay estimate.}  Combining \eqref{pctype estimate} and \eqref{gradient estimate} gives
$$m_0|\D_0|^{\frac{1}{2}}\leq C\ld\bar y|\D_0|^{\frac{1}{2}}$$
for some $C=C(B_*, S_*,\bar H, \g,\v)$, provided
$$r\leq \bar r\min\{\ld,1\}{\rm dist}(\bar{X},\pt\D)\leq\bar r\min\{\ld,1\}\bar y$$
where $\bar r:=\min\{r_1,r_2\}$. If $|\D_0|>0$, then $m_0\leq C\ld\bar y$, which contradicts with  \eqref{m0} since the constant $C_m$ is sufficiently large. Otherwise, $\psi^*=\phi$ a.e. in $B_{2}({X_0})$. By the interior regularity theory for elliptic equations (cf. \cite{GT_book}), $\psi^*$ and $\phi$ are continuous. Thus $\psi^*=\phi$ pointwise in $B_{2}({X_0})$, which contradicts with the fact that ${X_0}$ is a free boundary point. This finishes the proof of the lemma.
\end{proof}

Lemma \ref{linear growth} together with the elliptic estimates away from the free boundary gives the Lipschitz regularity of minimizers.

\begin{proposition}\label{Lipschitz}
Let $\psi$ be a minimizer of \eqref{minimization problem}, then $\psi\in C_{{\rm loc}}^{0,1}(\D)$. Moreover, for any connected domain ${\D'}\Subset \D$ containing a free boundary point, the Lipschitz constant of $\psi$ in $\D'$ is estimated by $C\ld$, where $C$ depends on $B_*,\, S_*,\,\bar H,\, \g,\, \v,\,\D'$, and $\D$.
\end{proposition}

\begin{proof}
(i) Let $\bar X\in\{\psi<Q\}\cap \D$. If ${\rm dist}(\bar X,\G_\psi)>\bar r\min\{\ld,1\}{\rm dist}(\bar X,\pt\D)$, where $\bar r$ is given in Lemma \ref{linear growth}, then the standard interior elliptic estimate gives $\psi\in C^{2,\a}$. If ${\rm dist}(\bar X,\G_\psi)\leq\bar r\min\{\ld,1\}{\rm dist}(\bar X,\pt\D)$, then let
$$\psi^*_{\bar X,r}(X):=\frac{Q-\psi(\bar X+rX)}{r}, \q r:={\rm dist}(\bar X,\G_\psi).$$
By Lemma \ref{linear growth}, there is a constant $C=C(B_*, S_*,\bar H, \g,\v)>0$ such that
$$0\leq \psi^*_{\bar X,r}(X)\leq C\ld\bar y \q\text{for any } X\in B_1.$$
Note that $\psi^*_{\bar X,r}$ satisfies \eqref{ellipticequ scale} in $B_1$. By virtue of \eqref{G dpp}, $\pt_x\psi^*_{\bar X,r}$ (or $\pt_y\psi^*_{\bar X,r}$) solves a uniformly elliptic equation of divergence form and thus is $C^{0,\a}$ by the De Giorgi-Nash-Moser estimate (cf. \cite{GT_book}).
Hence one has
$$|\n\psi(\bar X)|=|\n\psi^*_{\bar X,r}(0)|\leq C\left(\|\psi^*_{\bar X,r}\|_{L^{\infty}(B_1)}+\|\pt_z\tilde\MG_r\|_{L^{\infty}(B_1)}\right)\leq C\ld $$
for some $C=C(B_*, S_*,\bar H, \g,\v,\bar y)$, where the upper bound for $|\pt_z\tilde\MG_r|$ in \eqref{eq_rescale3} has been used to get the last inequality. Combining the above two cases we conclude that $\psi\in C^{0,1}_{\rm loc}(\mathcal{D})$. Furthermore, $|\nabla\psi(x)|\leq C\lambda$, where $C=C(B_*, S_*,\bar H, \g,\v,\bar y)$.

(ii) If ${\D'}\Subset \D$ is connected and ${\D'}$ contains a free boundary point, then it follows from the Harnack inequality (\cite[Theorem 8.17]{GT_book}), the connectedness of ${\D'}$, and Lemma \ref{linear growth} that
$$Q-\psi\leq C\ld \q\text{in }{\D'}$$
for some $C$ depending on $B_*,\, S_*,\,\bar H, \,\g,\,\v,\, {\D'}$, and $\D$. Thus similar arguments as in the second case of (i) give that $|\nabla \psi(x)|\leq C \lambda$.
\end{proof}

\begin{remark}\label{Lipschitz boundary}
By the boundary estimate for the elliptic equation, $\psi$ is Lipschitz up to the $C^{1,\a}$ portion $\Sigma\subset \pt \D\cap\{y>0\}$ as long as the boundary data $\psi^\sharp\in C^{0,1}(\D\cup \Sigma)$.  Moreover, if $\D'$ is a subset of $\overline{\D}\cap\{y>0\}$ with ${\D'}\cap \pt\D$ being $C^{1,\a}$, ${\D'}\cap \D$ is connected, and ${\D'}$ contains a free boundary point, then $|\n\psi|\leq C\ld$ in ${\D'}$.
\end{remark}

The next lemma gives the nondegeneracy of minimizers, whose proof is also based on the minimality of the energy and suitable barrier functions as in Lemma \ref{linear growth}.

\begin{lemma}\label{nondegeneracy}
	Let $\psi$ be a minimizer of \eqref{minimization problem}. Then for any $\vartheta>1$ and any $0<a<1$, there exist constants $c_a,\, r_*>0$  depending on $B_*,\, S_*,\,\bar H,\, \g,\,\v,\, \vartheta$, and $a$, such that for any $B_r(\bar{X})\subset\D$ with $\bar X=(\bar x,\bar y)$ and $r\leq \min\{r_*\bar y,\frac{c_a\lambda}{\d\v}\}$, if
	\begin{equation}\label{nondegeneracy 1}
		\frac 1r\left(\dashint_{B_r(\bar{X})}|Q-\psi|^\vartheta\right)^{\frac 1{\vartheta}}\leq c_a\ld \bar y,
	\end{equation}
	then $\psi=Q$ in $B_{ar}(\bar{X})$.
\end{lemma}
\begin{proof}
	As in Lemma \ref{linear growth}, let $\psi_{\bar X, r}^*(X)$ be defined in \eqref{def:psi*} for $r\in(0,r_*\bar y)$, where $r_*>0$ is a small constant to be determined later. Then $\psi^*_{\bar X,r}$ satisfies \eqref{ellipticsub scale}-\eqref{ellipticequ scale}. Now we need to prove that if
	$$\left(\dashint_{B_1}|\psi_{\bar X,r}^*|^\vartheta\right)^{\frac 1{\vartheta}}\leq c_a\ld \bar y,$$
	then $\psi^*_{\bar X,r}=0$ in $B_{a}$. From now on, we drop the subscript $\bar X$ and $r$ of $\psi^*_{\bar X,r}$ without ambiguity. By the $L^{\infty}$-estimate for the subsolution of the elliptic equation \eqref{ellipticequ scale} (cf. \cite[Theorem  8.17]{GT_book}) and the estimate \eqref{eq_rescale3}, for any $a\in(0,1)$, there exists a constant $c_*=c_*(B_*, S_*,\bar H, \g,\v, \vartheta,a)>0$ such that
	\begin{equation}\label{Ma def}
		\begin{split}
		\sup_{B_{\sqrt{a}}}\psi^*\leq&  \frac{c_*}{4}(\|\psi^*\|_{L^\vartheta(B_1)}+\|(\bar y+ry)\pt_{z}\tilde{\MG}_r\|_{L^{\infty}(B_1)})
		\leq\frac{c_*}2(c_a\ld\bar y+\bar yr\d\v)\\
	    \leq& c_*c_a\ld\bar y:=M_a,
		\end{split}
	\end{equation}
	provided $r\leq\min\{\frac{\bar y}4,\frac{c_a\ld}{\d\v}\}$. Thus it suffices to show that if $c_a$ is sufficiently small depending on $B_*$, $S_*$, $\bar H$, $\g$, $\v$, $\vartheta$, and $a$,
	then $\psi^*=0$ in $B_{a}$. This is proved in three steps.
	
	\emph{Step 1. Upper bound of the energy.} Let $c_*$ in \eqref{Ma def} be sufficiently large. We claim that if $r\leq\min\{r_*\bar y,\frac{c_a\ld}{\d\v}\}$ for some $r_*$ sufficiently small depending on $B_*,\, S_*,\,\bar H, \,\g,\,\v$, and $a$, then there exists $C=C(B_*, S_*,\bar H, \g,a)>0$ such that
	\begin{equation}\label{upper energy}
		\int_{B_{a}}(\bar y+ry)\left[\tilde\MG_r\left(\frac{\n\psi^*}{\bar y+ry},\psi^*\right)+\ld^2\chi_{\{\psi^*>0\}}\right]
		\leq 
		C\frac{M_a}{\bar y}\int_{\pt B_{a}}\psi^*,
	\end{equation}
	where $\tilde\MG_r$ is defined in \eqref{Gsacle}. Let
	$$\Phi_0(X):=\begin{cases}
		M_a\frac{e^{-\tau a^2}-e^{-\tau|X|^2}}{e^{-\tau a^2-e^{-\tau a}}} &\text{for } X\in B_{\sqrt{a}}\backslash B_{a},\\
		0 &\text{for } X\in B_{a}.
	\end{cases}$$
	Similar as the proof for Lemma \ref{linear growth}, choose $\tau=\tau(a)$ sufficiently large and let $r\leq r_*\bar y$ for a sufficiently small $r_*=r_*(B_*, S_*,\bar H, \g,\v,a)>0$, then there exists $C_1=C_1(B_*, S_*,\bar H, \g,a)>0$ such that
	\begin{equation}\label{ineq4}
		\pt_i\pt_{p_i}\tilde\MG_r\left(\frac{\n\Phi_0}{\bar y+ry},\Phi_0\right)\leq -C_1\frac{M_a}{\bar y+ry}\q \text{in } B_{\sqrt a}\backslash B_{a}.
	\end{equation}
	Let $\phi:=\min\{\psi^*,\Phi_0\}$. Note that $\phi=\psi^*$ on $\pt B_{\sqrt a}$, $\phi\equiv 0$ in $B_{a}$, and $\{X\in B_{\sqrt a}:\phi(X)>0\}\subset\{X\in B_{\sqrt a}:\psi^*(X)>0\}$. Since $\psi^*$ is an energy minimizer, one has
	\begin{equation}\label{ineq5}
		\begin{split}
			&\int_{B_{a}}(\bar y+ry)\left[\tilde\MG_r\left(\frac{\n\psi^*}{\bar y+ry},\psi^*\right)+\ld^2\chi_{\{\psi^*>0\}}\right]\\
			\leq&\int_{B_{\sqrt a}\backslash B_{a}}(\bar y+ry)\left[\tilde\MG_r\left(\frac{\n\phi}{\bar y+ry},\phi\right)-\tilde\MG_r\left(\frac{\n\psi^*}{\bar y+ry},\psi^*\right)\right].
	\end{split}\end{equation}
	By the convexity of $\bp\mapsto\tilde\MG_r(\bp,z)$ one has
	\begin{equation}\label{ineq6}\begin{split}
			&\int_{B_{\sqrt a}\backslash B_{a}}(\bar y+ry)\left[\tilde\MG_r\left(\frac{\n\phi}{\bar y+ry},\phi\right)-\tilde\MG_r\left(\frac{\n\psi^*}{\bar y+ry},\psi^*\right)\right]\\
			\leq&\int_{B_{\sqrt a}\backslash B_{a}}-\pt_{p_i}\tilde\MG_r\left(\frac{\n\phi}{\bar y+ry},\phi\right)\pt_i(\psi^*-\phi)\\
			&+\int_{B_{\sqrt a}\backslash B_{a}} (\bar y+ry)\int_0^1(\phi-\psi^*)\pt_z\tilde\MG_r\left(\frac{\n\psi^*}{\bar y+ry},s\phi+(1-s)\psi^*\right)ds\\
			\leq&\int_{(B_{\sqrt a}\backslash B_{a})\cap\{\psi^*>\Phi_0\}}\left[-\pt_{p_i}\tilde\MG_r\left(\frac{\n\Phi_0}{\bar y+ry},\Phi_0\right)\pt_i(\psi^*-\Phi_0)+r\delta(\bar y+ry)(\psi^*-\Phi_0)\right],
	\end{split}\end{equation}
	where the last inequality follows from the bound for $\pt_z\tilde \MG_r$ in \eqref{eq_rescale3}.
	Multiplying \eqref{ineq4} by $(\psi^*-\Phi_0)_+$ and integrating by parts (noting that $\psi^*-\Phi_0=\psi^*$ on $\pt B_{a}$ and $(\psi^*-\Phi_0)_+=0$ on $\pt B_{\sqrt a}$) one has
	\begin{equation}\label{ineq9}
	\begin{split}
		&\int_{(B_{\sqrt a}\backslash B_{a})\cap\{\psi^*>\Phi_0\}}-\pt_{p_i}\tilde\MG_r\left(\frac{\n\Phi_0}{\bar y+ry},\Phi_0\right)\pt_i(\psi^*-\Phi_0)\\
		\leq&-\int_{\pt B_{a}}\psi^*\pt_{p_i}\tilde\MG_r\left(\frac{\n\Phi_0}{\bar y+ry},\Phi_0\right)\c\mathbf{\nu}_i-\int_{(B_{\sqrt a}\backslash B_{a})\cap\{\psi^*>\Phi_0\}}C_1\frac{M_a}{\bar y+ry}(\psi^*-\Phi_0).
	\end{split}\end{equation}
	Since $r\leq\frac{c_a\ld}{\d\v}$ and $c_*$ in \eqref{Ma def} is sufficiently large, the inequality \eqref{ineq9} together with \eqref{ineq5}-\eqref{ineq6} yields that
	$$\int_{B_{a}}(\bar y+ry)\left[\tilde\MG_r\left(\frac{\n\psi^*}{\bar y+ry},\psi^*\right)+\ld^2\chi_{\{\psi^*>0\}}\right]
	\leq -\int_{\pt B_{a}}\psi^*\pt_{p_i}\tilde\MG_r\left(\frac{\n\Phi_0}{\bar y+ry},\Phi_0\right)\c\mathbf{\nu}_i.$$
	In view of the expression for $|\n\Phi_0|$ on $\pt B_a$ as well as \eqref{Gr_dp}, there is a constant $C=C(B_*, S_*,\bar H, \g,a)$ such that
	$$-\pt_{p_i}\tilde\MG_r\left(\frac{\n\Phi_0}{\bar y+ry},\Phi_0\right)\c\mathbf{\nu}_i\leq C\frac{M_a}{\bar y} \quad\text{on }\partial B_a.$$
	Combining the above two estimates gives the desired estimate \eqref{upper energy}.
	
	\emph{Step 2. Lower bound of the energy.} We claim that there is a constant $C=C(B_*, S_*,\bar H, \g,\v, \vartheta,a)>0$ such that
	\begin{equation}\label{lowerbound}
		\int_{\pt B_{a}}\psi^*\leq \frac {C(1+c_a^2)}{\ld}\left(\frac{M_a}{\ld\bar y}+1\right)
		\int_{B_{a}}(\bar y+ry)\left[\tilde\MG_r\left(\frac{\n\psi^*}{\bar y+ry},\psi^*\right)+\ld^2\chi_{\{\psi^*>0\}}\right],
	\end{equation}
	provided $r\leq \min\{\frac{\bar y}4,\frac{c_a\ld}{\d\v}\}$. Indeed, by the trace estimate there exists a constant $C=C(a)>0$ such that
	$$\int_{\pt B_{a}}\psi^*\leq C\left(\int_{B_{a}}\psi^*+\int_{B_{a}}|\n\psi^*|\right).$$
	Since $0\leq\psi^*\leq M_a$ in $B_{a}$ implies
	$$\int_{B_{a}}\psi^*\leq M_a\int_{B_{a}}\chi_{\{\psi^*>0\}}\leq \frac{M_a}{\bar y}\int_{B_{a}}(\bar y+ry)\chi_{\{\psi^*>0\}},$$
	and since the Cauchy-Schwarz inequality gives
	\begin{align*}
		\int_{B_{a}}|\n\psi^*|\leq\int_{B_{a}}\frac{|\n\psi^*|^2}{\ld(\bar y+ry)}+\int_{B_{a}}\ld(\bar y+ry)\chi_{\{\psi^*>0\}},
	\end{align*}
	it follows from the above inequalities that
	$$\int_{\pt B_{a}}\psi^*\leq \frac C{\ld}\left(\frac{M_a}{\ld
		\bar y}+1\right)\int_{B_{a}}(\bar y+ry)\left(\left|\frac{\n\psi^*}{\bar y+ry}\right|^2+\ld^2\chi_{\{\psi^*>0\}}\right).$$
	By \eqref{eq_rescale2} and \eqref{Ma def} one has
	\begin{equation}\label{ineq8}\begin{split}
			\int_{B_{a}}(\bar y+ry)\left|\frac{\n\psi^*}{\bar y+ry}\right|^2
			&\leq \frac2{\mathfrak b_*}\int_{B_{a}}(\bar y+ry)\left[\tilde\MG_r\left(\frac{\n\psi^*}{\bar y+ry},\psi^*\right)+\d rM_a\chi_{\{\psi^*>0\}}\right]\\
			&\leq C(1+c_a^2)\int_{B_{a}}(\bar y+ry)\left[\tilde\MG_r\left(\frac{\n\psi^*}{\bar y+ry},\psi^*\right)+\ld^2\chi_{\{\psi^*>0\}}\right],
	\end{split}\end{equation}
    where $C=C(B_*, S_*,\bar H, \g,\v, \vartheta,a)$.
    Combining all the above estimates together yields \eqref{lowerbound}.
	
	\emph{Step 3. Conclusion.} From \eqref{upper energy} and \eqref{lowerbound}, if $r\leq \min\{r_*\bar y,\frac{c_a\lambda}{\d\v}\}$, there is a constant $C=C(B_*, S_*,\bar H, \g,\v,\vartheta,a)$ such that
	\begin{align*}
		&\int_{B_{a}}(\bar y+ry)\left[\tilde\MG_r\left(\frac{\n\psi^*}{\bar y+ry},\psi^*\right)+\ld^2\chi_{\{\psi^*>0\}}\right]\\
		\leq&C(1+c_a^2)\frac{M_a}{\ld\bar y}\left(\frac{M_a}{\ld\bar y}+1\right)\int_{B_{a}}(\bar y+ry)\left[\tilde\MG_r\left(\frac{\n\psi^*}{\bar y+ry},\psi^*\right)+\ld^2\chi_{\{\psi^*>0\}}\right].
	\end{align*}
	Take $c_a$ in \eqref{Ma def} sufficiently small such that $M_a/(\ld\bar y)$ is sufficiently small. Then
	$$\int_{B_{a}}(\bar y+ry)\left[\tilde\MG_r\left(\frac{\n\psi^*}{\bar y+ry},\psi^*\right)+\ld^2\chi_{\{\psi^*>0\}}\right]=0.$$
	This together with the inequality \eqref{ineq8} yields $\psi^*\equiv 0$ in $B_{a}$. Hence the proof of the lemma is complete.
\end{proof}

\subsection{Measure theoretic properties and  regularity of the free boundary}

From the Lipschitz regularity (Proposition \ref{Lipschitz}) and nondegeneracy property  (Lemma \ref{nondegeneracy}) of a minimizer $\psi$, one has the following measure theoretic properties of the free boundary $\Gamma_{\psi}:=\pt\{\psi<Q\}\cap \mathcal{D}$.
The proof is the same as that for Theorem 2.8 and Theorem 3.2 in  \cite{ACF84}, so we omit the details here.
\begin{proposition}\label{measure}
Let $\psi$ be a minimizer of \eqref{minimization problem}. Then the following statements hold.
\begin{itemize}
  \item [\rm (i)] $\mathcal{H}^1(\G_{\psi}\cap {\D'})<\infty$ for any ${\D'}\Subset \D$.

  \item [\rm (ii)] There is a Borel measure $\zeta_\psi$ such that
  $$\pt_i\pt_{p_i}\MG\left(\frac{\n\psi}{y},\psi\right)-y\pt_z\MG\left(\frac{\n\psi}{y},\psi\right)=\zeta_\psi \mathcal{H}^{1}\lfloor\G_\psi.$$

\item [\rm (iii)] For any $\D'\Subset \mathcal{D}$, there exist  positive constants $C$, $c$, and (small) $r_0$ depending on $B_*,\, S_*,\,\bar H,\, \g,\,\v,\,\ld,\,\D'$, and $\D$, 
such that for any $B_r(X)\subset \D'$ with $X\in \Gamma_\psi$ and $r\in(0,r_0)$,
one has
\begin{equation*}
	\begin{split}
		&c \leq \frac{|B_r(X)\cap \{\psi<Q\}|}{|B_r(X)|} \leq C;\\
		&c\leq \zeta_\psi\leq C,\quad cr\leq \mathcal{H}^{1}(B_r(X)\cap \Gamma_\psi)\leq Cr.
	\end{split}
\end{equation*}
  \end{itemize}
\end{proposition}

It follows from Lemma \ref{linear growth}, Lemma \ref{nondegeneracy} and Remark \ref{ld Ld} that $|\n\psi/y|=O(\Ld)$ near the free boundary in a weak sense.
In order to get the regularity of the free boundary, we derive the following strengthened estimates for $|\nabla \psi/y|$ near the free boundary.

\begin{lemma}\label{gradientpsi estimate}
Let $\psi$ be a minimizer of \eqref{minimization problem}.
Then for any ${\D'}\Subset \D$, there exist $\a\in(0,1)$, $r_0\in(0,1)$, and $C>0$ depending on $B_*,\, S_*,\,\bar H,\, \g,\,\v,\,\ld,\,\D'$, and $\D$, such that for any ball $B_{2r}(\bar X)\subset {\D'}$ with $\bar X\in\Gamma_\psi$ and $r\in(0,r_0)$, one has
$$\sup_{B_r(\bar X)}\left|\frac{\n\psi}{y}\right|^2\leq \Ld^2+Cr^\a.$$
\end{lemma}
\begin{proof}
The proof is similar to that for \cite[Lemma 4.11]{LSTX2023}. It follows from the equation of $\psi$ and the methods in \cite[Chapter 15]{GT_book} that $w:=|\n\psi/y|^2$ satisfies
\begin{align*}
	&\frac12\pt_i\left(\mathfrak{a}^{ij}\pt_j  w\right)-\mathfrak a^{ij}\pt_i\left(\frac{\pt_k\psi}{y}\right)\pt_j\left(\frac{\pt_k\psi}{y}\right)+\pt_i\left(\mathfrak b^{ij}\frac{\pt_j\psi}{y}\right)-\mathfrak b^{ij}\pt_i\left(\frac{\pt_j\psi}{y}\right)\\&+\pt_z\mathfrak a^{ij}\pt_i\left(\frac{\pt_j\psi}{y}\right)yw-\pt_z\mathfrak a^{ij}\pt_i\psi\pt_k\left(\frac{\pt_j\psi}{y}\right)\frac{\pt_k\psi}{y}=0
\end{align*}
in $\D\cap\{\psi<Q\}$, where
\begin{align}\label{eq_ab}
	&\mathfrak{a}^{ij}:=\pt_{p_ip_j}\MG\left(\frac{\n\psi}{y},\psi\right)\q\text{and}\q
	\mathfrak{b}^{ij}:=\frac1{y^2}\mathfrak a^{ik}(\pt_j\psi\d_{k2}-\pt_k\psi\d_{j2})+\mathfrak c\d_{ij}
\end{align}
with
$$\mathfrak{c}:=\pt_{p_iz}\MG\left(\frac{\n\psi}{y},\psi\right)\pt_i\psi-y\partial_{z}\MG\left(\frac{\n\psi}{y},\psi\right).$$
By Young's inequality it holds that
$$\left|\mathfrak b^{ij}\pt_i\left(\frac{\pt_j\psi}{y}\right)\right|
\leq \frac{1}{\mathfrak b_\ast}|\mathfrak b^{ij}|^2+\frac{\mathfrak b_\ast}{4}\left|\n\left(\frac{\n\psi}y\right)\right|^2.$$
Moreover, by the definitions of $\MG$ in \eqref{notation JDG} and $G_\v$ in \eqref{G def} one derives
\begin{align*}
	\pt_z\mathfrak a^{ij}(\bp,z)=2\pt_{tz}G_\epsilon(|\bp|^2,z) \delta_{ij}+ 4\pt_{ttz}G_\epsilon(|\bp|^2,z) p_ip_j.
\end{align*}
This combined with $2|\bp|^2\pt_{tz}G_\epsilon(|\bp|^2,z)=\bp\cdot \pt_{\bp z}\MG(\bp,z)$ gives that
\begin{align*}
	&\pt_z\mathfrak a^{ij}\pt_i\left(\frac{\pt_j\psi}{y}\right)yw-\pt_z\mathfrak a^{ij}\pt_i\psi\pt_k\left(\frac{\pt_j\psi}{y}\right)\frac{\pt_k\psi}{y}\\
	=&2\pt_{tz}G_\epsilon \left(\left|\frac{\nabla\psi}y\right|^2,\psi\right)
	\cdot y\left(\pt_i\left(\frac{\pt_i\psi}{y}\right)\left|\frac{\nabla\psi}y\right|^2 -\frac{\pt_i\psi}{y}\pt_k\left(\frac{\pt_i\psi}{y}\right)\frac{\pt_k\psi}{y}\right)\\
	=&\n\psi\c\pt_{\bp z}\MG \left(\frac{\nabla\psi}y,\psi\right)
	\c\left(\n\c\left(\frac{\n\psi}{y}\right) -\frac{\n\psi}{|\n\psi|}\n\left(\frac{\n\psi}{y}\right)\frac{(\n\psi)^T}{|\n\psi|}\right).
\end{align*}
Thus using Young's inequality again yields
\begin{align*}
	&\left|\pt_z\mathfrak a^{ij}\pt_i\left(\frac{\pt_j\psi}{y}\right)yw-\pt_z\mathfrak a^{ij}\pt_i\psi\pt_k\left(\frac{\pt_j\psi}{y}\right)\frac{\pt_k\psi}{y} \right|\\
	\leq& \frac{1}{\mathfrak b_\ast}\left|\nabla\psi\cdot \pt_{\bp z}\MG\left(\frac{\nabla\psi}y,\psi\right)\right|^2+ \frac{\mathfrak b_\ast}{2}\left|\n\left(\frac{\n\psi}y\right)\right|^2.
\end{align*}
Since
\begin{align*}
\mathfrak a^{ij} \pt_i\left(\frac{\pt_k\psi}{y}\right)\pt_j\left(\frac{\pt_k\psi}{y}\right) \geq \mathfrak b_\ast\left|\n\left(\frac{\n\psi}y\right)\right|^2,
\end{align*}
which follows from the convexity of the functional, cf. \eqref{G dpp}, one concludes that $w$ is a subsolution to the uniformly elliptic linear equation
\begin{align*}
	\mathcal{L}w:=\frac{1}{2} \pt_i(\mathfrak a^{ij}\pt_j w) \geq - \pt_i \mathfrak g^i-\mathfrak f,
\end{align*}
where
\begin{align*}
	 \mathfrak g^i:= \mathfrak b^{ij}\frac{\pt_j \psi}y \quad\text{and}\q
	 \mathfrak f:=\frac{1}{\mathfrak b_\ast}\left(|\mathfrak b^{ij}|^2+\left|\nabla\psi\cdot \pt_{\bp z}\MG\left(\frac{\nabla\psi}y,\psi\right)\right|^2\right).
\end{align*}
Note that it follows from the definition of $\mathfrak b^{ij}$ in \eqref{eq_ab}, \eqref{G dpp}, and \eqref{G dzdzpdzz} that
$$|\mathfrak b^{ij}|\leq C(B_*, S_*,\bar H, \g,\v,\D')\left(\|\n\psi\|_{L^\infty(\D')}+1\right).$$
Thus there exists $C=C(B_*, S_*,\bar H, \g,\v,\D')>0$ such that $\mathfrak g^i$ and $\mathfrak f$ satisfy
\begin{align*}
	&\|\mathfrak g^i\|_{L^\infty(\D')}
	\leq C \|\nabla\psi\|_{L^\infty(\D')}(\|\nabla\psi\|_{L^\infty(\D')}+1),\\
	&\|\mathfrak f\|_{L^\infty(\D')}\leq C\mathfrak b_\ast^{-1}(\|\nabla\psi\|^2_{L^\infty(\D')}+1).
\end{align*}

Now for $\tau\geq0$, let
$$W_\tau:=\begin{cases}
(|\frac{\n\psi}{y}|^2-\Ld^2-\tau)^+ &\text{in }\{\psi<Q\},\\
0 & \text{in }\{\psi=Q\}.
\end{cases}$$
Using the same arguments as in \cite[Lemma 3.4]{ACF84} yields
$$\limsup_{X\to \bar X,\psi(X)<Q}\left|\frac{\n\psi(X)}{y}\right|=\Lambda,$$
thus combining this with the interior regularity of $\psi$ together one has $W_\tau\in C^{0,1}(B_r(\bar X))$.
Denote $W_\tau^*(r):=\sup_{B_r(\bar X)}W_{\tau}$. Then $W_{\tau}^*(r)-W_{\tau}$ satisfies
$$\begin{cases}
\mathcal L(W_\tau^*(r)-W_\tau)\leq \pt_i \mathfrak g^i+\mathfrak f &\text{in } B_{2r}(\bar X),\\
W_{\tau}^*(r)-W_{\tau}=W_{\tau}^*(r) &\text{in } B_{2r}(\bar X)\cap\{\psi=Q\}.
\end{cases}$$
In view of \cite[Theorem 8.18]{GT_book} with $\vartheta\in[1,\infty)$ one has
$$r^{-\frac2{\vartheta}}\|W_{\tau}^*(r)-W_{\tau}\|_{L^\vartheta(B_r(\bar X))}\leq C\left(\inf_{B_{r/2}(\bar X)}(W_{\tau}^*(r)-W_{\tau})+k(r)\right),$$
where $C=C(B_*, S_*,\bar H, \g,\v,\vartheta)$ and $k(r)=\mathfrak b_*^{-1}(r\|\mathfrak f\|_{L^\infty(\D')}+r^2\sum_i\|\mathfrak g^i\|_{L^\infty(\D')})$. From the positive density property of $B_r(\bar X)\cap\{\psi=Q\}$ (cf. Proposition \ref{measure} (iii)) we infer that
$$r^{-\frac2{\vartheta}}\|W_{\tau}^*(r)-W_{\tau}\|_{L^\vartheta(B_r(\bar X))}\geq c W_{\tau}^*(r)$$
for some $c=c(B_*, S_*,\bar H, \g,\v,\ld,\D',\D)$, provided $r\leq r_0$.
Taking $\tau\to0$ and a rearrangement of the above two inequalities yield
$$W_0^*\left(\frac r2\right)=\sup_{B_{r/2}(\bar X)}W_0\leq C_0 W_0^*(r)+k(r),\q C_0=1-\frac {c}{C}\in(0,1).$$
It follows from the iteration lemma (cf. \cite[Lemma 8.23]{GT_book}) that for any $\beta\in (0,1)$
\begin{align*}
	W^*_0(r)\leq C W^*_0(r_0)\left(\frac{r}{r_0}\right)^\alpha+ k(r^\beta r_0^{1-\beta}),
\end{align*}
where $C=C(C_0)>0$ and $\alpha=\alpha(C_0,\beta)\in (0,1)$. Recalling the Lipschitz bound for $\psi$ in Proposition \ref{Lipschitz}, one  obtains the desired estimate.
\end{proof}

\begin{remark}\label{rmk_gradu}
In Lemma \ref{gradientpsi estimate}, if $B_{2r}(\bar X)\subset B_{2R}(\bar X)\subset \D'$ with $R\in(0,r_0)$, then there exists a constant $C=C(B_*, S_*,\bar H, \g,\v,\ld,\D',\D)>0$ such that
\begin{equation}\label{eq_gpsiy}
\sup_{B_r(\bar X)}\left|\frac{\nabla\psi}y\right|^2\leq \Lambda^2+C\left(\frac rR\right)^{\alpha}.
\end{equation}
This follows from applying the proof of Lemma \ref{gradientpsi estimate} to $\psi_{\bar X,R}(X)=\psi(\bar X+RX)/R$.
\end{remark}

Similar arguments as in \cite[Proposition 4.12]{LSTX2023} give the following gradient estimate. 

\begin{proposition}\label{prop_1}
Let $\psi$ be a minimizer of \eqref{minimization problem}.
Then for any $B_{2r}(\bar X)\subset \D'\Subset\D$ with $\bar X\in\G_\psi$ and $r\in(0,r_0)$, where $r_0$ is the same as in Lemma \ref{gradientpsi estimate}, one has
$$\dashint_{B_r(\bar X)\cap\{\psi<Q\}}\left(\Ld^2-\left|\frac{\n\psi}{y}\right|^2\right)^+\leq C|\ln r|^{-1},$$
where $C=C(B_*, S_*,\bar H, \g,\v,\ld,\D',\D)>0$ is a constant.
\end{proposition}
Combining Proposition \ref{prop_1} and the arguments in \cite[Corollary 4.4]{ACF84} together we have the following corollary. 

\begin{corollary}\label{halfplane sol}
Let $\psi$ be a minimizer of \eqref{minimization problem}.
Then every blow up limit of $\psi$ at $\bar{X}=(\bar x,\bar y)\in\G_\psi\cap \D$ is a half plane solution with slope $\Ld\bar y$ in a neighborhood of the origin.
\end{corollary}

As in the two-dimensional case (\cite{LSTX2023}), now we can use the improvement of flatness arguments to show that the free boundary is locally a $C^{1,\a}$ graph for some $\a\in(0,1)$  (cf. \cite{ACF84}). Higher regularity of the free boundary follows from \cite[Theorem 2]{KN1977_regularity}. For the completeness, we state the result here without proof.

\begin{proposition}
Let $\psi$ be a minimizer of \eqref{minimization problem}. 
The free boundary $\G_\psi$ is locally $C^{k+1,\a}$ if $\MG(\bp,z)$ is $C^{k,\a}$ in its components ($k\geq1$, $0<\alpha<1$), and it is locally real analytic if $\MG(\bp,z)$ is real analytic.
\end{proposition}
\begin{remark}
In our case the function $\mathcal{G}(\bp,z)$ is $C^{1,1}$ in its components (cf. Section \ref{subsec_existence}), thus the free boundary is locally $C^{2,\alpha}$ for any $\alpha\in (0,1)$.
\end{remark}

\section{Fine properties for the free boundary problem}\label{sec fine properties}

In this section, we obtain the uniqueness and monotonicity of the solution for the truncated problem \eqref{variation problem} with given specific boundary conditions. The monotonicity of the solution is crucial to prove the graph property of the free boundary, as well as the equivalence between the Euler system and the stream function formulation.

\begin{center}
	\includegraphics[height=5cm, width=10cm]{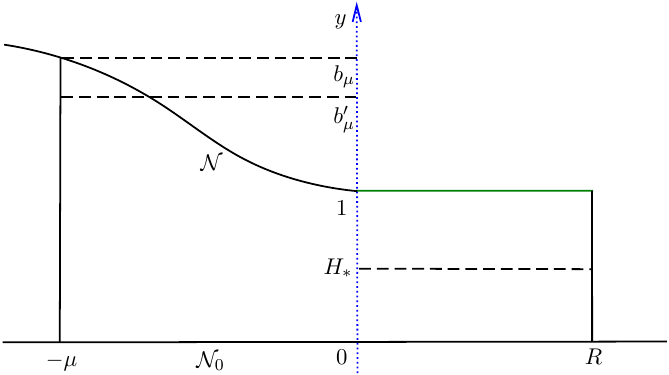}\\
	{\small Figure 1. The truncated domain $\Omega_{\mu, R}$.}
\end{center}

To describe the boundary data on $\pt\Omega_{\mu,R}$, some notations are needed. Let $b_\mu\in(1,\bar H)$ be such that $N(b_\mu)=-\mu$, where $N$ is defined in \eqref{nozzle}. Choose a point $(-\mu,b_\mu')$ with  $0<b_\mu-b'_\mu<(b_\mu-1)/4$. 
Let $H_*:=H_*(\Lambda)$ be such that $\Ld H_*^2e^{1-H_*}=Q$ if $\Ld>Q$, and $H_*=1$ if $\Ld\leq Q$. Define
	\begin{equation*}
		\psi^\dag(y) :=\left\{
		\begin{aligned}
			&\min\left(\Ld y^2e^{1-y}, Q\right)\quad &\text{if } H_*<1,\\
			&Q y^2e^{1-y}\quad &\text{if } H_*=1.
		\end{aligned}
		\right.
	\end{equation*}
	Let $s\in(3/2,2)$ be a fixed constant.
Set
\begin{eqnarray}\label{psi0}
	\psi^\sharp_{\mu, R}(x,y):=
	\left\{ \begin{split}
		& 0 &&\text { if } x=-\mu,\, 0<y<b_\mu',\\
		& Q\left(\frac{y-b_\mu'}{b_\mu-b'_\mu}\right)^{s} && \text{ if } x=-\mu,\, b_\mu'\leq y\leq b_\mu,\\
		& Q &&\text{ if } (x,y)\in \N\cup ([0,R]\times \{1\}),\\
		& \psi^\dag(y) &&\text{ if } x=R,\, 0<y<1,\\
		& 0 &&\text{ if } -\mu\leq x\leq R,\, y=0.
	\end{split}\right.
\end{eqnarray}
	Note that $\psi^\sharp_{\mu, R}$ is continuous and it satisfies $0\leq \psi^\sharp_{\mu, R}\leq Q$.

\begin{lemma}\label{psi0 sup-subsol}
	Assume that the constant $\k$ in \eqref{BS condition2} is sufficiently small depending on $B_*,\, S_*,\, \bar H$, and $\g$. Then  $\psi^\sharp_{\mu,R}(-\mu,\c)$ is a subsolution to \eqref{EL equ} and $\psi^\sharp_{\mu,R}(R,\c)$ is a supersolution to \eqref{EL equ} in $\O_{\mu,R}$.
\end{lemma}
\begin{proof}
	Denote $\vp(y):=Q\left(\frac{y-b_\mu'}{b_\mu-b'_\mu}\right)^{s}$. Straightforward computations yield
	\begin{align*}
		&\left(g_\v\left(\left|\frac{\vp'}{y}\right|^2,\vp\right)\frac{\vp'}{y}\right)'\\
		=&\left(g_\v+2\pt_tg_\v\left(\frac{\vp'}{y}\right)^2\right)\left(\frac{\vp''}{y}-\frac{\vp'}{y^2}\right)+\pt_zg_\v\frac{(\vp')^2}{y}\\
		=&\left(g_\v+2\pt_tg_\v\left(\frac{\vp'}{y}\right)^2\right)\frac{Qs}{y(b_\mu-b'_\mu)^2}\left(\frac{y-b_\mu'}{b_\mu-b'_\mu}\right)^{s-2}\left(s-2+\frac{b_\mu'}{y}\right)
		+\pt_zg_\v\frac{(\vp')^2}{y}.
	\end{align*}
	Recall that the estimates \eqref{g dtgmt},  \eqref{dzgm_dtgm}, and \eqref{G dzdzpdzz} give
	$$g_\v+2\pt_tg_\v\left(\frac{\vp'}{y}\right)^2\geq C>0,
	\q |\pt_zg_\v|\leq C\k_0\pt_tg_\v, \q |\pt_zG_\v|\leq C\k_0(\k_0+1).$$
    Moreover, since $\k$ in \eqref{BS condition2} is sufficiently small depending on $B_*,\, S_*,\, \bar H$, and $\g$, by \eqref{k0} and Proposition \ref{incoming data} one has
	\begin{equation}\label{label_2}
	\k_0\leq C\k^{1-\frac1{2\g}}\leq C\k^{\frac1{4\g}}\leq CQ,
    \end{equation}
	where $C$ depends on $B_*,\, S_*,\, \bar H$, and $\g$.
	Therefore, the function $\vp$ satisfies
	\begin{equation}\label{subsolution 1}
		\left(g_\v\left(\left|\frac{\vp'}{y}\right|^2,\vp\right)\frac{\vp'}{y}\right)'-y\pt_zG_\v\left(\left|\frac{\vp'}{y}\right|^2,\vp\right)>0,\q \text{on }(b_\mu',b_\mu).
	\end{equation}
	As $0$ is a solution to \eqref{EL equ}, the function $\psi^\sharp_{\mu,R}(-\mu,\c)$ is a subsolution to \eqref{EL equ}.
	
	Similarly one can show $\psi^\sharp_{\mu,R}(R,\c)$ is a supersolution to \eqref{EL equ}, provided $\k$ is sufficiently small depending on $B_*,\, S_*,\, \bar H$, and $\g$. This finishes the proof of the lemma.
\end{proof}

\begin{lemma}\label{psi bound lem}
	Let $\psi$ be a minimizer of the truncated problem \eqref{variation problem} in $\O_{\mu,R}$ with the boundary value $\psi^\sharp_{\mu,R}$ constructed in \eqref{psi0}. Assume that the constant $\k$ in \eqref{BS condition2} is sufficiently small depending on $B_*,\, S_*,\, \bar H$, and $\g$. Then
	\begin{equation}\label{psi bound}
		\psi^\sharp_{\mu,R}(-\mu,y)<\psi(x,y)\leq\psi^\sharp_{\mu,R}(R,y) \q\text{for all } (x,y)\in \O_{\mu,R}.
	\end{equation}
\end{lemma}
\begin{proof}
	The proof is divided into two steps.
	
	\emph{Step 1. Lower bound.}
	Let $\D':=\O_{\mu,R}\cap\{b_\mu'<y<b_\mu\}$. By Lemma \ref{psi0 sup-subsol}, $\psi^\sharp_{\mu,R}(-\mu,\c)\leq\psi$ on $\pt \D'$. Furthermore, it follows from \eqref{subsolution 1} that $\psi^\sharp_{\mu,R}(-\mu,\c)$ is a strict subsolution in $\D'$. Thus by the comparison principle and the strong maximum principle, $\psi^\sharp_{\mu,R}(-\mu,y)<\psi(x,y)$ in $\D'$. Since $\psi^\sharp_{\mu,R}(-\mu,y)\equiv0$ if $y\in[0,b_\mu']$, one has $\psi^\sharp_{\mu,R}(-\mu,y)<\psi(x,y)$ in $\O_{\mu,R}$.

\emph{Step 2. Upper bound.}
The idea is the same as that in \cite[Lemma 4.1]{ACF85}.
We consider $\psi$ as a limit of minimizers for the modified minimization problems.
For any small $\iota>0$, introduce the functional
\begin{equation}\label{J_approx}
	\mathcal J_\iota(\phi):=\int_{\Omega_{\mu,R}}(y+\iota)\bigg[G_\v\bigg(\left|\frac{\n\phi}{y+\iota}\right|^2,\phi\bigg)+\ld_\v^2\chi_{\{\phi<Q\}}\bigg]dX,
\end{equation}
where $\ld_\epsilon$ is the same as in \eqref{Jm def}. As in Section \ref{sec existence and regularity}, one can prove that
\begin{itemize}
	\item [\rm (a)] there exists a minimizer $\psi_\iota$ for the functional $\J_\iota$ in the admissible set \eqref{K_admissible} with the boundary value $\psi^\sharp_{\mu,R}$ constructed in \eqref{psi0};
	\item [\rm (b)] the minimizer $\psi_\iota$ satisfies $0\leq\psi_\iota\leq Q$ in $\O_{\mu,R}$ and  $\psi_\iota\in C^{0,1}_{\rm loc}(\O_{\mu,R})$;
	\item [\rm (c)] the minimizer $\psi_\iota$ satisfies the elliptic equation $\mathscr L\psi_\iota=0$ in $\O_{\mu,R}\cap\{0<\psi_\iota<Q\}$, where
	\begin{equation}\label{equ_psi_iota}
		\mathscr L\phi:=\n\c\bigg(g_\v\bigg(\left|\frac{\n\phi}{y+\iota}\right|^2,\phi\bigg)\frac{\n\phi}{y+\iota}\bigg)-(y+\iota)\pt_zG_\v\bigg(\left|\frac{\n\phi}{y+\iota}\right|^2,\phi\bigg).
	\end{equation}
\end{itemize}
 Furthermore, 
 there exists a sequence $\iota_k\to0$, such that the corresponding sequence of minimizers  $\psi_{\iota_k}$ converges to a minimizer $\psi$ of the problem \eqref{variation problem}. Thus if we show that
\begin{equation}\label{ineq_1}
	\psi_\iota(x,y)\leq \psi_{\mu,R}^\sharp(R,y+\iota) \q\text{in }\O_{\mu,R}
\end{equation}
for sufficiently small $\iota>0$, then the upper bound for $\psi$ follows.

Consider the function
$\Psi^{\tau}(y):= \psi^\sharp_{\mu, R}(R,y+\iota)+\tau$ for $\tau>0$. In view of the proof of Lemma \ref{psi0 sup-subsol}, the function $\Psi^{\tau}$ is a supsolution to \eqref{equ_psi_iota} if $\k_0$ is sufficiently small. Since $\psi_\iota$ is bounded, if $\tau$ is large enough then $\psi_\iota\leq \Psi^{\tau}$. Denote $\tau_*=\inf\{\tau:\psi_\iota\leq\Psi^{\tau}\}$. We claim $\tau_*=0$. Suppose by contradiction that $\tau_*>0$. Note that $\psi_\iota$ is continuous up to $\{y=0\}$ and $\psi_\iota=\psi_{\mu,R}^\sharp$ on $\pt\O_{\mu,R}$. Thus $\psi_\iota=\Psi^{\tau}$ cannot be achieved at $\partial\O_{\mu,R}$. Since $\mathscr L\psi_\iota=0$ and $\mathscr L\Psi^{\tau}\leq0$ in $\O_{\mu,R}\cap\{0<\psi_\iota<Q\}$, the strong maximum principle implies that the functions $\psi_\iota$ and $\Psi^{\tau}$ must touch at a free boundary point $X_*=(x_*,y_*)$ with $0<y_*<1$. Then by the Hopf lemma
$$\Ld=\left|\frac{\n\psi_\iota(X_*)}{y_*+\iota}\right|>\left|\frac{\n\Psi^{\tau_*}(X_*)}{y_*+\iota}\right|=\Ld(2-y_*-\iota)e^{1-y_*-\iota}>\Ld,$$
which is a contradiction. Therefore $\tau_*=0$. This means that the inequality \eqref{ineq_1} holds.

Let $\iota\to0$ one obtains $\psi(x, y)\leq\psi_{\mu,R}^\sharp(R,y)$. This finishes the proof of the lemma.
\end{proof}

\begin{remark}\label{rmk_psi_regularity}
	 Note that every function $\psi_{\iota_k}$ constructed in Lemma \ref{psi bound lem} is a solution to the elliptic equation \eqref{equ_psi_iota} (with $\iota$ being replaced by $\iota_k$) in $\O_{\mu,R}\cap\{0<\psi_{\iota_k}<Q\}$ with the boundary condition $\psi_{\iota_k}=\psi_{\mu,R}^\sharp$ on $\pt\O_{\mu,R}$. By elliptic estimates  $\psi_{\iota_k}\in C^{2,\a}_{\rm loc}(\O_{\mu,R}\cap\{\psi_{\iota_k}<Q\})\cap C^{0}(\overline{\O_{\mu,R}\cap\{\psi_{\iota_k}<Q\}})$ for any $\a\in(0,1)$. Then the limit $\psi$ of $\{\psi_{\iota_k}\}$, which is a minimizer of the minimization problem \eqref{variation problem}, satisfies $\psi\in C^{2,\a}_{\rm loc}(\O_{\mu,R}\cap\{\psi<Q\})\cap C^{0}(\overline{\O_{\mu,R}\cap\{\psi<Q\}})$.	
\end{remark}

With Lemma \ref{psi bound lem} at hand, we can prove the following proposition, which claims that the minimizer for the truncated problem \eqref{variation problem} is unique and monotone increasing in the $x$-direction.

\begin{proposition}\label{psi monotonic}
Under the same assumption as in Lemma \ref{psi bound lem}, $\psi$ is the unique minimizer to \eqref{variation problem}. Furthermore, $\pt_{x}\psi\geq0$ in $\O_{\mu,R}$.
\end{proposition}
\begin{proof}
 The proof is essentially the same as that in \cite[Theorem 4.2]{ACF85}. For completeness we provide the details here.

\emph{Step 1. $\pt_{x}\psi\geq0$.}
For $\tau>0$ small, let
$$\psi^\tau(x,y):=\psi(x-\tau,y) \q \text{for } (x,y)\in \O_{\mu,R}^\tau:=\{(x+\tau,y): (x,y)\in \O_{\mu,R}\}.$$
Assume
$$\psi(x,y)=\psi(R,y) \ \text{ in } \O^\tau_{\mu,R}\backslash \O_{\mu,R} \q\text{and} \q \psi^\tau(x,y)=\psi(-\mu,y) \ \text{ in } \O_{\mu,R}\backslash\O^\tau_{\mu,R}.$$
Since $\psi$ is a minimizer for $J_{\mu,R,\Ld}^\v$ over $\K_{\psi^\sharp_{\mu,R}}$, then $\psi^\tau$ is a minimizer for $J_{\mu,R,\Ld}^{\v,\tau}$ in $\K^\tau_{\psi^{\sharp}_{\mu,R}}$, where
$$J_{\mu,R,\Ld}^{\v,\tau}(\phi):=\int_{\O_{\mu,R}}y\bigg[G_\v\bigg(\left|\frac{\n\phi(x+\tau,y)}{y}\right|^2,\phi(x+\tau,y)\bigg)+\ld_\v^2\chi_{\{\phi(x+\tau,y)<Q\}}\bigg]dX,$$
and
$$\K^\tau_{\psi^{\sharp}_{\mu,R}}:=\{\phi\in H^{1}(\O_{\mu,R}):\phi(x+\tau,y)=\psi^\sharp_{\mu,R}(x,y) \text{ on }\pt\O_{\mu,R}\}.$$
For ease of notations, we denote
$$J:=J_{\mu,R,\Ld}^\v, \q J^\tau:=J_{\mu,R,\Ld}^{\v,\tau},\q \K:=\K_{\psi^\sharp_{\mu,R}},\q \K^\tau:=\K^\tau_{\psi^{\sharp}_{\mu,R}}.$$
By Lemma \ref{psi bound lem} one has
$$\min\{\psi^\tau,\psi\}\in \K^{\tau} \q\text{and}\q \max\{\psi^\tau,\psi\}\in\K.$$
Since the energy functional $J(\phi)$ depends only on $|\n\phi/y|^2$ and $\phi$, one can  verify that
$$J^\tau(\psi^\tau)+J(\psi)=J^\tau(\min\{\psi^\tau,\psi\})+J(\max\{\psi^\tau,\psi\}).$$
This together with the minimality of $\psi^\tau$ and $\psi$, gives
\begin{equation}\label{J equ1}
J(\psi)=J(\max\{\psi^\tau,\psi\}).
\end{equation}

We shall show that $\psi^\tau<\psi$ in $\O_{\mu,R}$. By Lemma \ref{psi bound lem} and the continuity of $\psi$, one has $\psi^\tau<\psi$ near $\{x=-\mu\}\cap \O_{\mu,R}$. If the strict inequality was not true in the connected component of $\O_{\mu,R}\cap\{\psi<Q\}$ containing $\{x=-\mu\}$, then one would find a ball $B\Subset \O_{\mu,R}$ and $\bar X\in\pt B$ with $\psi(\bar X)>0$ such that
$$\psi^\tau<\psi \q\text{in } B,\q \psi^\tau(\bar X)=\psi(\bar X).$$
By the Hopf lemma one has
$$\lim_{X\to\bar X \atop X\in B}\frac{\pt(\psi^\tau-\psi)}{\pt\nu}(X)>0,$$
where $\nu$ is the outer unit normal of $\pt B$ at $\bar X$. Thus there exists a curve through $\bar X$ such that along this curve $\max\{\psi^\tau,\psi\}$ is not $C^1$ at $\bar X$. On the other hand, since $\max\{\psi^\tau,\psi\}$ is a minimizer of $J$ by \eqref{J equ1}, and $\max\{\psi^\tau,\psi\}(\bar X)<Q$, then by the interior regularity for the minimizers, $\max\{\psi^\tau,\psi\}\in C^{1,\a}$ around $\bar X$. This is a contradiction. Since by the strong maximum principle every component $\{\psi<Q\}$ has to touch $\pt\O_{\mu,R}$, and moreover, by the construction of boundary value, $\{\psi<Q\}\cap \pt\O_{\mu,R}$ is a connect arc, then $\{\psi<Q\}\cap\O_{\mu,R}$ consists only one component. Consequently, $\psi^\tau<\psi$ in $\O_{\mu,R}\cap\{\psi<Q\}$. Taking $\tau\to0$ yields $\pt_{x}\psi\geq0$.

\emph{Step 2. Uniqueness.} Assume that $\psi_1$ and $\psi_2$ are two minimizers with the same boundary data. Let $\psi_1^\tau(x,y)=\psi_1(x-\tau,y).$ Same arguments as above give that $\psi_1^\tau\leq\psi_2$ in $\O_{\mu,R}$. Taking $\tau\to0$ yields $\psi_1\leq\psi_2$. Similarly $\psi_2\leq \psi_1$. Thus $\psi_1=\psi_2$.
\end{proof}

Denote the free boundary points which lie strictly below $\{y=1\}$ by
\begin{equation}\label{Gamma truncated}
	\G_{\mu,R,\Ld}:=\pt\{\psi<Q\}\cap\{(x,y): (x,y)\in (-\mu,R]\times (0,1)\}.
\end{equation}
Before proving that the free boundary is a graph, we first show the non-oscillation property of the free boundary.

\begin{lemma}\label{nonoscillation}
Let $\D'\subset\{(x,y): y\geq H_*\}$ be a domain in $\O_{\mu,R}\cap\{\psi<Q\}$ bounded by two disjoint arcs $\G_1$ and $\G_1'$ of the free boundary and by $\{x=k_1\}$ and $\{x=k_2\}$. Suppose that $\G_1$ ($\G_1'$) lies in $\{k_1<x<k_2\}$ with endpoints $(k_1,l_1)$ and $(k_2, l_2)$ ($(k_1,l_1')$ and $(k_2, l_2')$). Suppose that ${\rm dist}((0,1),\D')\geq c_0$ for some $c_0>0$ and $|k_1-k_2|\leq 1$. Then
$$|k_1-k_2|\leq C\max\{|l_1-l_1'|,|l_2-l_2'|\},$$
for some $C=C(B_*,S_*,\bar H,\g,\v,H_*,c_0,\Ld)$.
\end{lemma}
\begin{proof}
Integrating the equation \eqref{EL equ} in $\D'$ yields
\begin{align*}
&\int_{\G_1\cup\G_1'}g_\v\left(\left|\frac{\n\psi}{y}\right|^2,\psi\right)\frac{\pt_{\nu}\psi}{y}+\int_{\pt \D'\cap(\{x=k_1\}\cup\{x=k_2\})}g_\v\left(\left|\frac{\n\psi}{y}\right|^2,\psi\right)\frac{\pt_{\nu}\psi}{y}\\
=&\int_{\D'}y\pt_zG_\v\left(\left|\frac{\n\psi}{y}\right|^2,\psi\right).
\end{align*}
Since $\pt_{\nu}\psi/y=\Lambda$ along $\G_1$ and $\G_1'$ and $g_\v\geq C_*(B_*,S_*,\bar H,\g)$ by \eqref{gm bound}, one has
$$\int_{\G_1\cup\G_1'}g_\v\left(\left|\frac{\n\psi}{y}\right|^2,\psi\right)\frac{\pt_{\nu}\psi}{y}\geq C_*\Ld\mathcal H^1(\G_1\cup\G_1')\geq C_*\Lambda|k_1-k_2|.$$
Applying the interior Lipschitz regularity (Proposition \ref{Lipschitz}) and the boundary estimate for $\psi$ (Remark \ref{Lipschitz boundary}) gives
$$\left|\int_{\pt \D'\cap(\{x=k_1\}\cup\{x=k_2\})}g_\v\left(\left|\frac{\n\psi}{y}\right|^2,\psi\right)\frac{\pt_{\nu}\psi}{y}\right|\leq C\lambda\max\{|l_1-l_1'|,|l_2-l_2'|\},$$
where $C$ depends on the $B_*$, $S_*$, $\bar H$, $\g$, $\v$, $H_*$,  and $c_0$. Moreover, note that $|k_1-k_2|\leq 1$. By \eqref{G dzdzpdzz} it holds
\begin{equation*}
	\left|\int_{\D'}y\pt_zG_\v\left(\left|\frac{\n\psi}{y}\right|^2,\psi\right)\right|
	\leq \d\bar H|\D'|\leq \d\bar H\max\{|l_1-l_1'|,|l_2-l_2'|\}.
\end{equation*}
Combining the above estimates and \eqref{ld_Ld} together gives the desired inequality.
\end{proof}

The next lemma gives the uniqueness of the Cauchy problem for the elliptic equation, which is crucial for the proof of the graph property of the free boundary.

\begin{lemma}\label{Cauchy uniqueness}
Let $\tilde\D\Subset\R^2\cap\{y>0\}$ be a bounded connected piecewise $C^{2,\alpha}$ domain with $\alpha\in(0,1)$, and let $\Sigma\subset \pt\tilde\D$ be a nonempty open set. Let $\psi,\ \tilde{\psi}\in H^{1}_{\rm loc}(\tilde\D)$ be two solutions to the Cauchy problem
\begin{equation*}\left\{
\begin{aligned}
	&\pt_i\pt_{p_i}\MG\bigg(\frac{\n\psi}{y},\psi\bigg)-y\mathcal W\bigg(\frac{\n\psi}{y},\psi\bigg)=0 &&\text{ in } \tilde\D,\\
	&\psi=Q, \ \ \frac1y\pt_{\nu}\psi=\Ld &&\text{ on } \Sigma,
\end{aligned}	
\right.
\end{equation*}
where $Q$ and $\Ld$ are constants. Assume that $\MG$ is uniformly elliptic, cf. \eqref{G dp}-\eqref{G dpp}, and that $\pt_{p_ip_jp_k}\mathcal{G}$, $\pt_{p_ip_jz} \mathcal{G}$, $\pt_{p_iz}\mathcal{G}$,
$\pt_{p_i}\mathcal{W}$, $\pt_z\mathcal{W}\in L^\infty$. Then $\psi=\tilde{\psi}$.
\end{lemma}

\begin{proof}
The proof relies on the unique continuation for the uniformly elliptic equations with Lipschitz coefficients. First, it follows from the boundary regularity for the elliptic equation 
that $\psi$, $\tilde\psi\in C^{1,\a}(\tilde\D\cup\Sigma)$. Let $\vp:=\tilde\psi-\psi$. Then $\vp$ solves
$$\begin{cases}\begin{split}
&\pt_i\left(\mathfrak{a}_{ij}\pt_j\vp\right)=-\pt_i(\mathfrak{b}_i\vp)+\mathfrak{\tilde b}_i\pt_i\vp+\mathfrak{c}\vp &&\text{ in } \tilde\D,\\
&\vp=\frac1y\pt_{\nu}\vp=0 &&\text{ on } \Sigma,
\end{split}\end{cases}$$
where
\begin{align*}
&\mathfrak{a}_{ij}:=\frac1y\int_0^1\pt_{p_ip_j}\MG\left(\frac{\n\tilde\psi_s}{y},\tilde\psi_s\right)ds,\q \mathfrak{b}_i:=\int_0^1\pt_{p_iz}\MG\left(\frac{\n\tilde\psi_s}{y},\tilde\psi_s\right)ds,\\
&\mathfrak{\tilde b}_i:=\int_0^1\pt_{p_i}\mathcal W\left(\frac{\n\tilde\psi_s}{y},\tilde\psi_s\right)ds,
\q \mathfrak{c}:=y\int_0^1\pt_z\mathcal W\left(\frac{\n\tilde\psi_s}{y},\tilde\psi_s\right)ds,
\end{align*}
with $\tilde\psi_s=\tilde\psi+s(\psi-\tilde\psi)$.
The Cauchy data allow to extend $\vp$ into a solution in a neighborhood $U\supset \Sigma$ by setting $\vp=0$ outside of $U\cap \tilde\D$. It follows from the assumptions on $\MG$ and $\mathcal W$ that $\mathfrak{b}_i$, $\mathfrak{\tilde b}_i$ and $\mathfrak{c}\in L^{\infty}$, and $\mathfrak a_{ij}$ is uniformly elliptic and $C^{1,\alpha}$. Thus by the unique continuation property, cf. for example \cite[Theorem 1.1]{KT2001_unique_continuation}, one has $\vp\equiv0$ in $U\cup \tilde\D$. This means that $\psi=\tilde\psi$. Since $\tilde{\D}$ is connected, using chain of balls argument one has $\psi=\tilde\psi$ in $\tilde{\D}$.
\end{proof}

Now we are ready to prove the free boundary is a graph.

\begin{proposition}\label{y graph}
Let $\G_{\mu,R,\Ld}$ be defined in \eqref{Gamma truncated}.
If $\G_{\mu,R,\Ld}\neq\emptyset$, then it can be represented as a graph of a function of $y$, i.e., there exists a function $\ur_{\mu,R,\Ld}$ and $\ul H_{\mu,R,\Ld}\in[H_*,1)$ such that
\begin{equation}\label{fb_express}
	\G_{\mu,R,\Lambda}=\{(x,y): x=\ur_{\mu,R,\Ld}(y),\, y\in(\ul H_{\mu,R,\Ld},1)\}.
\end{equation}
Furthermore, $\ur_{\mu,R,\Ld}$ is continuous, $-\mu<\ur_{\mu,R,\Ld}\leq R$ and $\lim_{y\to1-}\ur_{\mu,R,\Ld}(y)$ exists.
\end{proposition}

\begin{proof}
It follows from Proposition \ref{psi monotonic} that $\pt_{x}\psi\geq0$. Hence there is a function $\ur_{\mu,R,\Ld}(y)$ with values in $[-\mu,R]$ such that for any $y\in(0,1)$,
\begin{equation}\label{111}
	\psi(x,y)<Q \q \text{if and only if } x<\ur_{\mu,R,\Ld}(y).
\end{equation}
Since $\psi(-\mu,y)=0$ for $y\in[0, b_\mu']$, then $\ur_{\mu,R,\Ld}>-\mu$.

It follows from Lemma \ref{psi bound lem} that one has
$$\psi(x,y)\leq \psi^\sharp_{\mu,R}(R,y)<Q \q\text{in }\R\times [0,H_*).$$
This implies that there is no free boundary points in the strip $\R\times[0,H_*)$. Thus there exists a constant $\ul H_{\mu,R,\Ld}\geq H_*$ such that the free boundary is a graph $x=\ur_{\mu,R,\Ld}(y)$ for $y\in(\ul H_{\mu,R,\Ld},1)$.

Finally, one can show $\ur_{\mu,R,\Ld}$ is continuous. Clearly, by \eqref{111} the function $\ur_{\mu,R,\Ld}$ is lower semi-continuous.  Suppose there is a jump point $\bar y$ such that
$$\lim_{y\to\bar y}\ur_{\mu,R,\Ld}(y)=\ur_{\mu,R,\Ld}(\bar y)+\tau$$
for some $\tau>0$. Let $\bar x:=\ur_{\mu,R,\Ld}(\bar y)$. Since $\pt_x\psi\geq0$, the line segment $[\bar x,\bar x+\tau]\times\{\bar y\}$ belongs to the free boundary. Therefore, one can find an upper or lower rectangular neighborhood $U$ of $[\bar x,\bar x+\tau]\times\{\bar y\}$ such that $\psi$ solves following system
$$\begin{cases}\begin{split}
&\n\c\bigg(g_\v\bigg(\left|\frac{\n\psi}{y}\right|^2,\psi\bigg)\frac{\n\psi}{y}\bigg)-y\pt_z G_\v\bigg(\left|\frac{\n\psi}{y}\right|^2,\psi\bigg)=0 && \text{ in } U,\\
&\psi=Q, \q  \frac1y\pt_{\nu}\psi=\Ld && \text{ on } \pt U\cap\{y=\bar y\},
\end{split}\end{cases}$$
which admits a one-dimensional solution $\psi_{1d}(y)$. Then by the uniqueness of the Cauchy problem (cf. Lemma \ref{Cauchy uniqueness}), one has $\psi=\psi_{1d}$ in $\O_{\mu,R}$. This is however impossible. Thus $\ur_{\mu,R,\Ld}$ is continuous. The existence of $\lim_{y\to1-}\ur_{\mu,R,\Ld}(y)$ follows from Lemma \ref{nonoscillation}.
\end{proof}

\section{Continuous fit and smooth fit}\label{sec fit}
In this section, we first establish the continuous fit between the free boundary and the nozzle. More precisely, for given Bernoulli function $\bar B$, entropy function $\bar S$ and mass flux $Q$ at the upstream as in Proposition \ref{incoming data}, there exists a momentum $\Ld>0$ along the free boundary, such that the free boundary attaches the outlet of the nozzle. The continuous fit implies the smooth fit, that is, the free boundary joins the nozzle in a continuous differentiable fashion.

\subsection{Continuous fit}\label{subsec continuousfit}
In this subsection, $\mu$, $R$ and $\v$ are fixed. To emphasize the dependence on the parameter $\Ld$, we let $\psi_{\Ld}:=\psi_{\mu,R,\Ld}$ denote the minimizer of the truncated problem \eqref{variation problem} and let
$$\G_\Ld:=\G_{\mu,R,\Ld}=\left\{(x,y): x=\ur_\Ld(y) \text{ for } y\in(\ul H_{\Ld},1)\right\}$$
denote the the free boundary in \eqref{fb_express}, where $\ur_\Ld(y):=\ur_{\mu,R,\Ld}(y)$ and $\ul H_{\Ld}:=\ul H_{\mu,R,\Ld}.$

The following lemma shows that $\psi_{\Ld}$ and $\ur_\Ld$ depend on $\Ld$ continuously, whose proof is basically the same as that in \cite[Lemma 6.1]{LSTX2023} so we omit it here.

\begin{lemma}\label{Ld dependence}
If $\Ld_n\to\Ld$, then $\psi_{\Ld_n}\to \psi_{\Ld}$ uniformly in $\O_{\mu,R}$ and $\ur_{\Ld_n}(y)\to \ur_{\Ld}(y)$ for each $y\in (l_{\Ld},1]$.
\end{lemma}

The continuous dependence of the solution on the parameter $\Ld$ together with the nondegeneracy of the minimizer (Lemma \ref{nondegeneracy}) gives the next lemma.

\begin{lemma}\label{continuous fit}
There exists a constant $C=C(B_*,S_*,\bar H,\g,\v)$ such that the following statements hold.
\begin{itemize}
	\item[(i)] If $\Ld\geq CQ$, then the free boundary $\G_{\Ld}$ is nonempty and $\ur_{\Ld}(1)<0$.
	
	\item[(ii)] If $\Ld\leq C^{-1}Q$, then $\ur_{\Ld}(1)>0$.
\end{itemize}
\end{lemma}
\begin{proof}
(i) Suppose that $\G_{\Ld}$ is nonempty and $\ur_{\Ld}(1)\geq0$. We claim that $\Ld/Q$ cannot be too large. Let $r_*=r_*(B_*,S_*,\bar H,\g,\v)$ and $c_{1/2}=c_{1/2}(B_*,S_*,\bar H,\g,\v)$ be the constants determined in Lemma \ref{nondegeneracy} with $\vartheta=2$ and $a=1/2$. Since $\G_{\Ld}$ connects $(\ur_{\Ld}(1),1)$ to $(R,\ul H_{\Ld})$, there exist $X_*=(x_*,y_*)\in \G_\Ld$ (with $y_*>1/4$) and $r\in(0,r_*/4)$
independent of $\Ld$ and $Q$ such that $B_{ry_*}(X_*)\subset \O_{\mu,R}\cap\{y>1/4\}$ or $B_{ry_*}(X_*)\subset \{x>0,y>1/4\}$. If $\frac{c_{1/2}\ld}{\d\v}\leq r_*y_*$, then in view of the relation between $\ld$ and $\Ld$ in \eqref{ld_Ld} the proof is finished (note that $\d\v\leq CQ$ for some $C=C(B_*,S_*,\bar H,\g)$ by the definition of $\d$ in \eqref{delta}, the upper bound of $\k_0$ in \eqref{k0}, and the lower bound of $Q$ in Proposition \ref{incoming data}). If $\frac{c_{1/2}\ld}{\d\v}>r_*y_*$, then it follows from the nondegeneracy lemma (Lemma \ref{nondegeneracy}) and \eqref{ld_Ld} that
$$\frac Q{ry_*}\geq \frac1{ry_*}\left(\dashint_{B_{ry_*}(X_*)}|Q-\psi|^2\right)^{\frac12}\geq c_{1/2}\ld y_* \geq C \Ld y_*,$$
where $C>0$ depends on $B_*$, $S_*$, $\bar H$, $\g$, and $\v$. This leads to a contradiction if $\Ld/Q$ is large.

If the free boundary is empty, then one can get a contradiction by taking any $X_*=(x_*,1)$ with $x_*>0$ and using the linear nondegeneracy for nonnegative solutions along flat boundaries. This finishes the proof of (i).

(ii) Suppose that $\ur_{\Ld}\leq 0$. We show that $\Ld/Q$ cannot be too small. For $\tau>0$ and $\bar\sigma>0$ sufficiently small such that the comparison principle holds in $\R\times(0,2\bar \sigma)$ (cf. Lemma \ref{comparison principle}), denote
$$\eta_\tau(x):=\begin{cases}\begin{split}
&\tau \exp\left\{\frac{(x-\frac R2)^2}{(x-\frac R2)^2-\bar\sigma}\right\} &&  \text{if } \left|x-\frac R2\right|<\bar\sigma,\\
&0 && \text{otherwise}.
\end{split}\end{cases}$$
Choose $\bar\tau$ such that $y=\eta_{\bar\tau}(x)$ touches the free boundary at some point $\bar X=(\bar x,\bar y)$ with $0<\bar x<R$. Let $\D'=\O_{\mu,R}\cap\{(x,y): y<\eta_{\bar\tau}(x)\}$ and $\phi$ be a solution to
$$\begin{cases}\begin{split}
&\n\c\bigg(g_\v\bigg(\left|\frac{\n\phi}{y}\right|^2,\phi\bigg)\frac{\n\phi}{y}\bigg)-y\pt_zG_\v\bigg(\left|\frac{\n\phi}{y}\right|^2,\phi\bigg)=0 \q \text{in } \D',\\
&\phi=Q \q \text{on } \{y=\eta_{\tau_0}(x)\},\q \phi=0 \q \text{on  }\R\times \{0\}.
\end{split}\end{cases}$$
It follows from the comparison principle that $\phi\geq \psi_{\Ld}$ in $\D'$. Hence
$$\frac1{\bar y}\frac{\pt\phi}{\pt\nu}(\bar X)\leq \frac1{\bar y}\frac{\pt\psi_\Ld}{\pt\nu}(\bar X)=\Ld.$$
On the other hand, using the linear asymptotics of $\phi$ at $\bar X$, which can be obtained by solving the problem for the blow up limit of $\phi/Q$ at $\bar X$, one has
$$\frac1{\bar y}\frac{\pt\phi}{\pt\nu}(\bar X)\geq CQ$$
for some $C>0$ independent of $\Ld$. Then ${\Ld}/{Q}\geq C$. This finishes the proof.
\end{proof}

Lemma \ref{Ld dependence} and Lemma \ref{continuous fit} imply the following continuous fit property of the solution for the truncated problem.

\begin{corollary}\label{Ld determine}
There exists $\Ld=\Ld(\mu,R)>0$ such that $\ur_{\Ld}(1)=0$. Furthermore,
$$C^{-1}Q\leq \Ld(\mu,R)\leq CQ$$
for a positive constant $C$ depending on $B_*$, $S_*$, $\bar H$, $\g$, and $\v$, but independent of $\mu$ and $R$.
\end{corollary}

\subsection{Smooth fit}
With the continuous fit property at hand, one can argue along the same lines as \cite[Theorem 6.1]{ACF85} to conclude that the free boundary $\G_\psi$ fits the outlet of the nozzle in a $C^1$ fashion.

\begin{proposition}\label{smooth fit}
Let $\psi:=\psi_{\mu,R,\Ld}$ be the minimizer of the truncated problem \eqref{variation problem} with $\Ld$ as in  Corollary \ref{Ld determine}. Then $\N\cup\G_{\psi}$ is $C^1$ in a neighborhood of the point $A=(0,1)$, and $\n\psi$ is continuous in a $\{\psi<Q\}$-neighborhood of $A$.
\end{proposition}
\begin{proof}
Suppose that $\{X_n\}$ is a sequence of points in $\{\psi<Q\}$ which converges to $A$. Define the scaled functions
$$\psi_{n}(X):=Q-\frac{Q-\psi(X_n+rX)}{r},\q r\in(0,1).$$
Then $\psi_{n}$ satisfies the quasilinear equation
\begin{align*}
&\n\c \bigg(g_\v\bigg(\left|\frac{\n\psi_{n}}{y_n+ry}\right|^2,Q-r(Q-\psi_{n})\bigg)\frac{\n\psi_{n}}{y_n+ry}\bigg)\\
&\q\q-r(y_n+ry)\pt_zG_\v\bigg(\left|\frac{\n\psi_{n}}{y_n+ry}\right|^2,Q-r(Q-\psi_{n})\bigg)=0
\end{align*}
in $\{\psi_{n}<Q\}$. Since $\pt_zG_\v\in L^{\infty}$, then any blow up limit along a subsequence $\{\psi_{n}\}$ satisfies
\begin{equation}\label{smoothfit equ1}
\n\c(g_\v(|\n\psi|^2,Q)\n\psi)=0.
\end{equation}
Note that \eqref{smoothfit equ1} is the same as the governing equation for the irrotational flows. Thus with the scaling property at hand,
one can use the compactness arguments as in \cite[Theorem 6.1]{ACF85} together with the unique continuation property for the limiting equation \eqref{smoothfit equ1} to conclude Proposition \ref{smooth fit}.
\end{proof}

\section{Removal of the truncations and the far fields behavior}\label{sec remove truncation}
In this section, we first remove the domain and subsonic truncations, and then study the far fields behavior of the jet flows at the upstream and downstream. Consequently, we obtain the existence of subsonic jet flows which satisfy all properties in Problem \ref{problem 2}.

\subsection{Remove the domain truncations}
In this subsection, we remove the truncations of the domain by letting $\mu, R\to\infty$ to get a limiting solution in $\Omega$, which is bounded by $\N_0$ and $\N\cup([0,\infty)\times\{1\})$. The limiting solution inherits the properties of solutions in the truncated domains, which are summarized in the following proposition.

\begin{proposition}\label{psi determine}
Let the nozzle boundary $\N$ defined in \eqref{nozzle} satisfy \eqref{nozzle condition}. Given the Bernoulli function and the entropy function $(\bar{B}, \bar{S})\in (C^{1,1}([0,\bar{H}]))^2$ at the upstream which satisfy \eqref{BS min}-\eqref{BS condition2}, where the constant $\k=(B_*,S_*,\gamma,\bar H)$ in \eqref{BS condition2} is sufficiently small. Given the mass flux $Q\in(Q_*,Q^*)$, where $Q_*$ and $Q^*$ are defined in Proposition \ref{incoming data}.
	Let $\psi_{\mu,R,\Ld}$ be the minimizer of the problem \eqref{variation problem}
	with $\psi^\sharp_{\mu,R}$ defined in \eqref{psi0}. Then for any $\mu_j, R_j\to\infty$, there is a subsequence (still labeled by $\mu_j$ and $R_j$) such that $\Ld_j:=\Ld(\mu_j,R_j)\to \Ld_\infty$ for some $\Ld_\infty\in (0,\infty)$ and $\psi_{\mu_j,R_j,\Ld_j}\to \psi_\infty$ in $C^{0,\a}_{\rm loc}(\O)$ for any $\a\in (0,1)$. Furthermore, the following properties hold.
	\begin{itemize}
		\item [\rm (i)] The function $\psi:=\psi_\infty$ is a local minimizer for the energy functional, i.e., for any $D\Subset \Omega$, one has $J^\v(\psi)\leq J^\v({\phi})$ for any $\phi$ satisfying ${\phi}=\psi$ on $\partial D$, where
		$$J^\v(\phi):= \int_{D}y\bigg[G_\v\bigg(\left|\frac{\n\phi}{y}\right|^2,\phi\bigg)+\ld^2_{\v,\infty}\chi_{\{\phi<Q\}}\bigg]dX,\q \ld_{\v,\infty}:=\sqrt{\Phi_{\v}(\Ld_\infty^2,Q)}.$$
		In particular, $\psi$ solves
		\begin{equation}\label{eq_limiting_sol}
			\left\{
			\begin{aligned}
				&\n\c\bigg(g_\v\bigg(\left|\frac{\n\psi}{y}\right|^2,\psi\bigg)\frac{\n\psi}{y}\bigg)-y\pt_zG_\v\bigg(\left|\frac{\n\psi}{y}\right|^2,\psi\bigg)=0  &&\text{ in } \mathcal{O},\\
				&\psi =0 &&\text{ on } \N_0,\\
				&\psi =Q &&\text{ on } \N \cup \Gamma_{\psi},\\
				&\left|\frac{\nabla \psi}y\right| =\Lambda_\infty &&\text{ on } \Gamma_{\psi},
			\end{aligned}
			\right.
		\end{equation}
		where $\mathcal{O}:=\Omega\cap \{\psi<Q\}$ is the flow region, $\Gamma_{\psi}:=\pt\{\psi<Q\}\backslash \N$ is the free boundary, and  $\pt_zG_\v(|\n\psi/y|^2,\psi)$ satisfies \eqref{Gdz expression} in the subsonic region $|\nabla\psi/y|^2\leq \t_c(\psi)-\epsilon$.
		
		\item [\rm (ii)] The function $\psi$ is in $C^{2,\alpha}(\mathcal{O})\cap C^1(\overline{\mathcal{O}})$, and it satisfies $\pt_{x}\psi\geq 0$ in $\Omega$.
		
		\item [\rm (iii)] The free boundary $\G_{\psi}$ is given by the graph $x=\ur(y)$, where $\Upsilon$ is a $C^{2,\alpha}$ function as long as it is finite.
		
		\item[(iv)] At the orifice $A=(0,1)$ one has $\lim_{y\to1-}\ur(y)=0$. Furthermore, $\N\cup\Gamma_{\psi}$ is $C^1$ around $A$, in particular, $N'(1)=\lim_{y\rightarrow 1-}\Upsilon'(y)$.
		
		\item [\rm (v)] There is a constant $\ul H\in(0,1)$ such that $\ur(y)$ is finite if and only if $y\in (\ul H,1]$, and $\lim_{y\to \ul H+}\ur(y)=\infty$. Furthermore, there exists an $R_0$ sufficiently large, such that $\G_{\psi}\cap\{x>R_0\}=\{(x,f(x)): R_0<x<\infty\}$ for some $C^{2,\a}$ function $f$ and $\lim_{x\to+\infty} f'(x)=0$.
	\end{itemize}
\end{proposition}
The proof is similar to \cite[Proposition 6.5]{LSTX2023}, so we shall not repeat it here.

The next proposition shows the positivity of $\partial_{x}\psi_\infty$ in the region $\O\cap\{\psi<Q\}$, which plays an important role in the equivalence between the stream function formulation and the original Euler system.

\begin{proposition}\label{psix strictly pro}
	The solution $\psi:=\psi_\infty$ obtained in Proposition \ref{psi determine} satisfies $\pt_{x}\psi>0$ in $\MO:=\O\cap\{\psi<Q\}$.
\end{proposition}
\begin{proof}
	It follows from Proposition \ref{psi determine} (ii) that $\pt_{x}\psi\geq0$ in $\MO$. The strict inequality can be proved by the strong maximum principle. Indeed, assume that $\pt_{x}\psi(X_0)=0$ for some $X_0\in\MO$. Let $V:=\pt_{x}\psi$. Then $V\geq0$ in $\MO$ and it satisfies
	$$\pt_i\left(\frac1y\mathfrak{a_{ij}}\pt_j V+\mathfrak{b_i} V\right)-\mathfrak{b_i}\pt_iV-\mathfrak{c}V=0 \q\text{ in }\MO,$$
	where
	$$\mathfrak{a_{ij}}:=\pt_{p_ip_j}G_\v\left(\left|\frac{\n\psi}{y}\right|^2,\psi\right),\q \mathfrak{b_i}:=\pt_{p_iz}G_\v\left(\left|\frac{\n\psi}{y}\right|^2,\psi\right),\q \mathfrak{c}:=y\pt_{zz}G_\v\left(\left|\frac{\n\psi}{y}\right|^2,\psi\right).$$
	The strong maximum principle 
	gives that $V=0$ in $\MO$, and hence all streamlines are horizontal. This leads to a contradiction. Hence $V>0$ in $\MO$ and thus the proof of the proposition is complete.
\end{proof}

\subsection{Remove the subsonic truncation}
In this subsection, we remove the subsonic truncation when the mass flux is
sufficiently small. This means that for given suitable mass flux, the solution obtained in Proposition \ref{psi determine} is actually a solution of the problem \eqref{jet_equ}.

\begin{proposition}\label{subsonic pro}
	Let $\psi:=\psi_\infty$ be a solution obtained in Proposition \ref{psi determine} with $\Ld:=\Ld_\infty$. Assume that the constant $\k=\k(B_*,S_*,\gamma,\N)$ in \eqref{BS condition2} is sufficiently small.
	Then
	\begin{equation}\label{subsonic}
		\left|\frac{\n\psi}{y}\right|^2\leq\t_c(\psi)-\v
	\end{equation}
	as long as $Q\in(Q_*,Q_m)$ for some $Q_m=Q_m(B_*,S_*,\gamma,\N)\in(Q_*,Q^*)$ sufficiently small, where $Q_*$ and $Q^*$ are defined in Proposition \ref{incoming data}.
\end{proposition}
\begin{remark}\label{rmk_subsonic}
		In view of \eqref{tc_bound}, $\t_c(\psi)$ is uniformly bounded from below by a positive constant depending only on $B_*$, $S_*$, $\bar H$, and $\g$. Thus one can fix a truncation parameter $\v>0$ small depending only on $B_*$, $S_*$, $\bar H$, and $\g$. 
		For fixed $\v>0$, let $\psi$ be a solution of the elliptic equation in \eqref{eq_limiting_sol} which satisfies the estimate \eqref{subsonic}. By the definition of $g_\v$ in \eqref{gm} and the expression of $\partial_zG_\epsilon$ in the subsonic region, cf. \eqref{Gdz expression}, the equation in \eqref{eq_limiting_sol} is exactly the equation in \eqref{jet_equ}.
\end{remark}
\noindent\textit{Proof of Proposition \ref{subsonic pro}.}
For fixed $\v>0$ sufficiently small, we claim that
\begin{equation}\label{psi_gradient_Q}
	\left|\frac{\n\psi}{y}\right|\leq CQ
	\q\text{in }\MO:=\O\cap\{\psi<Q\}
\end{equation}
for some constant $C=C(B_*,S_*,\gamma,\N)>0$. There are two cases:

\textit{Case 1. Estimate of $|\n\psi/y|$ away from the symmetry axis $\N_0$.}
Let $X_0=(x_0,y_0)\in\MO\cap\{y_0>1/4\}$. By elliptic estimates there exists a constant $C=C(B_*,S_*,\gamma,\N)>0$ such that
$$\left|\frac{\n\psi(X_0)}{y_0}\right|\leq C(\|\psi\|_{L^\infty(\MO)}+\|\pt_zG_\v\|_{L^\infty(\MO)}).$$
Then $|\n\psi(X_0)/y_0|\leq CQ$ follows from $0\leq\psi\leq Q$, the bound of $|\pt_zG_\epsilon|$ in \eqref{G dzdzpdzz}, and \eqref{label_2}.

\textit{Case 2. Estimate of $|\n\psi/y|$ near the symmetry axis $\N_0$.}
Let $X_0=(x_0,y_0)\in \MO\cap\{0<y_0\leq1/4\}$. For any $X=(x,y)\in B_1$, define
\begin{equation*}\label{def:psi_0} \psi_0(X)=\frac{\psi(X_0+rX)}{r^2}, \q  r=\frac{y_0}2.
\end{equation*}
The function $\psi_0$ satisfies
$$\n\c\left(g_\v\left(\left|\frac{\n\psi_0}{y+2}\right|^2,\frac{y_0^2}{4}\psi_0\right)\frac{\n\psi_0}{y+2}\right)-\frac{y_0^2}4(y+2)\pt_z G_\v\left(\left|\frac{\n\psi_0}{y+2}\right|^2,\frac{y_0^2}{4}\psi_0\right)=0 \q\text{in }B_1.$$
Besides, since by Lemma \ref{psi bound lem}
\begin{align}\label{psi_y2}
	0<\psi(x,y)\leq\min\{\Ld y^2e^{1-y},Q\}
\end{align}
and by Corollary \ref{Ld determine} it holds $\Ld\leq CQ$ , one has
$$0<\psi_0(X)\leq 4\Ld e\left(1+\frac{y}{2}\right)^2\leq CQ \q\text{in }  B_1$$
where $C=C(B_*,S_*,\bar H,\g)>0$. Thus using elliptic estimates, \eqref{G dzdzpdzz}, and \eqref{label_2} again yields
$$\left|\frac{\n\psi(X_0)}{y_0}\right|=\frac12|\n\psi_0(0)|\leq C(\|\psi_0\|_{L^\infty(B_1)}+\|\pt_zG_\v\|_{L^\infty(\MO)})\leq CQ$$
for some constant $C=C(B_*,S_*,\gamma,\N)>0$. Thus \eqref{psi_gradient_Q} holds true. 

Note that $|\n\psi/y|$ is uniformly H\"{o}lder continuous up to the symmetry axis $\N_0$ (cf. Appendix \ref{appendix}). Besides, by Corollary \ref{Ld determine} one has $|\n\psi/y|=\Ld\leq CQ$ on the free boundary $\Omega\cap\partial\{\psi<Q\}$. Thus in view of Remark \ref{rmk_subsonic}, one can find a small constant $Q_m=Q_m(B_*,S_*,\gamma,\N)\in(Q_*, Q^*)$, where $Q_*$ and $Q^*$ are defined in Proposition \ref{incoming data}, such that the estimate \eqref{subsonic} holds true as long as the mass flux $Q\in(Q_*,Q_m)$. 
This finishes the proof of Proposition \ref{subsonic pro}.
\bx

\subsection{Far fields behavior of the jet flow}

The next proposition gives the asymptotic behavior of the subsonic jet flows as $x\to\pm\infty$.

\begin{proposition}\label{pro_asymptotic_upstream}
	Let $\psi:=\psi_\infty$ be the solution obtained in Proposition \ref{psi determine} with $\Ld:=\Ld_\infty$. Then for any $\alpha\in(0,1)$, as $x\to-\infty$,
	\begin{equation}\label{psi upstream}
		\psi(x,y)\to\bar{\psi}(y):=\int_0^{y}s(\bar{\r}\bar{u})(s)ds
		\q \text{in } C_{{\rm loc}}^{2,\alpha}(\R\times(0,\bar H)),
	\end{equation}
	where $\bar\rho$ and $\bar u$ are determined in Proposition \ref{incoming data}; as $x\to+\infty$,
	\begin{equation}\label{psi downstream}
	\psi(x,y)\to\ul\psi(y):=\int_0^ys(\ul\r\ul u)(s)ds 	\q \text{in } C_{{\rm loc}}^{2,\alpha}(\R\times(0,\ubar H)),
	\end{equation}
	where $\ul{\r},\,\ul{u}\in C^{1,\alpha}((0,\ul{H}])$ are positive functions. Moreover, the density $\ul\r$, axial velocity $\ul u$, and  height $\ul H$ of the flow at the downstream are uniquely determined by $\bar B$, $\bar S$, $\bar H$, $\g$, $Q$, and $\Ld$.
\end{proposition}
\begin{proof}
	(i) \emph{Upstream asymptotic behavior.} Let $\psi^{(-n)}(x,y): =\psi(x-n,y)$, $n\in\mathbb Z$. Since the nozzle is asymptotically horizontal with the height $\bar H$, there exists a subsequence $\psi^{(-n)}$ (relabeled) converges to a function $\hat\psi$ in $C_{\rm loc}^{2,\alpha}(\R\times(0,\bar H))$ for any $\alpha\in(0,1)$, where $\hat\psi$ satisfies \eqref{subsonic} and solves the Dirichlet problem in the infinite strip
	\begin{equation}\label{system upstream}\begin{cases}
			\begin{split}
				&\n\c\bigg(g_\v\bigg(\bigg|\frac{\n\hat\psi}{y}\bigg|^2,\hat\psi\bigg)\frac{\n\hat\psi}{y}\bigg)-y\pt_z G_\v\bigg(\bigg|\frac{\n\hat\psi}{y}\bigg|^2,\hat\psi\bigg)=0 \q \text{ in } \R\times(0,\bar H),\\
				&\hat\psi=Q \ \text{ on } \R\times\{\bar H\},\q \hat\psi=0 \ \text{ on } \R\times\{0\}
	\end{split}\end{cases}\end{equation}
	and it satisfies $0\leq\hat\psi\leq Q$ in $\R\times[0,\bar H]$.
	On the other hand, it follows from the energy estimates (cf. \cite[Proposition 4.1]{DD2011_axially_small_vorticity}) that the problem \eqref{system upstream} has a unique solution.
	Since $\hat\psi$ satisfies \eqref{subsonic} and the function $\bar\psi$ defined in \eqref{psi upstream} satisfies the equation in \eqref{jet_equ}, by Remark \ref{rmk_subsonic} one has $\hat\psi(x,y)=\bar \psi(y)$ in $\R\times[0,\bar H]$.
	This proves the asymptotic behavior of the flows in the upstream.

	(ii) \emph{Downstream asymptotic behavior.} Let $\psi^{(n)}(x,y):=\psi(x+n,y)$, $n\in\mathbb Z$. By the $C^{2,\a}$ regularity of the free boundary and the boundary regularity for elliptic equations, there exists a subsequence $\psi^{(n)}$ (relabeled) converging to a function $\ubar\psi$ in $C_{\rm loc}^{2,\alpha}(\R\times(0,\ubar H))$ for any $\alpha\in(0,1)$, where $\ubar\psi$ satisfies  \eqref{subsonic} and solves
	\begin{equation}\label{system downstream}\begin{cases}\begin{split}
				&\n\c\bigg(g_\v\bigg(\bigg|\frac{\n\ubar\psi}{y}\bigg|^2,\ubar\psi\bigg)\frac{\n\ubar\psi}{y}\bigg)-y\pt_z G_\v\bigg(\bigg|\frac{\n\ubar\psi}{y}\bigg|^2,\ubar\psi\bigg)=0 &&\text { in } \R\times(0,\ul H),\\
				&\ubar\psi=Q,\q \frac1y\pt_{y}\ubar\psi=\Ld &&\text{ on } \R\times\{\ul H\},\\
				&\ubar\psi=0  && \text{ on }\R\times\{0\}
	\end{split}\end{cases}\end{equation}
	with $0\leq\ubar\psi\leq Q$ in $\R\times[0,\ul H]$.
	By Lemma \ref{Cauchy uniqueness}, the solution to the above Cauchy problem is unique for given positive constants $Q$, $\Ld$, and $\ul H$.
	Moreover, using the energy estimates (\cite[Proposition 4.1]{DD2011_axially_small_vorticity}) one can conclude that $\pt_{x}\ubar\psi=0$, that is, $\ubar\psi$ is a one-dimensional function. 
	Therefore, by \eqref{psi gradient} the flow radial velocity $v\to0$ as $x\to+\infty$. In view of \eqref{eq:rup_psi}, the downstream density $\ubar\rho$, axial velocity $\ubar u$, and pressure $\ubar p$ (which is a constant, cf. Remark \ref{rmk_1}) are determined by
	$$(\ubar\r,\ubar u,\ubar p)=\left(\frac1{g(|\ubar\psi'/y|^2,\ubar\psi)}, \frac{g(|\ubar\psi'/y|^2,\ubar\psi)\ubar\psi'}{y}, \frac{(\gamma-1)\S(Q)}{\gamma g^\gamma(\Lambda^2,Q)}\right).$$
	Consequently, $\ubar\psi$ can be expressed as in \eqref{psi downstream}. The regularity of $\ubar\rho$ and $\ubar u$ follows from the regularity of $\ubar \psi$.
	
	It remains to show that the downstream axial velocity $\ubar u$ is positive and the downstream height $\ubar H$ is uniquely determined by $\bar B$, $\bar S$, $\bar H$, $\g$, $Q$, and $\Ld$.
	For this we let $\th(y)$ be the position at downstream if one follows along the streamline with asymptotic height $y$ at the upstream, i.e., $\th:[0,\bar H]\to[0,\ul H]$ satisfies
	\begin{equation}\label{theta def}
		\ul\psi(\th(y))=\bar\psi(y),\q y\in[0,\bar H],
	\end{equation}
	where $\bar\psi$ and $\ul\psi$ are defined in \eqref{psi upstream} and \eqref{psi downstream}, respectively.
	Then the map $\th$ satisfies
	\begin{align}\label{theta}
		\begin{cases}
			\th'(y)=\frac{y(\bar{\r}\bar{u})(y)}{\th(y)(\ul{\r}\ul{u})(\th(y))},\\
			\th(0)=0,
	\end{cases}\end{align}
	where by \eqref{S def}
	\begin{equation}\label{ul r}
		\ul\r(\th(y))=\left(\frac{\g\ul p}{(\g-1)\bar S(y)}\right)^{\frac1\g},
	\end{equation}
	and by \eqref{BS_relation1} and \eqref{BS_psi_bar}
	\begin{equation}\label{ul u}
		\ul u(\th(y))
		=\sqrt{2\bigg(\bar B(y)-\left(\frac{\g\ul p}{\g-1}\right)^{1-\frac1\g}\bar S^{\frac1\g}(y)\bigg)}.
	\end{equation}
	Using same arguments as the proof for \eqref{ubar ineq1} in Proposition \ref{incoming data}, one gets
	$$C(B_*,S_*,\bar H,\g)\k^{\frac1{4\g}}\leq\ul u(\th(y))\leq \sqrt{2B^*}.$$
	Thus substituting \eqref{ul r} and \eqref{ul u} into \eqref{theta} gives that $\th(y)$, in particular $\ul H=\theta(\bar H)$, is uniquely determined. 
	This finishes the proof of the proposition.
\end{proof}

In view of Propositions \ref{psi determine}-\ref{pro_asymptotic_upstream}, we have proved the existence of solutions to {Problem} \ref{problem 2} when the mass flux $Q$ is in a suitable range. Then the existence of solutions to Problem \ref{probelm 1} follows from Proposition \ref{prop:equiv_sol}.


\section{Uniqueness of the outer pressure}\label{sec uniqueness}
In this section, we show that for given Bernoulli function $\bar B$, entropy function $\bar S$, and mass flux $Q$ at the upstream, there is a unique momentum $\Ld$ on the free boundary such that Problem \ref{problem 2} has a solution. In view of \eqref{Ld} and the proof of Lemma \ref{g}(i), the uniqueness of $\Ld$ implies the uniqueness of the outer pressure $\ul p$. More precisely, the constant $\r_0$ in \eqref{Ld} is determined by $\r_0=1/g(\Ld^2,Q)$, where $g$ is the function defined in Lemma \ref{g}(i), hence $\ul p=(\g-1)\bar S(\bar H)\r_0^\g/\g$ is uniquely determined.

\begin{proposition}\label{uniqueness pro}
Let the nozzle boundary $\N$ defined in \eqref{nozzle} satisfy \eqref{nozzle condition}. Given the Bernoulli function and the entropy function $(\bar{B}, \bar{S})\in (C^{1,1}([0,\bar{H}]))^2$ at the upstream which satisfy \eqref{BS min}-\eqref{BS condition2}, where the constant $\k=(B_*,S_*,\gamma,\N)$ in \eqref{BS condition2} is sufficiently small. Given the mass flux $Q\in(Q_*,Q^*)$, where $Q_*$ and $Q^*$ are defined in Proposition \ref{incoming data}. Suppose that $(\psi_i,\G_i,\Ld_i)$ $(i=1,2)$ are two uniformly subsonic solutions to Problem \ref{problem 2}, then $\Ld_1=\Ld_2$.
\end{proposition}
\begin{proof}
Let $(\ul\r_i,\ul u_i,\ul p_i, \ul H_i)$ ($i=1,2$) be the associated downstream asymptotic states.
Without loss of generality, assume that $\Ld_1<\Ld_2$. 
By \eqref{BSpsi relation}, along the free boundary $\G_1$ and $\G_2$ one has
$$\frac{\Ld_1^2}{2\ul\r_1^2(\ul H_1)}+\ul\r_1^{\g-1}(\ul H_1)\bar S(\bar H)=\bar B(\bar H)
=\frac{\Ld_2^2}{2\ul\r_2^2(\ul H_2)}+\ul\r_2^{\g-1}(\ul H_2)\bar S(\bar H).$$
Since the solutions are subsonic and $\Ld_1<\Ld_2$, it follows from the proof of Lemma \ref{g}(i) that $\ul\r_1(\ul H_1)>\ul\r_2(\ul H_2)$.
Hence by \eqref{S def} one has
$$\ul p_1=\frac{(\g-1)}{\g}\bar S(\bar H)\ul \r_1^\g(\ul H_1)>\frac{(\g-1)}{\g}\bar S(\bar H)\ul \r_2^\g(\ul H_2)=\ul p_2.$$
Let $\th_i(y)$ $(i=1,2)$ be the function defined in \eqref{theta} associated with $(\ul\r_i,\ul u_i)$ ($i=1,2$).
Then
$$\ul\r_1(\th_1(y))=\left(\frac{\g\ul p_1}{(\g-1)\bar S(y)}\right)^{\frac1\g}>\left(\frac{\g\ul p_2}{(\g-1)\bar S(y)}\right)^{\frac1\g}=\ul\r_2(\th_2(y)).$$
Combining this inequality with
$$\frac{(\ul\r_1\ul u_1)^2(\th_1(y))}{2\ul\r_1^2(\th_1(y))}+\ul\r_1^{\g-1}(\th_1(y))\bar S(y)=\bar B(y)
=\frac{(\ul\r_2\ul u_2)^2(\th_2(y))}{2\ul\r_2^2(\th_2(y))}+\ul\r_2^{\g-1}(\th_2(y))\bar S(y)$$
(which is obtained from \eqref{psi gradient} and \eqref{BSpsi relation}), and the proof of Lemma \ref{g}(i) together yields
$$(\ul\r_1\ul u_1)(\th_1(y))<(\ul\r_2\ul u_2)(\th_2(y)).$$
Thus in view of \eqref{theta} it follows
\begin{equation}\label{theta ineq1}
(\th_1^2(y))'>(\th_2^2(y))'.
\end{equation}
Note that $\th_1(0)=\th_2(0)$. After an integration of  \eqref{theta ineq1} on $[0,\bar H]$ one gets
$$\ul H_1=\th_1(\bar H)>\th_2(\bar H)=\ul H_2.$$

Let $\MO_i$ be the domain bounded by $\N_0$, $\N$, and $\G_i$. Since $\ul H_1>\ul H_2$ and that $\N\cup \G_i$ is a $y$-graph, one can find a $\tau\geq0$ such that the domain $\MO_1^{\tau}=\{(x,y):(x-\tau,y)\in\MO_1\}$ contains $\MO_2$. Let $\tau_*$ be the smallest number such that $\MO^{\tau_*}_1$ contains $\MO_2$ and they touch at some point $(x_*,y_*)\in\Gamma_2$. Define  $\psi_1^{\tau_*}(x,y):=\psi_1(x-\tau_*,y)$. 
Now we prove $\psi_1^{\tau_*}\leq \psi_2$  in $\mathcal O_2$. Suppose that there exists a point $(\bar x,\bar y)\in\mathcal O_2$ such that $\psi_1^{\tau_*}(\bar x,\bar y)>\psi_2(\bar x,\bar y)$. Since $\psi_1^{\tau_*}$ and $\psi_2$ have the same asymptotic behavior as $x\to-\infty$, and $\psi_1^{\tau_*}\leq\psi_2$ as $x\to+\infty$ by the comparison principle (cf. Lemma \ref{comparison principle}), one can find a domain $\mathcal O_2'\subset\mathcal O_2$ containing $(\bar x,\bar y)$ such that $\psi_1^{\tau_*}-\psi_2$ obtains its maximum in the interior of $\mathcal O_2'$. This contradicts the strong maximum principle. Hence $\psi_1^{\tau_*}\leq\psi_2$ in $\MO_2$.
Since at the touching point one has $\psi_1^{\tau_*}(x_*,y_*)=\psi_2(x_*,y_*)$, by the Hopf lemma it holds
$$\Ld_1=\frac1{y_*}\frac{\pt\psi_1^{\tau_*}}{\pt\nu}(x_*,y_*)>\frac1{y_*}\frac{\pt\psi_2}{\pt\nu}(x_*,y_*)=\Ld_2.$$
This leads to a contradiction. Hence one has $\Ld_1=\Ld_2$.
\end{proof}
Combining all the results in previous sections, Theorem \ref{result} is proved.

\appendix
\section{ }\label{appendix}

In this appendix, we exclude the singularity of the solution $\psi$ obtained in Proposition \ref{psi determine} near the symmetry axis. We refer to \cite{LW1998} for the study on elliptic equations with singularity at the boundary under different conditions.

The function $\psi$ obtained in Proposition \ref{psi determine} is a solution of the elliptic equation \eqref{elliptic equG} in $\MO:=\Omega\cap\{\psi<Q\}$,
where $\MG(\bp,z):=G_\v(|\bp|^2,z)$ satisfies all properties in Proposition \ref{Gproperties pro}. Besides, $\psi$ satisfies the estimates \eqref{psi_gradient_Q} and $0<\psi(x,y)\leq Cy^2$ (cf. \eqref{psi_y2}),  provided $\k$ in \eqref{BS condition2} is sufficiently small. 
We aim to show that $|\n\psi/y|$ is uniformly H\"{o}lder continuous up to $\{y=0\}$. 

For simplicity, we denote 
\begin{align}\label{def_L}
\mathcal L:=\mathfrak a_{ij}\pt_{ij}-\frac{\mathfrak a_{i2}}{y}\pt_i \q (i,j=1,2)
\end{align}
as a linear elliptic operator, where the coefficients $\mathfrak a_{ij}\in C^1(\R\times(0,1])$  $(i,j=1,2)$ satisfy 
\begin{align*}
&\mathfrak b_*|\mathbf{\xi}|^2\leq \mathfrak a_{ij}\mathbf{\xi}_i\mathbf{\xi}_j\leq \mathfrak b^*|\mathbf{\xi}|^2, \quad \mathbf{\xi}\in \R^2
\end{align*}
with $\mathfrak b_*,\mathfrak b^*>0$; moreover, we denote
\begin{align*}
\mathcal Q_{k}:=\left[-\frac{M}{2^k},\frac{M}{2^k}\right]\times\left[0,\frac 1{2^k}\right], \q k\geq0,
\end{align*}
where $M>0$ is a constant to be chosen later. 
 
First, we give the following key lemma.

\begin{lemma}\label{lem_axi}
Let $\phi\in C^{2,\alpha}_{{\rm loc}}(\R\times(0,1))\cap C^0(\R\times[0,1])$ be a solution of
\begin{align}\label{pde}
	\mathcal L\phi=y^2{\mathfrak f}(x,y) \q\text{in }\R\times (0,1),
\end{align}
where $\mathcal L$ is defined in \eqref{def_L} and ${\mathfrak f}\in L^{\infty}(\R\times(0,1))$. Suppose that $\|{\mathfrak f}\|_{L^{\infty}(\R\times(0,1))}\leq F$ for some constant $F>0$ sufficiently small depending on $\mathfrak b_*$ and $\mathfrak b^*$. If $0\leq \phi(x,y)\leq y^2$ in $\mathcal Q_0$, then there exists a constant $\sigma=\sigma(\mathfrak b_*,\mathfrak b^*)\in(0,1/8)$, such that
$$0\leq \phi(x,y)\leq (1-\sigma) y^2 \q\text{or}\q \sigma y^2\leq \phi(x,y)\leq y^2 \q \text{ in } \mathcal Q_1.$$
\end{lemma}
\begin{proof}
	Since $0\leq \phi(0,1/2)\leq 1/4$, one has $\phi(0,1/2)\geq 1/8$ or $\phi(0,1/2)\leq 1/8$. Assume that $\phi(0,1/2)\geq 1/8$ (otherwise we consider the function $y^2-\phi$). Using the Harnack inequality yields $\phi(x,y)\geq 2\sigma$ in $[-M+1,M-1]\times [1/{8},{7}/{8}]$, where $\sigma=\sigma(\mathfrak b_*,\mathfrak b^*)>0$ can be sufficiently small. Define a function
	$$P(x,y):=\sigma y^2+\sigma y^3-\frac{8\sigma}{M^2}|x|^2y^2.$$
	Direct computations give that
	\begin{align*}
		\mathcal L(\phi-P)=y^2 {\mathfrak f}(x,y)+\frac{16\sigma}{M^2}\mathfrak a_{11}y^2+\frac{48\sigma}{M^2}\mathfrak a_{12}xy-3\sigma\mathfrak a_{22}y.	
	\end{align*}
 Choose a sufficiently large $M=M(\mathfrak a_{ij})>0$. Since $\|{\mathfrak f}\|_{L^{\infty}(\R\times(0,1))}$ is sufficiently small, one can check that $\mathcal L(\phi-P)\leq0$ in $\mathcal Q_1$ and $\phi\geq P$ on $\partial \mathcal Q_{1}$. By the comparison principle we conclude $\phi\geq P$ in $\mathcal Q_1$. Thus $\phi(0,y)\geq P(0,y)\geq\sigma y^2$. Using a translation one gets $\phi(x,y)\geq \sigma y^2$ in $\mathcal Q_1$. This finishes the proof of the lemma.
\end{proof}

Without loss of generality we assume that $0\leq \psi\leq y^2$. The elliptic equation in \eqref{elliptic equG} can be rewritten as
\begin{align}\label{psi_eq_appendix}
	\tilde{\mathfrak a}_{ij}\pt_{ij}\psi-\frac{\tilde{\mathfrak a}_{i2}}y\pt_i\psi=y^2\tilde{\mathfrak f},
\end{align}
where
\begin{equation}\label{def:cof}\begin{split}
&\tilde{\mathfrak a}_{ij}:=\pt_{p_ip_j}\MG\bigg(\frac{\n\psi}{y},\psi\bigg),\q
\tilde{\mathfrak f}:=\pt_{p_iz}\MG\bigg(\frac{\n\psi}{y},\psi\bigg)\frac{\pt_i\psi}{y}+\pt_{z}\MG\bigg(\frac{\n\psi}{y},\psi\bigg).	
\end{split}\end{equation}
Note that $\tilde{\mathfrak f}\in {L^\infty(\R\times (0,1))}$, cf. \eqref{G dzdzpdzz} and \eqref{psi_gradient_Q}. Let
\begin{equation}\label{defpsi1}
\psi_1(x,y):=\frac{\psi(r_0x,r_0y)}{r_0^2}, \q r_0>0.
\end{equation} 
Then the function $\psi_1$ satisfies an elliptic equation of the form 
\begin{equation*}	
\mathcal L\psi_1=y^2\mathfrak f_1(x,y) \q\text{in }\mathcal Q_0, \q\text{where } \mathfrak f_1(x,y):=r_0^2\tilde{\mathfrak f}(r_0x,r_0y),
\end{equation*} 
and $0\leq \psi_1\leq y^2$ in $\mathcal Q_0$. In view of Lemma \ref{lem_axi}, if we choose a  sufficiently small $r_0=r_0(\mathfrak b^*,\mathfrak b_*,\|\tilde{\mathfrak f}\|_{L^\infty(\R\times (0,1))})$ such that $\|\mathfrak  f_1\|_{L^\infty(\R\times (0,1))}\leq F$, then
\begin{equation}\label{eq2}
m_1y^2\leq \psi_1(x,y)\leq M_1y^2 \q \text{ in } \mathcal Q_{1},
\end{equation}
where $m_1\in\{0,\sigma\}$, $M_1\in\{1-\sigma,1\}$, and $M_1-m_1=1-\sigma$. 
Define 
$$\psi_2(x,y):=\frac{\psi_1-m_1y^2}{M_1-m_1}\left(\frac12x,\frac12y\right).$$
Since $\psi_2$ also satisfies an elliptic equation of the form
\begin{align}\label{1}
	\mathcal L\psi_2=y^2\mathfrak f_2(x,y) \q\text{in } \mathcal Q_0, \q\text{where }
	\mathfrak f_2(x,y):=\frac{\mathfrak f_1\left(\frac x2,\frac y2\right)}{2^4(M_1-m_1)}
\end{align}
and $0\leq \psi_2(x,y)\leq y^2/4$ in $\mathcal Q_{0}$. 
Applying Lemma \ref{lem_axi} again to $4\psi_2$ (note that $2^2(M_1-m_1)\geq 1$ since $\sigma\leq 1/8$) one has 
$$m_2y^2\leq \psi_2(x,y)\leq M_2y^2 \q\text{in } \mathcal Q_1,$$
where $m_2\in\{0,\sigma/4\}$, $M_2\in\{(1-\sigma)/4,1/4\}$, and $M_2- m_2=(1-\sigma)/4$. Repeating the above arguments, for each integer $k>0$, we have 
$$\tilde m_ky^2\leq \psi_1(x,y)\leq  \tilde M_ky^2 \q \text{in } \mathcal Q_{k},$$
where $\tilde m_k,\tilde M_k\in[0,1]$ and $\tilde M_k-\tilde m_k=(1-\sigma)^{k}$. In view of the definition of $\psi_1$ in \eqref{defpsi1}, we get
$$|\psi(x,y)-\mathfrak c(0)y^2|\leq Cy^{2+\alpha} \q\text{in } B_\sigma^+:=B_\sigma\cap\{y>0\}$$
for some $C^{\alpha}$ ($\alpha\in(0,1)$) function $\mathfrak c(x)$ and constants $C>0$, $0<\sigma\ll 1$. 

Now for any $(x_0,y_0)\in B_{\sigma/2}^+$, let $r=y_0/2$. Since \eqref{psi_eq_appendix} holds in $B_{r}(x_0,y_0)$, one has 
\begin{equation*}
	|\psi(x_0+x,y_0+y)-\mathfrak c(0)(y_0+y)^2|\leq C(y_0+y)^{2+\alpha}\leq Cr^{2+\alpha} \q\text{in }  B_r.
\end{equation*}
Let 
\begin{equation*}
\varphi(x,y):=\frac{\psi(x_0+rx,y_0+ry)-\mathfrak c(0)(y_0+ry)^2}{r^2}.
\end{equation*}
Then $\varphi$ satisfies 
\begin{equation*}
\hat{\mathfrak {a}}_{ij}\pt_{ij}\varphi-\frac{\hat{\mathfrak a}_{i2}}{2+y}\pt_{i}\varphi=\hat{\mathfrak f} \q\text{in } B_1,
\end{equation*}
where 
\begin{equation*}\begin{split}
\hat{\mathfrak {a}}_{ij}(x,y):=	\tilde{\mathfrak {a}}_{ij}(x_0+rx,y_0+ry),\q 
\hat{\mathfrak f}(x,y):=(y_0+ry)^2\tilde {\mathfrak f}(x_0+rx,y_0+ry),
\end{split}
\end{equation*}
and $\tilde{\mathfrak {a}}_{ij}$, $\tilde {\mathfrak f}$ are defined in \eqref{def:cof}. 
Thus by the interior regularity of elliptic equations one has 
$$\left|\frac{\n\psi(x_0,y_0)}{y_0}-(0,2\mathfrak c(0))\right|=\frac12|\n\varphi(0)|\leq C(\|\varphi\|_{L^\infty(B_1)}+\|\hat{\mathfrak f}\|_{L^{\infty}(B_1)})\leq Cy_0^{\alpha},$$
where $C=C(\hat {\mathfrak a}_{ij})>0$. Using a coordinate translation we  conclude that $|\n\psi/y|$ is uniformly H\"{o}lder continuous up to $\{y=0\}$. Actually we can further get $\psi\in C^{2,\alpha}(\MO\cup\{y=0\})$ since now we know $\tilde {\mathfrak f}$ is $C^{\alpha}$. 

\medskip {\bf Acknowledgement.} 
The author would like to thank Professors Lihe Wang and Chunjing Xie for many helpful discussions.

\bibliographystyle{abbrv}
\bibliography{reference}

\end{document}